\documentclass[11pt,letterpaper]{amsart}
\usepackage[utf8]{inputenc}
\usepackage[margin=1in]{geometry}
\usepackage{amsmath}
\usepackage{amssymb}
\usepackage{amsthm}
\usepackage{amscd}
\usepackage{enumerate}
\usepackage{xcolor}
\usepackage{subcaption}

\usepackage{graphicx}
\usepackage{amsmath,amsthm,amsfonts,amscd,amssymb,comment,eucal,latexsym,mathrsfs}
\usepackage{stmaryrd}
\usepackage[all]{xy}

\usepackage{epsfig}

\usepackage[all]{xy}
\xyoption{poly}
\usepackage{fancyhdr}
\usepackage{wrapfig}
\usepackage{epsfig}

\theoremstyle{plain}
\newtheorem{thm}{Theorem}[section]
\newtheorem{prop}[thm]{Proposition}
\newtheorem{lem}[thm]{Lemma}
\newtheorem{cor}[thm]{Corollary}

\theoremstyle{definition}
\newtheorem{defn}[thm]{Definition}
\theoremstyle{remark}

\newtheorem{example}{Example}

\newtheorem{problem}{Problem}
\newtheorem{question}{Question}

  \def\C{{\mathbb{C}}}  \def\E{{\mathbb{E}}}         \def\N{{\mathbb{N}}}  \def\P{{\mathbb{P}}} \def\Q{{\mathbb{Q}}} \def\R{{\mathbb{R}}}        

    \def\cE{{\mathcal{E}}} \def\cF{{\mathcal{F}}} \def\cG{{\mathcal{G}}} \def\cH{{\mathcal{H}}}    \def\cL{{\mathcal{L}}} \def\cM{{\mathcal{M}}} \def\cN{{\mathcal{N}}}    \def\cR{{\mathcal{R}}} \def\cS{{\mathcal{S}}} \def\cT{{\mathcal{T}}} \def\cU{{\mathcal{U}}}     

\newcommand\Aut{\operatorname{Aut}}

\newcommand\dom{\operatorname{dom}}

\newcommand\esssup{\operatorname{esssup}}

\newcommand\Fix{{\operatorname{Fix}}}

\newcommand\id{\operatorname{id}}

\renewcommand\Im{\operatorname{Im}}

\newcommand\Meas{{\operatorname{Meas}}}

\newcommand\Prob{\operatorname{Prob}}
\newcommand\ran{\operatorname{ran}}

\newcommand\supp{\operatorname{supp}}

\newcommand\Span{\operatorname{span}}

\newcommand{\actson}{\curvearrowright}
\newcommand{\actons}{\curvearrowright}
\newcommand{\acston}{\actson}

\newcommand{\ip}[1]{\langle #1 \rangle}

\begin{document}
\title{Coamenability and cospectral radius for orbit equivalence relations}

\author{Ben Hayes}
\address{Department of Mathematics, University of Virginia\\
141 Cabell Drive, Kerchof Hall
P.O. Box 400137,
Charlottesville, VA 22904}
\email{brh5c@virginia.edu}

\begin{abstract}
We consider inclusions $\cS\leq \cR$ of discrete, probability measure-preserving orbit equivalence relations. In previous work with Ab\'{e}rt-Fra\c{c}zyk, we established the pointwise almost sure existence of the cospectral radius of a random walk on the $\cR$-classes. In this paper, we investigate the connections of this cospectral radius to the coamenability of the inclusion $\cS\leq \cR$. We also undertake a systematic study of coamenability for inclusions of relations, establishing several equivalent formulations of this notion.    
\end{abstract}

\maketitle

\section{Introduction}
This paper is concerned with developing a satisfactory understanding of what it means for an inclusion of orbit equivalence relations to be coamenable.
Recall that if $(X,\mu)$ is a standard probability space a \emph{discrete, probability measure-preserving orbit equivalence relation on $(X,\mu)$} (we will frequently drop ``orbit") is a Borel subset $\cR\subseteq X\times X$ so that 
\begin{itemize}
    \item the relation $x\thicksim y$ if and only if $(x,y)\in \cR$ is an equivalence relation,
    \item for every $x\in X$, we have $[x]_{\cR}:=\{y\in [x]_{\cR}:(x,y)\in \cR\}$ is countable,
    \item for every Borel $f\colon \cR\to [0,+\infty]$ the \emph{mass-transport principle} holds:
    \[\int_{X}\sum_{y\in [x]_{\cR}}f(x,y)\,d\mu(x)=\int_{X}\sum_{y\in [x]_{\cR}}f(y,x)\,d\mu(x).\]
\end{itemize}
Note that if $\Gamma\actons (X,\mu)$ is a probability-measure preserving action with $\Gamma$ countable, then $\cR_{\Gamma,X}:=\{(x,gx):g\in\Gamma,x\in X\}$ is a discrete, probability measure-preserving equivalence relation. In fact, all such relations arise this way \cite[Theorem 1]{FelMooreI}. Additionally, there is a natural notion of isomorphism for equivalence relations, and two groups are orbit equivalent if and only if they have essentially free  probability measure-preserving actions whose corresponding orbit equivalent relations are isomorphic. On the other hand, much of the theory of orbit equivalence relations can be developed by analogy with, and often in parallel to, the study of analytic properties of discrete groups. For example,
many salient notions for groups such as amenability, Property T, the Haagerup property etc. can be defined for equivalence relations in such a way that essentially free actions will have/fail to have the relevant property if and only if the group has/fails to have this property \cite{BOExactness, BOMd, FurmanMeasRigid, HAJoli, JoliWA, IshanGroupApprox, OzawaExact, PopaL2Betti}. A key observation for this is that for $\cR=\cR_{\Gamma,X}$ we have a natural map from $\Gamma$ into the full group, denoted $[\cR]$, of $\cR$ which is all $\gamma\in \Aut(X,\mu)$ so that $\gamma(x)\in [x]_{\cR}$ for almost every $x\in X$.
Thus equivalence relations provide a rich and useful abstraction of groups that remain close enough to groups to be a tractable area of study.

In this framework, amenability has a particular nice formulation: a discrete, probability measure-preserving equivalence relation is amenable if and only if it is \emph{hyperfinite}, namely it can be written as the increasing union (modulo null sets) of equivalence relations which have almost every equivalence class finite \cite{CFW, Dye, OrnWeiss}. This result is directly inspired by Connes' fundamental work showing amenability and hyperfiniteness are equivalent for von Neumann algebras \cite{Connes} (indeed a discrete, probability measure-preserving equivalence relation  is amenable if and only if the corresponding von Neumann algebra is). Hyperfinite relations can typically be studied by reducing to the case of a relation with finite equivalence classes, which is usually direct to understand.  
As in the case of groups, there are many equivalent formulations of amenability of probability measure-preserving relations and the seminal works \cite{CFW, OrnWeiss, ZimmerAMen1, ZimmerAmen2} lay out some of the most important formulations. See also \cite{KechrisAmenVersusHyper, DKStrucutreHyper} for important works investigating the equivalence of amenability and hyperfiniteness in the absence of a quasi-invariant measure, as well as \cite{JKL, KM} for excellent references on hyperfiniteness in general including comparison between \emph{Borel hyperfiniteness and measure hyperfiniteness} (throughout the paper we will use hyperfiniteness to mean measure hyperfiniteness).  

Amenability is ubiquitous in the study of analysis on discrete groups, particularly ergodic theory and functional analysis. Amenability also arises in the theory of random walks on groups. Let $\Gamma$ be a group and $\nu\in \Prob(\Gamma)$ be \emph{symmetric} (i.e. $\nu(\{g\})=\nu(\{g^{-1}\})$ for all $g\in \Gamma$).  We then have a natural random walk on $\Gamma$ which a sequence of $\Gamma$-valued random variables $(X_{n})_{n=0}^{\infty}$ where $X_{0}=e$ and with transition probabilities
\[\P(X_{n}=a|X_{n-1}=b)=\nu(b^{-1}a).\]
A well-studied quantity associated to such a random walk is the \emph{spectral radius}, defined as $r(\nu)=\lim_{n\to\infty}\P(X_{2n}=e)^{1/2n}$.  This quantity can also be shown to be equal to the operator norm of the \emph{Markov operator} associated to $\nu$, which is given by convolving with $\nu$ on $\ell^{2}(G)$ (see 
 \cite[Lemma 10.1]{Woess}). For $\nu$ as above we have that $r(\nu)=1$ if and only if $\ip{\supp(\nu)}$ is amenable (see \cite[Corollary 12.5]{Woess}). This provides a tight connection between operator theory, probability theory, and amenability.

Being such a successful notion, it is natural to consider when it can be relativized. Given an inclusion $\Lambda\leq \Gamma$ of countable groups, a notion of $\Lambda$ being coamenable in $\Gamma$ was developed in \cite{EymardCA,  MonodPopa, PopaCorr}. Coamenability of groups is closely related to amenability of actions on discrete sets: an action $\Gamma\actson X$ is amenable if and only if all stabilizers are coamenable subgroups. 
Coamenability of groups and amenability of actions has occurred in various contexts, see e.g. \cite[Corollary 4.2]{ElekSzaboDeterminant}, \cite{GMAmen, GNARelAmen, JMSimple, PAmen, TuckerMeans}.
Additionally, coamenability of $\Lambda\leq \Gamma$ can be formulated in terms of the cospectral radius $\lim_{n\to\infty}\P(X_{2n}\in \Lambda)^{1/2n}$ of the random walks in the above paragraph (we remark that the Cohen-Grigorchuk cogrowth formula \cite{CohenCG, GrigoCG} gives a different connection between random walks and coamenability).

It is thus natural to try and develop an equivalence relation analogue, and explore its connections to random walks on equivalence relations. Fortunately, a route to do this already exists. Namely, in \cite{PopaCorr} a notion of coamenability for von Neumann algebras was developed (see \cite{MonodPopa, 
OzPopaCartan,  PopaAI, PopaL2Betti, PopaVaesFree} for applications of this notion). Given an inclusion $\cS\leq \cR$ of discrete, probability measure-preserving equivalence relations, it is thus implicit from \cite{PopaCorr} to define $\cS$ to be coamenable in $\cR$ if the von Neumann algebra of $\cS$ is coamenable in the von Neumann algebra of $\cR$. What is not immediate from such a definition however, is how to reformulate this in terms of other well-known formulations of amenability such as almost-invariant vectors on $L^{p}$, spectral radius, F\o lner sets, cospectral radius etc.

Indeed, one difficulty is that the notion of cospectral radius is nontrivial to even define for relations. It is direct to establish the existence of $\lim_{n\to\infty}\P(X_{2n}\in \Lambda)^{1/2n}$ for the types of random walks on $\Gamma\geq \Lambda$ we discussed before. However, existence of the analogous limit is very nontrivial to establish in the relation case. If $\cR$ is a discrete, probability measure-preserving equivalence relation on $(X,\mu)$, we use the notation $\cS\leq \cR$ to mean that $\cS$ is a Borel subset of $\cR$ and that the relation $x\thicksim y$ if and only if $(x,y)\in \cS$ is an equivalence relation. We will often just call $\cS$ a \emph{subrelation of $\cR$.}
Given an inclusion $\cS\leq \cR$ of discrete, probability measure-preserving equivalence relations, and $\nu\in \Prob([\cR])$ we have a random walk $(X_{n,x})_{n=0}^{\infty}$ on $[x]_{\cR}$ with $X_{0,x}=x$ and with transition probabilities
\[\P(X_{n,x}=y|X_{n-1,x}=z)=\nu(\{\gamma\in [\cR]:\gamma(y)=z\}).\]
In \cite[Theorem 1.2]{AFH} (see also \cite{AFHGrowth}) it is shown that the pointwise cospectral radius
\[\rho^{\cS}_{\nu}(x)=\lim_{n\to\infty}\P(X_{2n,x}\in [x]_{\cS})^{1/2n}\]
exists almost everywhere. 
In contrast to the group situation, the proof of this is nontrivial and there is no obvious submultiplicativity to exploit (see \cite[discussion following Theorem 1.2]{AFH} for a more detailed discussion).

If $\nu$ is symmetric and \emph{generates $\cR$} (meaning that $\cR=\cR_{\ip{\supp(\nu)},X}$ modulo null sets), then the essential supremum of $\rho^{\cS}_{\nu}$ is the norm of a natural Markov operator. Since $\cS$ is a subrelation of $\cR$, every $\cR$-class is a disjoint union of $\cS$-equivalence classes. For $x\in X$, we use $[x]_{\cR}/\cS$ for the set of $\cS$-equivalence classes which are contained in $[x]_{\cR}$.    Consider
\[\cR/\cS:=\{(x,c):x\in X,c\in [x]_{\cR}/\cS\}.\]
It turns out there is a natural standard Borel structure on $\cR/\cS$ (see Section \ref{sec:basic construction equiv reln} for the details). 
We can equip $\cR/\cS$ with the Borel measure
\[\mu_{\cR/\cS}(E)=\int_{X}|\{c\in [x]_{\cR}/\cS:(x,c)\in E\}|\,d\mu(x).\]
We may think of $\cR/\cS$ as ``fibered over $X$" and for $\xi\in L^{p}(\cR/\cS)$ and $x\in X$, we may consider $\xi_{x}\in \ell^{p}([x]_{\cR}/\cS)$ given by $\xi_{x}(c)=\xi(x,c)$. 
As in \cite{AFH}, we have an induced unitary representation
\[\lambda_{\cR/\cS}\colon [\cR]\to \mathcal{U}(L^{2}(\cR/\cS))\]
via $(\lambda_{\cR/\cS}(\gamma)\xi)(x,c)=\xi(\gamma^{-1}(x),c).$ 
For a countably supported probability measure $\nu$ on $[\cR]$ we set
\[\lambda_{\cR/\cS}(\nu)=\sum_{\gamma\in [\cR]}\nu(\gamma)\lambda_{\cR/\cS}(\gamma).\]
We can think of $\lambda_{\cR/\cS}(\nu)$ as the appropriate Markov operator associated to the above random walk on $[x]_{\cR}$. Indeed,
when $\nu$ is generating and symmetric, we have that $\|\rho^{\cS}_{\nu}\|_{\infty}=\|\lambda_{\cR/\cS}(\nu)\|$.

Having developed a theory of cospectral radius for relations, it is natural to ask if coamenability can be formulated in terms of cospectral radius just as in the case of groups. Our main result says this is true, and also establishes several other equivalent formulations of coamenability. For simplicity, we state the results in the introduction when $\cR$ is ergodic, though we handle the general case in this paper.



\begin{thm} \label{thm: main theorem intro}
Let $\cS\leq \cR$ be discrete, probability measure-preserving equivalence relations on $(X,\mu)$ with $\cR$ ergodic. Then the following are equivalent:
\begin{enumerate}[(i)]
    \item $L(\cS)$ is coamenable in $L(\cR)$,
    \item $\lambda_{\cR/\cS}$ has almost invariant vectors, \label{item: almost invariant vectors intro}
    \item for every countably supported $\nu\in \Prob([\cR])$ which generates $\cR$, we have $\|\rho^{\cS}_{\nu}\|_{\infty}=1$, \label{item: co-spectral radius condition intro}
    \item  there is a sequence $\xi^{(n)}\in L^{2}(\cR/\cS)$ so that  for almost every $x\in X$ 
   we have $\|\xi^{(n)}_{x}\|_{2}=1$ and $\lim_{n\to\infty}\|\xi^{(n)}_{y}-\xi^{(n)}_{x}\|_{2}=0$ for all $y\in [x]_{\cR}$,
    \label{item:almost invariant fiberwise norm 1 intro}
    \item there is a countable $\Gamma\leq [\cR]$ such that $\Gamma x=[x]_{\cR}$ for almost every $x\in X$, and a sequence $\xi^{(n)}\in L^{2}(\cR/\cS)$ so that $\|\xi^{(n)}_{x}\|_{2}=1$ for almost every $x\in X$, and $\xi^{(n)}$ is $\Gamma$-almost invariant. \label{item:almost invariant fiberwise norm 1 intro 2}
\end{enumerate}
\end{thm}
There is an important subtlety here, that occurs both in the case of relations and in the case of groups. Namely, while coamenability of groups $\Lambda\leq \Gamma$ is equivalent to cospectral radius being $1$ for a measure on $\Prob(\Gamma)$ whose support generates $\Gamma$, it is not equivalent to cospectral radius $1$ for \emph{all} probability measures on $\Gamma$. This is related to examples of Monod-Popa \cite[Theorem 1]{MonodPopa} of groups $\Delta\leq\Lambda\leq \Gamma$ with $\Delta$ being coamenable in $\Gamma$, but not in $\Lambda$ (see also \cite[Section 4]{monodcomments} for many examples of this phenomenon). For the same reason, we cannot drop the assumption that $\nu$ generates $\cR$ in (\ref{item: co-spectral radius condition intro}).
 Secondly, it would be tempting (as in the group situation) to replace the assumption in (\ref{item: co-spectral radius condition intro})  of cospectral radius one for \emph{every} probability measure on $[\cR]$ which generates $\cR$ with the assumption that there is \emph{some} probability measure on $[\cR]$ which generates $\cR$ which has cospectral radius one. However, this is already false in the nonrelative setting (i.e. with $\cS$ being the trivial subrelation). Indeed, Kaimanovich in \cite{VKLeafcounterexample} showed that there is a relation which has a graphing with almost every fiberwise graph being amenable but still the relation itself is not amenable. For these reasons, we have to consider  arbitrary measures which generate the relation (however, see Section \ref{sec:cospectral radius} for cases where this issue can be removed).

We will say that $\cS$ is \emph{coamenable in $\cR$} if any of the above equivalent items hold for the inclusion $\cS\leq \cR$. We refer the reader to Section \ref{sec: relative amenability relns} for other equivalent conditions, as well as how to properly formulate these equivalences when $\cR$ is not assumed ergodic. Note that, in particular, item (\ref{item: co-spectral radius condition intro}) shows that cospectral radius is tightly related to coamenability of  $\cS\leq \cR$, just as in the group setting. This provides further justification for our notion of cospectral radius being the correct one for relations, and for being a natural object of study from the point of view of measured group theory. 

The most difficult equivalence in Theorem \ref{thm: main theorem intro} is to establish that  (\ref{item: almost invariant vectors intro}) and (\ref{item:almost invariant fiberwise norm 1 intro 2}) are equivalent. The direction (\ref{item:almost invariant fiberwise norm 1 intro 2}) implies  (\ref{item: almost invariant vectors intro}) is a fairly direct dominated convergence theorem argument (see Theorem \ref{thm: DCT argument}).
For the reverse direction, general considerations (see Theorem \ref{thm: folklore 1}) say that (\ref{item: almost invariant vectors intro}) implies the existence of an $[\cR]$-invariant mean on $L^{\infty}(\cR/\cS)$. When $\cR$ is ergodic, such an invariant mean must agree with integration against $\mu$ when restricted to the copy of $L^{\infty}(X,\mu)$ given by $(f\mapsto ((x,c)\mapsto f(x)))$. General Hahn-Banach arguments imply that such a mean can be approximated in the weak$^{*}$ topology by nonnegative elements in $L^{1}(\cR/\cS)$ of norm $1$. Since the mean agrees with integration on $L^{\infty}(X)$ we know that $m$ can be weak$^{*}$ approximated by functions so that $x\mapsto \|f_{i,x}\|_{\ell^{1}([x]_{\cR})}$ is weakly close to $1$ as an element of $L^{1}(X)$.
From here, another application of Hahn-Banach then tells us that we can replace $f_{i}$ with another sequence so that $\|\|f_{i,x}\|_{\ell^{1}([x]_{\cR})}-1\|_{L^{1}(X)}$ tends to $0$. We can then find vectors satisfying (\ref{item:almost invariant fiberwise norm 1 intro 2}) by a simple perturbation. In the case that $\cR$ is not ergodic, we use a slightly more general result (established, e.g. in \cite[Lemma 4.2]{HVTypeIIICartan}) showing that a mean on $L^{\infty}(\cR/\cS)$ which agrees with integration against $\mu$ on all $[\cR]$-invariant elements of $L^{\infty}(X,\mu)$ must in fact agree with integration against $\mu$ on all of $L^{\infty}(X,\mu)$. We remark that in the situation of plain amenability of $\cR$, many previous arguments establishing that the analogues of (\ref{item: almost invariant vectors intro})  and (\ref{item:almost invariant fiberwise norm 1 intro 2}) are each equivalent to amenability of $\cR$ went through first establishing that (\ref{item: almost invariant vectors intro}) implies that $\cR$ is hyperfinite (see e.g. \cite{KaimanovichAmenability, CFW}). As we will remark later, we do not know (and indeed it is likely false) whether coamenability implies some version of cohyperfiniteness and so we do not have access to this argument.

For inclusions $\Lambda\leq \Gamma$ of countable groups, we have that $\Lambda$ is coamenable in $\Gamma$ if and only if $\rho^{\Lambda}_{\nu}=1$ for every symmetric $\nu\in \Prob(\Gamma)$ whose support generates $\Gamma$. Thus, it is tempting to guess that $\cS$ being coamenable in $\cR$ is equivalent to $\rho^{\cS}_{\nu}=1$ almost everywhere for every probability measure on $[\cR]$ whose support generates $\cR$. However, this is not true. A counterexample can be given as follows. Recall that if $\cR$ is a discrete, probability measure-preserving equivalence relation on $(X,\mu)$ and $E\subseteq X$ is Borel and $\mu(E)>0$, then we have a restricted relation $\cR|_{E}:=\cR\cap (E\times E)$ on $E$, which preserves the measure $\frac{\mu|_{E}}{\mu(E)}$. Take a nonamenable group $\Gamma$, and consider a free, ergodic probability measure-preserving action $\Gamma\actson (X,\mu)$, and fix a measurable $E\subseteq X$ with $0<\mu(E)<1$. Let $\cR$ be the orbit equivalence relation of $\Gamma\actson (X,\mu)$, and let $\cS$ be the subrelation so that $\cS|_{E}=\cR|_{E}$ and $\cS|_{E^{c}}$ is the trivial relation. Then one can show  that $\cS$ is coamenable  in $\cR$, but for any $\nu\in [\cR]$ which is supported on $\Gamma$ and such that $\ip{\supp(\nu)}=\Gamma$ we have that $\rho^{\cS}_{\nu}|_{E^{c}}<1$ almost surely. See Example \ref{example: rel amen comes from small pieces} later in the paper for the details. 

Essentially what goes wrong in this example is the following. If we take any sequence $\xi^{(n)}\in L^{2}(\cR/\cS)$ of almost $\Gamma$-invariant vectors with $\|\xi^{(n)}_{x}\|_{2}=1$ almost everywhere, then we must have that $|\xi^{(n)}(x,c)|$ is ``relatively evenly distributed" as a function of $x$, by almost invariance. However, we can still have that $ |\xi^{(n)}(x,c)|$ is not very evenly distributed as a function of $c$ (making this precise is somewhat difficult as one cannot measurably define something as a function of $c$, and we do not pursue it here, but the heuristic is nevertheless helpful). Specifically in the example in the above paragraph, the vectors $\xi^{(n)}$ witnessing coamenability can be taken to be entirely supported on those $(x,c)$ where $c\subseteq E$.

A fix for this, which rules out the example above, is to use the notion of \emph{everywhere coamenability}.  
We say that $\cS$ is \textbf{everywhere coamenable in $\cR$} if for every measurable, positive measure $E\subseteq X$ we have that $\cS|_{E}$ is coamenable in $\cR|_{E}$. With this stronger version of coamenability, we can assert that the cospectral radius is pointwise almost everywhere equal to $1$. 

\begin{thm}\label{thm: special case cs equal intro}
Let $\cS\leq \cR$ be discrete, probability measure-preserving equivalence relations on a standard probability space $(X,\mu)$. Then $\cS$ is everywhere coamenable in $\cR$ if and only if for every $\nu\in \Prob([\cR])$ which is countably supported and generates $\cR$ we have $\rho^{\cS}_{\nu}=1$ almost everywhere.
\end{thm}

In fact, we can prove a slightly more general result which we highlight here.

\begin{thm}\label{thm: cs equality intro}
 Let $\cS_{1}\leq \cS_{2}\leq \cR$ be discrete, probability measure-preserving equivalence relations on a standard probability space $(X,\mu)$. Let $\nu\in \Prob([\cR])$ be countable supported and generate $\cR$. If $\cS_{1}$ is everywhere coamenable in $\cS_{2}$, then $\rho^{\cS_{1}}_{\nu}=\rho^{\cS_{2}}_{\nu}$ almost everywhere.   
\end{thm}

 The proof of Theorem \ref{thm: cs equality intro}  proceeds via equating the essential supremum of $\rho^{\cS_{i}}_{\nu}$ over a measurable subset $E\subseteq X$ with the norm of the Markov operator $\lambda_{\cR/\cS}(\nu)$ restricted to a certain invariant subspace of $L^{2}(\cR/\cS)$. This reduces  Theorem \ref{thm: cs equality intro} to an inequality of operator norms. From here we can follow the proof in the group setting by lifting almost invariant vectors in one coset space to another. We refer the reader to Proposition \ref{prop: co spectral radius descent} for precise details. Thus our methods show the utility of connecting cospectral radius to operator norms and applying Hilbert space methods. Given the connections between orbit equivalence relations, percolation theory, and unimodular random graphs (see, e.g. \cite{BHAExtension, SolidErg, Gaboriau-Lyons, LyonsSchramm}) particularly in relations to cospectral radius (see \cite[Proposition 2.1 and Section 5]{AFH} and also \cite{AFHGrowth}), we expect these Hilbert space methods to have interesting applications to the edge/vertex percolations on graphs (particularly Cayley graphs or unimodular random graphs).

The case where $\cS_{1}$ is the trivial subrelation is shown in \cite[Proposition 4.1]{AFH}.
In that case,  $\cS_{2}$ is amenable. The proof given in \cite{AFH} uses that amenable relations are hyperfinite. We cannot assert that $\cS_{1}$ being coamenable in $\cS_{2}$ implies that $\cS_{2}$ is ``hyperfinite" over $\cS_{1}$. 
As we alluded to before, we do not know if coamenability of $\cR/\cS$ implies some version of ``cohyperfiniteness" of $\cS$ inside of $\cR$. Indeed, the fact that coamenability does not coincide with everywhere coamenability indicates that such an implication would not hold. We leave it as an open problem to see if this implication holds with coamenability replaced with everywhere coamenability. 

\begin{problem}
Suppose that $\cS\leq \cR$ are discrete, probability measure-preserving equivalence relations over a standard probability space $(X,\mu)$. If $\cS$ is everywhere coamenable in $\cR$, does it follow that $\cR$ is ``hyperfinite over $\cS$" in some suitable sense?   
\end{problem}

We suspect that the answer is ``no". Indeed, as previously mentioned relative amenability has been introduced and extensively studied in the von Neumann algebras situation in \cite{PopaCorr}. No equivalence in the realm of von Neumann algebras has been established between relative amenability and some kind of ``relative hyperfiniteness". Given the tight connections between results and methods in the von Neumann algebras and measured group theory contexts (see e.g. the similarities between \cite{Connes} and \cite{CFW}), we expect similar difficulties to arise for equivalence relations. 


In the group case, it is direct to establish that if $\Lambda\leq \Gamma$ are countable groups, and $\Gamma\actson (X,\mu)$ is an essentially free, probability measure-preserving action, then $\cR_{\Lambda,X}$ is coamenable in $\cR_{\Gamma,X}$ if and only if $\Lambda$ is coamenable in $\Gamma$. We in fact establish a stronger connection to \emph{everywhere coamenability.}

\begin{thm}\label{thm: into rel amen of groups gives rel amen of relsn}
Let $(X,\mu)$ be a standard probability space, $\Lambda\leq \Gamma$ countable groups, and $\Gamma\actson (X,\mu)$ a probability measure-preserving action. Then: 
\begin{enumerate}[(i)]
    \item if $\Lambda$ is coamenable in $\Gamma$, then $\cR_{\Lambda,X}$ is everywhere coamenable in $\cR_{\Gamma,X}$,
    \item If $\Gamma\actson (X,\mu)$ is essentially free and $\Lambda$ is not coamenable in $\Gamma$, then for every measurable $E\subseteq X$ with $\mu(E)>0$ we have that $\cR_{\Lambda,X}|_{E}$ is not coamenable in $\cR_{\Gamma,X}|_{E}$.
\end{enumerate}

\end{thm}

Note that Theorem \ref{thm: into rel amen of groups gives rel amen of relsn} says that coamenability for inclusion of groups is invariant under \emph{measure equivalence of inclusions}. Here given inclusions $\Lambda_{i}\leq \Gamma_{i}$, $i=1,2$ of countable groups we say that $\Lambda_{1}\leq \Gamma_{1}$ is measure equivalent to $\Lambda_{2}\leq \Gamma_{2}$ if there are free, probability measure-preserving actions $\Gamma_{i}\actson (X_{i},\mu_{i})$, $i=1,2$ on standard probability spaces, and measurable sets $E_{i}\subseteq X_{i}$ with $\mu_{i}(E_{i})>0$ and a measure-preserving isomorphism $\Theta\colon (E_{1},\frac{\mu_{1}|_{E_{1}}}{\mu_{1}(E_{1})})\to (E_{2},\frac{\mu_{1}|_{E_{2}}}{\mu_{2}(E_{2})})$ so that $\Theta(\Gamma_{1}x\cap E_{1})=\Gamma_{2}\Theta(x)\cap E_{2}$ and $\Theta(\Lambda_{1}x \cap E_{1})=\Lambda_{2}x \cap E_{2}$ for almost every $x\in E_{1}$. A primary focus of measured group theory is the understanding of measure equivalence of groups (namely the case $\Lambda_{i}=\{e\}$ in the above). However, measure equivalence of inclusions implicitly appears in proofs of various results in the study of measure equivalence such as orbit equivalence being preserved by certain amalgamated free products \cite{GaboriauMEFree} and also by graph products \cite[Proposition 4.2]{HHGraphOE},\cite{demir2024measurableimbeddingsfreeproducts}.

\subsection{Remark}
What we call ``coamenability of von Neumann algebras" is sometimes called ``relative amenability" in the literature. On the other hand, there is an existing notion of ``relative amenability" for inclusions of groups \cite{RelAmenGroups}, which is different from coamenability of groups as used in this article. We will stick to ``coamenability" throughout, but the reader should double check conventions when consulting the literature on this topic.  Additionally, some references use ``countable equivalence relation" for what we will call a ``discrete equivalence relation".

\subsection{Organization of the paper} In Section \ref{sec:preliminaries}, we give the main preliminaries of the paper including the core background definitions. We also give the construction of the Borel structure and appropriate measure on $\cR/\cS$ here,  and state background on cospectral radius from \cite{AFH}. In this section, we also state some results in the literature connecting means, almost invariant vectors, and F\o lner sets for measure-preserving actions on $\sigma$-finite spaces. Stating these results in the appropriate generality will streamline much of the rest of the paper. In Section \ref{sec: relative amenability relns}, we prove the equivalence 
 of Theorem \ref{thm: main theorem intro} (\ref{item: almost invariant vectors intro})-(\ref{item:almost invariant fiberwise norm 1 intro 2}), and several other equivalence notions of coamenability.  In Section \ref{sec: pointwise cospectral radius}, we investigate everywhere coamenability and prove Theorems \ref{thm: special case cs equal intro} and \ref{thm: cs equality intro}. As mentioned above, in general in order to prove an inclusion of relations is coamenable, it is necessary to prove that the cospectral radius has essential supremum $1$ for every measure which generates the bigger relation. In Section \ref{sec:cospectral radius}, we show that this can be reduced to checking that the cospectral radius has essential supremum $1$ for some measure which generates the bigger relation, under an assumption of spectral gap of the action on the base space.  We use this to give an approach to understanding cospectral radius of i.i.d edge percolation when there are infinitely many infinite clusters.

In Section \ref{sec: vNa stuff}, we establish the equivalence between coamenability of von Neumann algebras, and the other items in Theorem \ref{thm: main theorem intro}.
The notion of coamenability for von Neumann algebras takes as input the \emph{Jones basic construction} (introduced in \cite{Jones83}) associated to an inclusion of von Neumann algebras . For the case of von Neumann algebras coming from relations, we can show that the basic construction is the same as the von Neumann algebra of the measure-preserving relation $\widehat{\cS}$ on the $\sigma$-finite space $(\cR/\cS,\mu_{\cR/\cS})$ given by $(x,c)\thicksim_{\widehat{\cS}}(y,c')$ if and only if $(x,y)\in \cR$ and $c=c'$. This is a result of Feldman-Sutherland-Zimmer (see \cite[Section 1]{FSZ}). Since we are unable to find an explicit proof in the literature, we give a proof in Section \ref{sec: basic construction}, which also contains a detailed discussion on what the von Neumann algebra of an equivalence relation is. 

In Section \ref{sec; general prop}, we investigate general properties of coamenability, including a proof of Theorem \ref{thm: into rel amen of groups gives rel amen of relsn}. Finally, in Section \ref{sec: ergodic decomp} we investigate how coamenability/everywhere coamenability behaves under the ergodic decomposition, showing that coamenability/everywhere coamenability are equivalent to almost every ergodic fiber being coamenable/everywhere coamenable.

\subsection*{Acknowledgements}
I thank Miklos Ab\'{e}rt, Miko\l aj Fra\c{c}zyk, Srivatsav Kunnawalkam Elayavalli, Felipe Flores, James Harbour, Samuel Mellick, and Robin Tucker-Drob for insightful conversations related to this topic. I thank Anush Tserunyan for numerous comments on an earlier draft of this paper. I thank the anonymous referee for their numerous comments, which greatly improved the readability of the paper. Part of this work was influenced by discussions at the ``Probability, Dynamics and the Geometry of Groups" conference at the University of M\"{u}nster in September 2024. I thank the organizers of this conference for hosting this program. 
The author acknowledges support from the NSF CAREER award DMS-21447.

\tableofcontents

\section{Preliminaries}\label{sec:preliminaries}

\subsection{Background definitions and notational conventions}
A \emph{standard measure space} is a  pair $(X,\mu)$ where $X$ is a standard Borel space, and $\mu$ is a $\sigma$-finite Borel measure on $X$. It is called a \emph{standard probability space} if $\mu$ is a probability measure. 
We say that $E\subseteq X$ is \emph{$\mu$-measurable} if it is in the domain of the completion of $\mu$. If $\mu$ is clear from the context, we will often simply say ``measurable".
An \emph{equivalence relation over $(X,\mu)$} is a Borel subset $\cR\subseteq X\times X$ so that the relation $\thicksim$ on $X$ given by $x\thicksim y$ if $(x,y)\in \cR$ is an equivalence relation. For $x\in X$, we let $[x]_{\cR}=\{y\in X:(x,y)\in \cR\}$. We say that $\cR$ is \emph{discrete} if for almost every $x\in X$ we have that $[x]_{\cR}$ is countable. If $\cR$ is discrete, we may turn $\cR$ into a $\sigma$-finite measure space by endowing $\cR$ with the Borel measure
\[\mu_{\cR}(E)=\int_{X}|\{y:(x,y)\in E\}|\,d\mu(x)\mbox{ for all Borel $E\subseteq \cR$.}\]
We will continue to use $\mu_{\cR}$ for the completion of $\mu_{\cR}$. If $\cR$ is discrete, we say that it is \emph{measure-preserving} if the map $\cR\to \cR$ given by $(x,y)\mapsto (y,x)$ preserves $\mu_{\cR}$. We say that $\cR$ is \emph{probability measure-preserving} if $\cR$ is  measure-preserving and $\mu$ is a probability measure.  If $(X,\mu)$ is a probability space, and $E\subseteq X$ is Borel with $\mu(E)>0$ we let
\[\cR|_{E}=\cR\cap (E\times E).\]
If $E$ has positive measure, then $\cR|_{E}$ is a measure-preserving relation over the probability space $(E,\frac{\mu(E\cap \cdot)}{\mu(E)})$.  If $E\subseteq X$, we use $\cR E$ for its \emph{saturation}, given by \[\cR E=\{y\in X:\textnormal{ there exists } x\in E \textnormal{ with } (x,y)\in \cR\}.\]
Using \cite[Theorem 1]{FelMooreI}, we see that if $E$ is Borel (resp. measurable) then $\cR E$ is Borel (resp. measurable).
We use $\cS\leq \cR$ to mean that $\cS$ is a Borel subset of $\cR$ which is also a Borel equivalence relation over $(X,\mu)$. For a group $\Gamma$, and $S\subseteq \Gamma$, we use $\ip{S}$ for the subgroup of $\Gamma$ generated by  $S$.
If $\Gamma$ is a countable group, and $\Gamma\actson (X,\mu)$ is a measure-preserving action, with $\Gamma\actson X$ Borel, then $\cR_{\Gamma,X}=\{(x,gx):g\in \Gamma\}$ is a discrete, measure-preserving equivalence relation. We let $[\cR]$ be the group of all bimeasurable bijections $\phi\colon X\to X$ so that $\phi(x)\in [x]_{\cR}$ for almost every $x\in X$. We identify two elements of $[\cR]$ if they agree almost everywhere. We have a natural metric $d$ on $[\cR]$ given by
\[d(\phi,\psi)=\mu(\{x\in X:\phi(x)\ne \psi(x)\}).\]
This is a complete, separable, translation-invariant metric on $[\cR]$ and this turns $[\cR]$ into a \emph{Polish group}. We let $[[\cR]]$ be all bimeasurable bijections $\gamma\colon A\to B$ such that:
\begin{itemize}
\item $A,B\subseteq X$ are measurable,
\item $\gamma(x)\in [x]_{\cR}$ for almost every $x\in A$.
\end{itemize}
We usually use $\dom(\gamma),\ran(\gamma)$ for $A,B$. We typically identify two elements of $[[\cR]]$ if they agree up to sets of measure zero. However just as with the theory of $L^{p}$-spaces we will often blur the lines between a bimeasurable bijection as above and its equivalence class in $[[\cR]]$. Thus, for example, we will say that $\gamma\in [[\cR]]$ is \emph{Borel} to mean a bimeasurable bijection as above where $A,B$ are Borel subsets of $X$ and $\gamma$ is a Borel map. 
For $\gamma\in [[\cR]]$, we let $\gamma^{-1}$ be the compositional inverse of $\gamma$. 
If $\gamma_{1},\gamma_{2}\in [[\cR]]$ we define $\gamma_{1}\circ \gamma_{2}$ by saying that $\dom(\gamma_{1}\circ \gamma_{2})=\dom(\gamma_{2})\cap \gamma_{2}^{-1}(\dom(\gamma_{1}))$, and $(\gamma_{1}\circ \gamma_{2})(x)=\gamma_{1}(\gamma_{2}(x))$. 
 For $\gamma\in [[\cR]]$,  and $f\colon \cR\to \C$ measurable, we let $\lambda(\gamma)(f)$ be the measurable function defined by
\[(\lambda(\gamma)(f))(x)=1_{\ran(\gamma)}(x)f(\gamma^{-1}(x)).\]

We use $\Prob([\cR])$ for the Borel probability measures on $[\cR]$. Since $[\cR]$ is a Polish group, the space $\Prob([\cR])$ can be made into a semigroup under convolution: so if $\nu_{1},\nu_{2}\in \Prob([\cR])$, then $\nu_{1}*\nu_{2}\in \Prob([\cR])$ is defined by
\[(\nu_{1}*\nu_{2})(\Omega)=\nu_{1}\otimes \nu_{2}(\{(\phi,\psi):\phi\psi\in \Omega\}).\]
Given a countable $\Gamma\leq [\cR]$, we say that $\Gamma$ \emph{generates} $\cR$ if $\Gamma x=[x]_{\cR}$ for almost every $x\in X$. Given a countably supported $\nu\in \Prob([\cR])$, we say that $\nu$ \emph{generates $\cR$} if $\ip{\supp(\nu)}$ generates $\cR$, where $\supp(\nu)=\{\phi \in [\cR]:\nu(\{\phi\})\ne 0\}$. We say that $\nu\in \Prob([\cR])$ is \emph{symmetric} if the map $\phi\mapsto \phi^{-1}$ preserves $\nu$.

For a Banach space $V$, we use $V^{*}$ for the \emph{continuous dual of $V$}, namely all continuous linear functionals $\phi\colon V\to \C$. For $V,W$ Banach spaces, we let $B(V,W)$ be all continuous linear operators $V\to W$ and use $B(V)$ if $V=W$. We endow $B(V,W)$ with the operator norm (note this includes the case of $V^{*}=B(V,\C)$).
For a $\sigma$-finite measure space $(Y,\nu)$, and $p\in [1,+\infty]$ we let $L^{p}(Y,\nu)_{+}$ be the functions $f\in L^{p}(Y,\nu)$ so that $f$ is almost everywhere nonnegative valued and use $f\geq 0$ as shorthand for saying that $f$ is almost everywhere nonnegative valued.
We say that $m\in L^{\infty}(Y,\nu)^{*}$ is a \emph{mean} on $X$ if $f\geq 0$ implies that $m(f)\geq 0$ and if $m(1)=1$. If $\Gamma$ is a group of measure-class preserving transformations of $(Y,\nu)$, we say that $m$ is \emph{$\Gamma$-invariant} if $m(f\circ \gamma)=m(f)$ for all $f\in L^{\infty}(Y,\nu)$, $\gamma\in \Gamma$. 

\subsection{Basic construction for equivalence relations} \label{sec:basic construction equiv reln}

Let $\cS\leq \cR$ be discrete, probability measure-preserving equivalence relations on a standard probability space $(X,\mu)$. Since $\cS\leq \cR$, for $x\in X$ we have that $\cS$ restricts to an equivalence relation on $[x]_{\cR}$. We use $[x]_{\cR}/\cS$ for the space of $\cS$-equivalence classes in $[x]_{\cR}$. We let $\cR/\cS=\{(x,c):x\in X,c\in [x]_{\cR}/\cS\}$. We  declare that $E\subseteq \cR/\cS$ is Borel if and only if for every Borel $\gamma\in [\cR]$, we have that $E_{\gamma}:=\{x\in X:(x,[\gamma(x)]_{\cS})\in E\}$ is Borel. Note that the Borel subsets of $\cR/\cS$ form a $\sigma$-algebra. If $Y$ is standard Borel we say that $f\colon \cR/\cS\to Y$ is \emph{Borel} if the inverse image of a Borel subset of $Y$ is Borel. The following proposition justifies this terminology by showing that $\cR/\cS$ equipped with this $\sigma$-algebra is a standard Borel space.
\begin{prop}\label{prop: standard}
Let $(X,\mu)$ be a standard probability space, and let $\cS\leq \cR$ be discrete, probability measure-preserving equivalence relations on $(X,\mu)$. 
\begin{enumerate}[(i)]
    \item If $\Gamma\leq [\cR]$  is countable and $\cR=\cR_{\Gamma,X}$, then $E\subseteq \cR/\cS$ is Borel if and only if for all $\gamma\in \Gamma$, we have that $\{x\in X:(x,[\gamma (x)]_{\cS})\in E\}$ is Borel. More generally, if $Y$ is standard Borel and $f\colon \cR/\cS\to Y$, then $f$ is Borel if and only if for every $\gamma\in \Gamma$ we have that $f_{\gamma}(x)=f(x,[\gamma x]_{\cS})$ is Borel.
    \label{item: Borel sections}
    \item $\cR/\cS$ with the above collection of Borel sets is a standard Borel space.
    \label{item: standard Borel space}
    \item If $f\colon \cR/\cS\to [0,+\infty]$ is Borel, then $\widetilde{f}\colon X\to [0,+\infty]$ given by $\widetilde{f}(x)=\sum_{c\in [x]_{\cR}/\cS}f(x,c)$ is Borel. \label{item: Borel integrand}
\end{enumerate}

\end{prop}

\begin{proof}

(\ref{item: Borel sections}): Given $\gamma\in [\cR]$, by countability of $\Gamma$, we may find Borel sets $(A_{g})_{g\in\Gamma}$ with $A_{g}\subseteq \{x:\gamma(x)=gx\}$ and $X=\bigsqcup_{g\in \Gamma}A_{g}$. Then:
\[\{x\in X:(x,[\gamma(x)]_{\cS})\in E\}=\bigsqcup_{g\in \Gamma}\{x\in A_{g}:(x,[gx])\in E\}=\bigsqcup_{g\in \Gamma}\left(\{x\in X:(x,[gx])\in E\}\cap A_{g}\right),\]
is Borel. 

For the ``more generally" claim, we simply note that if $B\subseteq Y$ is Borel, then 
\[f^{-1}(B)=\bigcup_{\gamma\in \Gamma}\{(x,[\gamma(x)]_{\cS}):x\in f_{\gamma}^{-1}(B)\},\]
\[f_{\gamma}^{-1}(B)=\{x\in X:(x,[\gamma(x)]_{\cS})\in f^{-1}(B)\}.\]
This reduces the intended result about functions to the case of sets.

(\ref{item: standard Borel space}): By \cite[Theorem 1]{FelMooreI}, we may fix  a countable $\Gamma\leq [\cR]$ whose elements are Borel and so that $\cR=\cR_{\Gamma,X}$. Fix an enumeration $(g_{n})_{n=1}^{\infty}$ of $\Gamma$ with $g_{1}=e$. Inductively define sets $(Y_{n})_{n=1}^{\infty}$ by $Y_{1}=\{(x,[x]_{\cS}):x\in X\}$, and 
\[Y_{n}=\{(x,[g_{n}x]_{\cS}):x\in X\}\setminus \bigsqcup_{j=1}^{n-1}\{(x,[g_{j}x]_{\cS}):x\in X\}=\bigcap_{j=1}^{n-1}\{(x,[g_{n}x]_{\cS}):x\in X, (g_{n}x,g_{j}x)\notin \cS\}.\]
Since $\cS$ is Borel, it follows that $Y_{n}$ is Borel. Moreover, $Y_{n}$ as a measurable space is isomorphic with
\[\bigcap_{j=1}^{n}\{x\in X:(g_{n}x,g_{j}x)\notin \cS\},\]
which is a  Borel subset of $X$. Thus $Y_{n}$ is a standard Borel space \cite[Corollary 13.4]{KechrisClassic}, and since 
\[\cR/\cS=\bigsqcup_{n}Y_{n},\]
we see that $\cR/\cS$ is standard Borel.

(\ref{item: Borel integrand}):
Let $\Gamma,(g_{n})_{n},(Y_{n})_{n}$ be as in (\ref{item: standard Borel space}).
Let $E_{n}=\{x\in X:(x,[g_{n}x]_{\cS})\in Y_{n}\}$, so that $E_{n}$ is Borel.
Then
\[\widetilde{f}(x)=\sum_{n=1}^{\infty}1_{E_{n}}(x)f(x,[g_{n}x]_{\cS}),\]
is a sum of Borel functions, hence Borel.
\end{proof}

We define a $\sigma$-finite Borel measure $\mu_{\cR/\cS}$ on $\cR/\cS$ by
\[\mu_{\cR/\cS}(E)=\int_{X}|\{c\in [x]_{\cR}/\cS:(x,c)\in E\}|\,d\mu(x).\]
Note that the above integrand is Borel, by Proposition \ref{prop: standard}. Approximating Borel functions by simple Borel functions, we see that 
\[\int f\,d\mu_{\cR/\cS}=\int \sum_{c\in [x]_{\cR/\cS}}f(x,c)\,d\mu(x)\]
for all Borel $f\colon \cR/\cS\to [0,+\infty]$.
It will be helpful to use the following fact, whose proof is left as an exercise to the reader.

\begin{prop}\label{prop:measurability exercise}
  Let $\cS\leq \cR$ be discrete, probability measure-preserving equivalence relations on a standard probability space $(X,\mu)$, and let $f\colon \cR/\cS\to \C$. Fix a countable $\Gamma\leq [\cR]$ which generates $\cR$. Then the following are equivalent:
  \begin{enumerate}[(i)]
   \item $f$ is $\mu_{\cR/\cS}$-measurable,
   \item for every $\gamma\in [\cR]$, the function $f_{\gamma}\colon X\to \C$ given by $f_{\gamma}(x)=f(x,[\gamma(x)]_{\cS})$ is $\mu$-measurable,
   \item for every $\gamma\in\Gamma$, the function $f_{\gamma}\colon X\to \C$ given by $f_{\gamma}(x)=f(x,[\gamma(x)]_{\cS})$ is $\mu$-measurable.
  \end{enumerate}
\end{prop}

Note that the proof of Proposition \ref{prop: standard} gives an alternate proof of \cite[Lemma 1.1]{FSZ}.

\begin{cor}\label{cor:choice functions}
 Let $\cS\leq \cR$ be discrete, probability measure preserving equivalence relations on a standard probability space $(X,\mu)$. Then there is a countable family $(\sigma_{j})_{j\in J}$ of Borel functions in $[[\cR]]$ with 
 \[[x]_{\cR}=\bigsqcup_{j\in J:x\in \dom(\sigma_{j})}[\sigma_{j}(x)]_{\cS}\]
 for every $x\in X$. 
\end{cor}
We call the $\sigma_{j}$ \emph{choice functions for $\cS\leq \cR$.}

We define a new equivalence relation $\widehat{\cS}$ on $\cR/\cS$ by saying that $((x,c),(x',c'))\in \widehat{\cS}$ if $(x,x')\in \cR$ and $c=c'$. We often abuse notation and regard $\widehat{\cS}=\{(x,x',c):(x,x')\in \cR, c\in [x]_{\cR}/\cS\}$.
As observed by Feldman-Sutherland-Zimmer \cite[Introduction]{FSZ},  the relation $\widehat{\cS}$ is tightly related to the \emph{Jones basic construction} associated to an inclusion of von Neumann algebras (we will give a precise definition of this construction in Section \ref{sec: basic construction}): namely, the basic construction of the von Neumann algebra of $\cR$ over the von Neumann algebra of $\cS$ can be naturally identified with the von Neumann algebra of $\widehat{\cS}$. For completeness, we give an explicit proof of this result of Feldman-Sutherland-Zimmer in Section \ref{sec: vNa stuff}. This result can actually be used to motivate a substantial portion of the paper.
Coamenability for von Neumann algebras is defined using Jones basic construction, and so it is natural to use the space $\cR/\cS$ to define coamenability of the inclusion $\cS\leq \cR$.

Note that $[\cR]$ naturally acts on $(\cR/\cS,\mu_{\cR/\cS})$ by measuring-preserving transformations by declaring that $\gamma\cdot (x,c)=(\gamma(x),c)$ for $\gamma\in [\cR]$. We will use $\lambda_{\cR/\cS}(\gamma)$ for the induced action on functions given by $(\lambda_{\cR/\cS}(\gamma)f)(x,c)=f(\gamma^{-1}(x),c)$. This agrees (modulo unitary conjugation) with the representation which was denoted $\lambda_{\cS}$ in \cite{AFH}. As explained in \cite[Section 3]{AFH}, we may even regard the action on $L^{2}(\cR/\cS)$ as a representation of $\cR$ (which is stronger than just being a strongly continuous representation of $[\cR]$), but we will not need the theory of representations of $\cR$ here. 

\subsection{Background on cospectral radius}
Let $\cS\leq\cR$ be discrete, probability measure-preserving equivalence relations on a standard probability space $(X,\mu)$. For $\nu\in \Prob([\cR])$ which is symmetric, and $x\in X$ we consider the random walk on $[x]_{\cR}$ which starts at $x$ and with transition probabilities $p_{y,z}=\nu(\{\gamma:\gamma(y)=z\})$. For $k\in \N$, we let $p_{k,x,\cS}$ be the probability that this random walk is in $[x]_{\cS}$ after $k$ steps. By \cite[Theorem 3.2]{AFH}, we have that 
\[\rho^{\cS}_{\nu}(x)=\lim_{k\to\infty}p_{2k,x,\cS}^{1/2k}\]
exists for almost every $x\in X$. Moreover, by \cite[Theorem 3.2]{AFH}
\begin{equation}\label{eqn: co spec radi is norm background 1}
\|\rho^{\cS}_{\nu}\|_{\infty}=\|\lambda_{\cR/\cS}(\nu)\|_{B(L^{2}(\cR/\cS))}.
\end{equation}
We will also use $\rho(\cR/\cS,\nu)$ for either side of this equation. It follows from \cite[Section 3.1]{AFH} that 
\begin{equation}\label{eqn: co spec radi is norm background 2}\rho(\cR/\cS,\nu)=\lim_{k\to\infty}\ip{\lambda_{\cR/\cS}(\nu)^{2k}\xi,\xi}^{1/2k}=\lim_{k\to\infty}\left(\int p_{2k,x,\cS}\,d\mu(x)\right)^{1/2k}
\end{equation}
where $\xi\in L^{2}(\cR/\cS)$ is given by $\xi(x,c)=1_{x\in c}$.

We need a slightly more general version restricted to measurable subsets of $X$. For a measurable $E\subseteq X$ with positive measure and a $k\in\N$, we let $p_{k,x,\cS,E}$ be the probability that the random walk is in $[x]_{\cS}\cap E$ after $k$ steps. It follows from \cite[Example 3]{AFH} that 
\[\rho^{\cS}_{\nu,E}(x)=\lim_{k\to\infty}p_{2k,x,\cS,E}^{1/2k}\]
exists almost surely.  Moreover, \cite[Theorem 3.14]{AFH} implies that $\rho^{\cS}_{E,\nu}=\rho^{\cS}_{\nu}1_{\cS E}$ almost everywhere.
In \cite[Corollary 3.17 and Theorem 3.20]{AFH} it is proved that 
\begin{equation}\label{eqn: co spectral radius is norm restricted to subset background}
\|\rho^{\cS}1_{\cS E}\|_{\infty}=\lim_{k\to\infty}\ip{\lambda_{\cR/\cS}(\nu)^{2k}\zeta_{E},\zeta_{E}}^{1/2k}  
\end{equation}
where $\zeta_{E}(x,c)=\frac{1}{\sqrt{\mu(E)}}1_{x\in c\cap E}.$ 
We use $\rho_{E}(\cR/\cS,\nu)$ for either side of this equation. It is worth noting that, by $\cS$-invariance of $\rho^{\cS}_{\nu}$ and the above, we have
\[\|\rho^{\cS}1_{ E}\|_{\infty}=\lim_{k\to\infty}\ip{\lambda_{\cR/\cS}(\nu)^{2k}\zeta_{E},\zeta_{E}}^{1/2k}  \]
\subsection{Generalities on invariant means, F\o lner sets, and almost invariant vectors}

We collect here some results (mostly folklore) establishing precisely the connection between invariant means, almost invariant vectors, and F\o lner sets for measure-preserving actions on $\sigma$-finite spaces. We do this because we will often need not just the equivalence of these results, but the mechanisms precisely underlying their equivalence. 
We stress that for the results in this section, all groups will be assumed to be \emph{discrete}, but will \emph{not} always be assumed to be countable. This is because we will often want to apply our results with the group in question being the full group of a probability measure-preserving relation. A measure-preserving action of a group $\Gamma$ on a sigma-finite space $(Y,\nu)$ should thus properly be interpreted then as a homomorphism $\Gamma\to \Aut(Y,\nu)$ (where elements of $\Aut(Y,\nu)$ are identified if they agree modulo null sets). We will still slightly abuse notation and not explicitly reference the homomorphism, which means that we have for all $g,h\in \Gamma$ that $g(hy)=(gh)y$ for almost every $y\in Y$, not for \emph{every $y\in Y$}. Of course when $\Gamma$ is countable, one can circumvent this difficulty but as mentioned we have  good reasons to consider the case that $\Gamma$ is not countable.

\begin{defn}
Let $\Gamma$ be a group, and $(Y,\nu)$ a standard $\sigma$-finite space with $\Gamma\actson (Y,\nu)$ by measure-preserving transformations.   For $p\in [1,+\infty]$, and $\gamma\in \Gamma$, we define $\alpha(\gamma)\in B(L^{p}(Y,\nu))$ by $(\alpha(\gamma)f)(x)=f(\gamma^{-1}x)$. We say that:
\begin{itemize}
    \item $m\in L^{\infty}(Y,\nu)^{*}$ is a \emph{$\Gamma$-invariant mean} if $m(\alpha(\gamma)f)=m(f)$ for all $f\in L^{\infty}(Y,\nu),\gamma\in \Gamma$,
    \item  If $\Lambda\leq \Gamma$, we say that a sequence $f_{n}\in L^{p}(Y,\nu)$ are \emph{almost $\Lambda$-invariant vectors} if $\|\alpha(\gamma)f_{n}-f_{n}\|_{p}\to_{n\to\infty}0$ for all $\gamma\in \Lambda$. If $\Lambda=\Gamma$, then we usually drop ``$\Gamma$-" and say that $(f_{n})_{n}$ are almost invariant vectors.
    \item For $\Phi\subseteq \Gamma$ finite and $\varepsilon>0$, we say that a measurable $E\subseteq Y$ is \emph{$(\Phi,\varepsilon)$-F\o lner} if $0<\nu(E)<+\infty$  and $\nu(\gamma E\Delta E)<\varepsilon \nu(E)$ for all $\gamma\in \Phi$. We say that $\Gamma\actson (Y,\nu)$ \emph{admits F\o lner sets}, if for every finite $\Phi\subseteq \Gamma$ and every $\varepsilon>0$ there is a $(\Phi,\varepsilon)$-F\o lner set. If there is a sequence $(E_{n})_{n}$ so that for $\Phi\subseteq \Gamma$ and every $\varepsilon>0$ we have that $\{n: (E_{n})_{n} \textnormal{ is not $(\Phi,\varepsilon)$ F\o lner}\}$ is finite, then we call $(E_{n})_{n}$ a \emph{F\o lner sequence.}
\end{itemize}
\end{defn}

Some of our equivalent conditions defining relative amenability will deal with vectors in $L^{p}(\cR/\cS)$ which are fiberwise unit vectors, and require that such vectors are almost invariant under the relation. On the other hand, in order to apply diagonal arguments and core measure theory results it will be convenient to reduce checking this condition to a countable subgroup. Since we will have to do this several times, it will be convenient to isolate a result which handles this in all cases in question. 

\begin{thm}\label{thm: DCT argument}
Let $(X,\mu)$ be a standard probability space and $\cS\leq \cR$ be discrete, probability measure preserving equivalence relations on $(X,\mu)$. Let $\Gamma\leq [\cR]$ be countable with $\cR=\cR_{\Gamma,X}$ modulo null sets. Fix $p\in [1,\infty)$.
\begin{enumerate}[(i)]
\item Suppose that $\xi^{(n)}\in L^{p}(\cR/\cS)$ is a sequence with $\|\xi^{(n)}_{x}\|_{p}=1$ for almost every $x\in X$ and all $n\in \N$. If $\xi^{(n)}$ are almost  $\Gamma$-invariant vectors in $L^{p}(\cR/\cS)$, then there is a subsequence $\xi^{(n_{k})}\in L^{p}(\cR/\cS)$ so that $\|\xi^{(n_{k})}_{x}-\xi^{(n_{k})}_{y}\|_{p}\to_{k\to\infty}0$ for almost every $(x,y)\in \cR$. \label{item: easy AE convergence}
    \item Suppose $\xi^{(n)}\in L^{p}(\cR/\cS)$ is a sequence with $\|\xi^{(n)}_{x}\|_{p}=1$ for almost every $x\in X$ and all $n\in \N$, and with $\|\xi^{(n)}_{x}-\xi^{(n)}_{y}\|_{p}\to_{n\to\infty}0$ for almost every $(x,y)\in \cR$. Then $\|\lambda_{\cR/\cS}(\gamma)\xi^{(n)}-\xi^{(n)}\|_{p}\to_{n\to\infty}0$ for every $\gamma\in [\cR]$.  \label{item: easy DCT argument}
\end{enumerate}
\end{thm}

\begin{proof}

(\ref{item: easy AE convergence}):
Note that, for $\gamma\in \Gamma$,
\[\|\lambda_{\cR/\cS}(\gamma)\xi^{(n)}-\xi^{(n)}\|_{p}^{p}=\int \|\xi^{(n)}_{\gamma^{-1}(x)}-\xi^{(n)}_{x}\|_{\ell^{p}([x]_{\cR}/\cS)}^{p}\,d\mu(x).\]
Hence by countability of $\Gamma$, and a diagonal argument, we may find a subsequence (see e.g. \cite[Theorem 2.30]{Folland}) $(\xi^{(n_{k})})_{k}$ so that for almost every $x\in X$ we have $\|\xi^{(n_{k})}_{\gamma^{-1}(x)}-\xi^{(n_{k})}_{x}\|_{p}\to_{k\to\infty}0$ for every $\gamma\in \Gamma$. The result now follows from the fact that $\cR=\cR_{\Gamma,X}$ modulo null sets. 

(\ref{item: easy DCT argument}): As in (\ref{item: easy AE convergence}), for every $\gamma\in [\cR]$:
\[\|\lambda_{\cR/\cS}(\gamma)\xi^{(n)}-\xi^{(n)}\|_{p}^{p}=\int_{X}\|\xi^{(n)}_{\gamma^{-1}(x)}-\xi^{(n)}_{x}\|_{p}^{p}\,d\mu(x).\]
Since the integrand above is dominated by $2^{p}$, the proof is completed by applying the dominated convergence theorem. 
\end{proof}

Invariant means will play an important role in our proof of the main equivalence definitions of coamenability, particularly as useful limits of almost invariant vectors. The existence of such means essentially always uses some compactness argument. Recall that if $V$ is a Banach space, then $V^{*}$ has the \emph{weak$^{*}$ topology}. This topology can be described by saying that if $\phi\in V^{*}$ then the sets
\[\mathcal{O}_{F,\varepsilon}(\phi)=\bigcap_{v\in F}\{\psi\in V^{*}:|(\phi-\psi)(v)|<\varepsilon\},\]
ranging over finite $F\subseteq V$ and $\varepsilon>0$ form a neighborhood basis for $\phi$ in the weak$^{*}$-topology. By the Banach-Alaoglu theorem, we have that $\{\phi\in V^{*}:\|\phi\|\leq 1\}$ is weak$^{*}$ compact. For a $\sigma$-finite measure space $(Y,\nu)$, we let $\cM(Y,\nu)$ be the space of means on $(Y,\nu)$. Observe that $\cM(Y,\nu)$ is a weak$^{*}$-closed subset of $\{\phi\in L^{\infty}(Y,\nu)^{*}:\|\phi\|\leq 1\}$, thus $\cM(Y,\nu)$ is weak$^{*}$-compact.

The following is proved exactly as in the proof of the standard equivalent definitions of an amenable group (see \cite[Theorems G.3.1 and G.3.2]{KazhdanTDef}).

\begin{thm}\label{thm: folklore 1}
Let $(Y,\nu)$ be a $\sigma$-finite measure space, $\Gamma$  a group and $\Gamma\actson (Y,\nu)$ a probability measure-preserving action. 
\begin{enumerate}[(i)]
    \item Suppose $\xi^{(n)}\in L^{2}(Y,\nu)$ has $\|\xi^{(n)}\|_{2}=1$ and $\xi^{(n)}$ are a sequence of almost invariant unit vectors for the induced action on $L^{2}(Y,\nu)$. Set $f_{n}=|\xi^{(n)}|^{2}$. Then $\|\alpha(\gamma)f_{n}-f_{n}\|_{1}\to_{n\to\infty}0$ for all $\gamma\in \Gamma$.
    \item If $f_{n}\in L^{1}(Y,\nu)_{+}$  and $\|f_{n}\|_{1}=1$, and $\|\alpha(\gamma)f_{n}-f_{n}\|_{1}\to_{n\to\infty}0$ for all $\gamma\in\Gamma$, then $\xi^{(n)}=f_{n}^{1/2}$ are almost invariant unit vectors for the induced action on $L^{2}(Y,\nu)$.
    \item Suppose that $f_{n}\in L^{1}(Y,\nu)$ is a sequence with $f_{n}\geq 0$, and $\|f_{n}\|_{1}=1$, and $\|\alpha(\gamma)f_{n}-f_{n}\|_{1}\to_{n\to\infty}0$ for all $\gamma\in\Gamma$. Let $m_{n}\in L^{\infty}(Y,\nu)^{*}$ be given by $m_{n}(g)=\int f_{n}g\,d\nu$. If $m$ is any weak$^{*}$ cluster point of $m_{n}$, then $m$ is a $\Gamma$-invariant mean. 
\end{enumerate}
\end{thm}

At various points in time (in particular for investigating cospectral radius), it will also be helpful to rephrase existence of almost invariant vectors in terms of norms of natural bounded operators on Hilbert spaces. This once again follows exactly as in the proof of the standard equivalent definitions of amenability (see e.g. \cite[Theorem G.3.2, Proposition G.4.2, and Theorem G.5.1]{KazhdanTDef}). 

\begin{thm}\label{thm: folklore 2}
 Let $(Y,\nu)$ be a $\sigma$-finite measure space, and let $\Gamma\acston (Y,\nu)$ be a measure-preserving action of a countable group $\Gamma$. Suppose that $S\subseteq \Gamma$ generates $\Gamma$. Then the following are equivalent. 
 \begin{enumerate}[(i)]
     \item The induced action of $\Gamma$ on $L^{2}(Y,\nu)$ has almost invariant vectors,
     \item there is a symmetric $\eta\in \Prob(\Gamma)$ with $\ip{\supp(\eta)}=\Gamma$, and $\|\alpha(\eta)\|_{B(L^{2}(Y,\nu))}=1$.
     \item for every $\eta\in \Prob(\Gamma)$ with $\ip{\supp(\eta)}=\Gamma$ we have $\|\alpha(\eta)\|_{B(L^{2}(Y,\nu))}=1$,
     \item the action $\Gamma\acston (Y,\nu)$ admits a F\o lner sequence,
     \item there is a sequence $F_{n}\subseteq Y$ of measurable sets with $0<\nu(F_{n})<+\infty$ and 
     \[\frac{\nu(\gamma F_{n}\Delta F_{n})}{\nu(F_{n})}\to_{n\to\infty}0,\]
     for all $\gamma\in S$. 
 \end{enumerate}
\end{thm}

For technical reasons, for our definitions of coamenability we will need to check that our invariant means on $L^{\infty}(\cR/\cS)$ in question agree with the integral on $L^{\infty}(X,\mu)$ (recall that we may view $L^{\infty}(X,\mu)$ inside $L^{\infty}(\cR/\cS)$ via $f\mapsto ((x,c)\mapsto f(x))$). The following result greatly reduces our work, e.g. in the case that $\cR$ is ergodic it says that this condition follows immediately from being a $[\cR]$-invariant mean. 

\begin{thm}\label{ETS agrees with measure on center}
Let $(X,\mu)$ be a probability space and $\Gamma\acston (X,\mu)$ a probability measure-preserving action of a group $\Gamma$. If $m\in L^{\infty}(X,\mu)^{*}$ is a $\Gamma$-invariant mean so that $m(1_{E})=\mu(E)$ for all $\Gamma$-invariant measurable sets $E$, then $m(f)=\int f\,d\mu$ for all $f\in L^{\infty}(X,\mu)$. 
\end{thm}

\begin{proof}
Arguing exactly as in
\cite[Lemma 4.2]{HVTypeIIICartan}, we see that $m$ agrees with $\mu$ on the indicator function of  any measurable subset of $X$. Since $L^{\infty}(X)$ is densely spanned by indicator functions of measurable sets, and both $m$,$\int\cdot \,d\mu$ are continuous on $L^{\infty}(X)$, we see that $m=\int \cdot\,d\mu$.     
\end{proof}

\section{coamenability for relations}\label{sec: relative amenability relns}

Recall that if $\cS\leq \cR$ are discrete, probability measure-preserving equivalence relations on a standard probability space $(X,\mu)$, then for $\gamma\in [[\cR]]$ and $f\colon \cR/\cS\to \C$ measurable, we define $\lambda_{\cR/\cS}(\gamma)f$ by $(\lambda_{\cR/\cS}(\gamma)f)(x,c)=1_{\ran(\gamma)}(x)f(\gamma^{-1}(x),c)$. 
We also identify $L^{\infty}(X)$ as a subset of $L^{\infty}(\cR/\cS)$ by identifying $f$ with the function $((x,c)\mapsto f(x))$. 

\begin{defn}
Let $\cR$ be an discrete, probability measure-preserving equivalence relation on a standard probability space $(X,\mu)$, and $\cS\leq \cR$. A \emph{global invariant mean} is a positive operator $\Phi\colon L^{\infty}(\cR/\cS)\to L^{\infty}(X)$ such that $\Phi(f)=f$ for all $f\in L^{\infty}(X)$, and so that $\Phi(\lambda_{\cR/\cS}(\gamma)(f))=\lambda_{\cR}(\gamma)(\Phi(f))$ for all $f\in L^{\infty}(\cR/\cS)$ and $\gamma\in [\cR]$.
\end{defn}

Sometimes authors include the assumption that $\Phi(\lambda_{\cR/\cS}(\gamma)(f))=\lambda_{\cR}(\gamma)(\Phi(f))$ for all $f\in L^{\infty}(\cR/\cS)$ and $\gamma\in [[\cR]]$, but as we will show this is automatic. 

\begin{prop}\label{prop: sometimes people refuse to learn basic cp map theory}
Let $\Phi\colon L^{\infty}(\cR/\cS)\to L^{\infty}(X)$ be a global invariant mean.
Then for any $\gamma\in [[\cR]],$ and any $f\in L^{\infty}(\cR/\cS)$, we have that $\Phi(\lambda_{\cR/\cS}(\gamma)(f))=\lambda_{\cR}(\gamma)(\Phi(f))$.
\end{prop}

\begin{proof}
Since $\Phi$ is positive and $\Phi|_{L^{\infty}(X)}$ is the identity, it follows by \cite[Theorem 4]{StinespringCP} and \cite[Theorem 3.18]{PaulsenCB}
  that $\Phi(kf)=k\Phi(f)$ for all $f\in L^{\infty}(\cR/\cS)$, and any $k\in L^{\infty}(X)$. We may write $\gamma=\id_{E}\sigma$ where $\sigma\in [\cR]$ and $E=\ran(\gamma)$. By the above:
\[\Phi(\lambda_{\cR/\cS}(\gamma)(f))=\Phi(1_{E}\lambda_{\cR/\cS}(\sigma)(f))=1_{E}\lambda_{\cR}(\sigma)(\Phi(f))=\lambda_{\cR}(\gamma)(\Phi(f)).\]
\end{proof}






One of our conditions for coamenability will be the existence of a $[\cR]$-invariant mean which agrees with $\int \cdot\,d\mu$ on $L^{\infty}(X)$. To relate this to cospectral radius, it will be helpful to isolate the following result. 

\begin{prop}\label{prop: restricted co-spectral radius to invariant means}
Suppose that $\rho_{E}(\cR/\cS,\nu)=1$ for all positive measure $\cR$-invariant sets $E$, and all countably supported, symmetric $\nu\in \Prob([\cR])$, which generate $\cR$. Then, given $\cR$-invariant positive measure sets $E_{1},\cdots,E_{k}$, and a countable $\Gamma\leq [\cR]$, there is a $[\cR]$-invariant mean $m\in L^{\infty}(\cR/\cS)^{*}$ with $m(1_{E_{i}})=\mu(E_{i})$, $i=1,\cdots,k$. 
\end{prop}

\begin{proof}
Fix a countably supported, symmetric $\nu\in \Prob([\cR])$ which generates $\cR$ and has $\ip{\supp(\nu)}\supseteq \Gamma$.
Let $\cF$ be the sub-$\sigma$-algebra generated by $\{E_{i}\}_{i=1}^{k}$. Since $\{E_{i}\}_{i=1}^{k}$ is finite, we may find a partition $(A_{j})_{j=1}^{\ell}$ of $X$ so that 
\[\cF=\left\{\bigcup_{j\in D}A_{j}:D\subseteq [\ell]\right\}.\]
By (\ref{eqn: co spectral radius is norm restricted to subset background}),
\[\rho_{A_{j}}(\cR/\cS,\nu)=1=\|\rho^{\cS}_{\nu}1_{A_{j}}\|_{\infty},\]
for all $j=1,\cdots,\ell$. Since $A_{j}$ is $\cR$-invariant, we may view $\cR|_{A_{j}}/\cS|_{A_{j}}$ as an $[\cR]$-invariant subset of $\cR/\cS$, and $\|\rho^{\cS}_{\nu}1_{A_{j}}\|_{\infty}$ may be viewed as the cospectral radius of $\cS|_{A_{j}}\leq \cR|_{A_{j}}$ with respect to $\nu$. Since $\nu$ is generating, equation (\ref{eqn: co spec radi is norm background 1}) implies that
\[\|\rho^{\cS}_{\nu}1_{A_{j}}\|_{\infty}=\|\lambda_{\cR|_{A_{j}}/\cS|_{A_{j}}}(\nu)\|_{B(L^{2}(\cR|_{A_{j}}/\cS|_{A_{j}}))}.\]
Thus Theorems \ref{thm: folklore 1} and \ref{thm: folklore 2} imply that for every $j=1,\cdots,\ell$ there is a $\Gamma$-invariant mean $m_{j}\in L^{\infty}(\cR|_{A_{j}}/\cS|_{A_{j}})^{*}$. For $j=1,\cdots,\ell$, define $\pi_{j}\colon \cR/\cS\to \cR|_{A_{j}}/\cS|_{A_{j}}$ by $\pi_{j}(f)=f|_{\cR|_{A_{j}}/\cS|_{A_{j}}}.$ Define $m\in L^{\infty}(\cR/\cS)^{*}$ by
\[m=\sum_{j=1}^{\ell}\mu(A_{j})m_{j}\circ \pi_{j}.\]
Then the $\cR$-invariance of $A_{j}$ implies that $m$ is $\Gamma$-invariant and that $m(1_{A_{j}})=\mu(A_{j})$ for $j=1,\cdots,\ell$. Our choice of $(A_{j})_{j}$ then implies that $m(1_{E})=\mu(E)$ for every $E\in \cF$, in particular for every $E_{i}$, $i=1,\cdots,k$.

\end{proof}

As in the proof of existence of invariant means and Reiter's condition for invariant means, we will rely on the Hahn-Banach theorem and a convexity argument. For a Banach space $V$, we define the \emph{weak topology} on $V$ by saying that for $v\in V$ the sets
\[\mathcal{O}_{F,\varepsilon}=\bigcap_{\phi\in F}\{w\in V:|\phi(v-w)|<\varepsilon\}\]
ranging over finite $F\subseteq V^{*}$ and $\varepsilon>0$, form a neighborhood basis at $v$.
Note that if $p\in [1,+\infty]$ and if $V_{1},\cdots,V_{n}$ are Banach spaces and we equip $\bigoplus_{j=1}^{n}V_{j}$ with the norm $\|v\|=\left(\sum_{j=1}^{n}\|v_{j}\|_{p}\right)^{1/p}$, then the weak topology on $\bigoplus_{j=1}^{n}V_{j}$ is the product of the weak topology on each $V_{j}$. 
One of the main utilities of the weak topology is that, being weaker than the norm topology, it is easier to approximate vectors weakly than in norm. On the other hand, under convexity assumptions one can deduce the \emph{existence} of norm approximates from weak approximates. Namely, if $V$ is a Banach space and $C\subseteq V$ convex then a typical application of separating Hahn-Banach (see \cite[Theorem IV.3.13]{Conway} for a precise statement) is that $\overline{C}^{\|\cdot\|}=\overline{C}^{weak}$,
 see \cite[Theorem V.1.4]{Conway}.
Another application of separating Hahn-Banach is the following.
\begin{prop}[Proposition 0.1 of \cite{Paterson}] \label{prop: folklore means}
 Let $(Y,\nu)$ be a standard $\sigma$-finite space. Then the space of means in $L^{\infty}(Y,\nu)^{*}$ is the weak$^{*}$-closure of $\{f\in L^{1}(Y,\nu)_{+}:\|f\|_{1}=1\}$.  
\end{prop}

With this in hand, we can close the loop and prove the equivalence of all the results stated above.

\begin{thm}\label{thm: big TFAE thm}
Let $\cR$ be a discrete, probability measure-preserving equivalence on a standard probability space $(X,\mu)$, and let $\cS\leq \cR$. The following are equivalent.
\begin{enumerate}[(i)]
    \item there is a global invariant mean $\Phi\colon L^{\infty}(\cR/\cS)\to L^{\infty}(X,\mu),$ \label{item:global invariant mean}
    \item there is a $[\cR]$-invariant mean $m$ on $L^{\infty}(\cR/\cS)$ such that $m|_{L^{\infty}(X)}=\int \cdot d\mu$,
    \label{item:inv mean restricting well}
    \item  there is a sequence $\xi^{(n)}\in L^{2}(\cR/\cS)$ with $\|\xi^{(n)}_{x}\|_{2}=1$ for almost every $x\in X$, and $\|\xi^{(n)}_{x}-\xi^{(n)}_{y}\|_{2}\to_{n\to\infty}0$ for almost every $(x,y)\in \cR$, \label{item:almost invariant fiberwise norm 1 nonergodic part}
    \item there is a sequence $\mu^{(n)}\in L^{1}(\cR/\cS)$ so that $\mu^{(n)}_{x}\in \Prob([x]_{\cR}/\cS)$ for almost every $x\in X$ and $\|\mu^{(n)}_{x}-\mu^{(n)}_{y}\|_{1}\to_{n\to\infty}0$ for almost every $(x,y)\in \cR$, \label{item: Rieter sequence fiberwise}
    \item there is a countable $\Gamma\leq [\cR]$ such $\Gamma x=[x]_{\cR}$ for almost every $x\in X$, and a sequence $\mu^{(n)}\in L^{1}(\cR/\cS)$ with $\mu^{(n)}_{x}\in \Prob([x]_{\cR}/\cS)$ for almost every $x\in X$ and so that $\mu^{(n)}$ is $\Gamma$-almost invariant, \label{item: Rieter sequence on subgroup}
    \item there is a countable $\Gamma\leq [\cR]$ such that $\Gamma x=[x]_{\cR}$ for almost every $x\in X$, and a sequence $\xi^{(n)}\in L^{2}(\cR/\cS)$ so that $\|\xi^{(n)}_{x}\|_{2}=1$ for almost every $x\in X$, and $\xi^{(n)}$ is $\Gamma$-almost invariant, \label{item:almost invariant fiberwise nonergodic intro part 2}  
     \item $\lambda_{\cR|_{E}/\cS|_{E}}$ has almost invariant vectors for all $\cR$-invariant measurable $E\subseteq X$ of positive measure, \label{item: almost invariant vectors nonergodic part inv sets}
     \item for every $\cR$-invariant measurable set $E\subseteq X$ with positive measure, the action $[\cR]\actons (\cR|_{E}/\cS|_{E},\eta_{\cR|_{E}/\cS_{E}})$ admits F\o lner sets, where $\eta=\frac{\mu|_{E}}{\mu(E)}$, \label{item: Folner sets in invariant sets}
     \item for every countably supported, symmetric $\nu\in \Prob([\cR])$ we have $\|\lambda_{\cR|_{E}/\cS|_{E}}(\nu)\|=1$ for every $\cR$-invariant set $E$ of positive measure, \label{item: rel amen in terms of Markov operator}
    \item for every countably supported, symmetric $\nu\in \Prob([\cR])$ which generates $\cR$, we have $\rho_{E}(\cR/\cS,\nu)=1$ for every $\cR$-invariant set $E$ of positive measure. 
    \label{item: co-spectral radius condition nonergodic part}
\end{enumerate}

\end{thm}

\begin{proof}

The equivalence of (\ref{item: Rieter sequence fiberwise}) and (\ref{item:almost invariant fiberwise norm 1 nonergodic part}) follows from Theorem \ref{thm: folklore 1}, the equivalence of (\ref{item: Rieter sequence on subgroup}) and (\ref{item:almost invariant fiberwise nonergodic intro part 2}) also follows from Theorem \ref{thm: folklore 1}. 

To see that  (\ref{item: almost invariant vectors nonergodic part inv sets}) and  (\ref{item: Folner sets in invariant sets}) are equivalent, note that (\ref{item: almost invariant vectors nonergodic part inv sets}) is equivalent to saying that $\lambda_{\cR|_{E}/\cS|_{E}}$ has $\Gamma$-invariant vectors for every finitely generated $\Gamma\leq [\cR]$ and every $\cR$-invariant set $E\subseteq X$ of positive measure, and (\ref{item: Folner sets in invariant sets}) is equivalent to saying that $\Gamma\actson (\cR|_{E}/\cS_{E},\eta_{\cR|_{E}/\cS|_{E}})$ has F\o lner sets for every finitely generated $\Gamma\leq [\cR]$ and every $\cR$-invariant set $E\subseteq X$ of positive measure. So the equivalence of (\ref{item: almost invariant vectors nonergodic part inv sets}) and  (\ref{item: Folner sets in invariant sets})
follows from Theorem \ref{thm: folklore 2}. 

 (\ref{item: almost invariant vectors nonergodic part inv sets}) implies (\ref{item: rel amen in terms of Markov operator}): Set $\Lambda=\ip{\supp(\nu)}$. Then $\lambda_{\cR|_{E}/\cS|_{E}}|_{\Lambda}$ has almost invariant vectors, so Theorem \ref{thm: folklore 2} implies that $\|\lambda_{\cR|_{E}/\cS_{E}}(\nu)\|=1$. 

 (\ref{item: rel amen in terms of Markov operator}) implies (\ref{item: almost invariant vectors nonergodic part inv sets}): Given a finite $S\subseteq [\cR]$, let $\nu$ be the uniform measure on $S$. Then Theorem \ref{thm: folklore 2} and $\|\lambda_{\cR|_{E}/\cS|_{E}}(\nu)\|=1$ tell us that $\lambda_{\cR|_{E}/\cS|_{E}}|_{\ip{S}}$ has almost invariant vectors. So given $\varepsilon>0$, we may find $\xi\in L^{2}(\cR|_{E}/\cS|_{E})$ with $\|\xi\|_{2}=1$ and 
\[\|\lambda_{\cR|_{E}/\cS|_{E}}(\gamma)\xi-\xi\|_{2}<\varepsilon, \textnormal{ for all $\gamma\in S$.}\]

(\ref{item:global invariant mean}) implies (\ref{item:inv mean restricting well}): 
Set $m(f)=\int \Phi(f)\,d\mu$. The $[\cR]$-invariance of $\Phi$ implies the $[\cR]$-invariance of $m$, and $m$ restricts to $\int\cdot\,d\mu$ on $L^{\infty}(X)$ by design.

(\ref{item:inv mean restricting well}) implies (\ref{item:global invariant mean}): For $k\in L^{\infty}(X,\mu)$ we have by positivity of $m$ that:
\[|m(kf)|\leq m(|fk|)\leq \|f\|_{\infty}m(|k|)=\|f\|_{\infty}\|k\|_{1}.\]
Since $\overline{L^{\infty}(X,\mu)}^{\|\cdot\|_{1}}=L^{1}(X,\mu)$, the above estimate implies that there is a unique $L_{f}\in L^{1}(X,\mu)^{*}$ with $\|L_{f}\|\leq \|f\|_{\infty}$ with $L_{f}(k)=m(fk)$ for all $k\in L^{\infty}(X,\mu)$. 
Thus there is a $\Phi(f)\in L^{\infty}(X,\mu)$ with  $m(fk)=\int \Phi(f)k\,d\mu$ for all $k\in L^{1}(X,\mu)$. By positivity of $m$ we have that $\Phi$ is positive. The $[\cR]$-invariance of $m$ proves that $\Phi$ is $[\cR]$-invariant, and since $m|_{L^{\infty}(X)}=\int \cdot\,d\mu$, it follows that $\Phi|_{L^{\infty}(X)}=\id$. 

(\ref{item:inv mean restricting well}) implies (\ref{item: Rieter sequence on subgroup}): Fix a countable $\Gamma\leq [\cR]$ so that $\Gamma x=[x]_{\cR}$ for almost every $x\in X$. We first prove the following

\emph{Claim: for every finite $F\subseteq [\cR]$,}
\[0\in \overline{\{(\|f_{x}\|-1,(\lambda_{\cR/\cS}(\gamma)f-f)_{\gamma\in F}):f\in L^{1}(\cR/\cS)_{+}, \textnormal{ for almost every $x\in X$}\}}^{\|\cdot\|}\]\[\subseteq L^{1}(X)\oplus L^{1}(\cR/\cS)^{\oplus F}.\]

Set $K=\{(\|f_{x}\|-1,(\lambda_{\cR/\cS}(\gamma)f-f)_{\gamma\in F}):f\in L^{1}(\cR/\cS)_{+}, \textnormal{ for almost every $x\in X$}\}$ then $K$ is convex so to prove the claim it suffices by separating Hahn-Banach to show that 
\[0\in \overline{K}^{wk},\]
Let $\varepsilon>0$, $k_{1},\cdots,k_{n}\in L^{\infty}(\cR/\cS)$ and $h_{1},\cdots,h_{s}\in L^{\infty}(X)$. By Proposition \ref{prop: folklore means}, we may choose a $f\in L^{1}(\cR/\cS)_{+}$ with $\|f\|_{1}=1$ and so that 
\[\left|m(\lambda_{\cR/\cS}(\gamma)^{-1}k_{j})-\int (\lambda_{\cR/\cS}(\gamma)^{-1}k_{j})f\,d\mu_{\cR/\cS}\right|<\varepsilon \textnormal{ for all $j=1,\cdots,n$ and $\gamma\in F$},\]
\[\left|m(h_{j})-\int h_{j}f\,d\mu_{\cR/\cS}(x)\right|<\varepsilon \textnormal{ for all $j=1,\cdots,s$}.\]
Since $f\geq 0$, our hypothesis implies that we can rewrite the last condition as 
\begin{equation}\label{item:basically done 1}
\left|\int h_{j}(x)(1-\|f_{x}\|_{1})\,d\mu(x)\right|<\varepsilon \textnormal{ for all $j=1,\cdots,s$}, 
\end{equation}
whereas the first condition and $[\cR]$-invariance of $m$ implies that
\begin{equation}\label{item: basically done 2}
   \left|\int (\lambda_{\cR/\cS}(\gamma)f-f)k_{j}\,d\mu_{\cR/\cS}\right|<2\varepsilon \textnormal{ for all $j=1,\cdots,n$ and $\gamma\in F$}. 
\end{equation}
Since the weak topology on $L^{1}(X)\oplus L^{1}(\cR/\cS)^{\oplus F}$ is the product of the weak topologies, (\ref{item:basically done 1}), (\ref{item: basically done 2}) prove that $0\in\overline{K}^{weak}$ which proves the claim by separating Hahn-Banach.

Having proved the claim, given a finite $F\subseteq \Gamma$, we may find a sequence $f^{(n)}\in L^{1}(\cR/\cS)$ so that $f^{(n)}\geq 0$, and $\|\lambda_{\cR/\cS}(\gamma)f^{(n)}-f^{(n)}\|_{1}\to_{n\to\infty}0$ for all $\gamma\in F$, and $\|\|f^{(n)}_{x}\|-1\|_{L^{1}(X)}\to_{n\to\infty}$. A perturbation argument shows that there is a $\mu^{(n)}\in L^{2}(\cR/\cS)$ with $\|\mu^{(n)}_{x}\|_{1}=1$ almost surely and $\|\mu^{(n)}-f^{(n)}\|_{1}\to_{n\to\infty}0$. Thus $\|\lambda_{\cR/\cS}(\gamma)\mu^{(n)}-\mu^{(n)}\|_{1}\to_{n\to\infty}0$ for all $\gamma\in F$. So
\[0\in \overline{\{(\lambda_{\cR/\cS}(\gamma)\mu-\mu)_{\gamma\in F}:\mu\in L^{1}(\cR/\cS)_{+},\|\mu_{x}\|_{1}=1 \textnormal{ for almost every $x\in X$}\}}^{\cdot}\subseteq L^{1}(\cR/\cS)^{\oplus F}.\]

(\ref{item:almost invariant fiberwise nonergodic intro part 2}) implies (\ref{item:almost invariant fiberwise norm 1 nonergodic part}):  This follows from Theorem \ref{thm: DCT argument}.

(\ref{item:almost invariant fiberwise norm 1 nonergodic part}) implies (\ref{item: almost invariant vectors nonergodic part inv sets}): Since $E$ is $\R$-invariant, we may view $\cR|_{E}/\cS|_{E}$ as a Borel, $[\cR]$-invariant subset of $\cR/\cS$. Set $\zeta^{(n)}=\xi^{n}|_{\cR|_{E}/\cS|_{E}}$. So $\|\zeta^{(n)}\|_{L^{2}(\cR|_{E}/\cS|_{E},\eta_{\cR|_{E}/\cS|_{E}})}=1$, where $\eta=\frac{\mu|_{E}}{\mu(E)}$. By $\cR$-invariance of $E$ and Theorem \ref{thm: DCT argument}, we have that 
\[\|\lambda_{\cR/\cS}(\gamma)\zeta^{(n)}-\zeta^{(n)}\|_{2}\leq \mu(E)^{-1}\|\lambda_{\cR/\cS}(\gamma)\xi^{(n)}-\xi^{(n)}\|_{2}\to_{n\to\infty}0\]
for all $\gamma\in [\cR]$.


(\ref{item: rel amen in terms of Markov operator}) implies (\ref{item: co-spectral radius condition nonergodic part}): Since $\nu$ is generating, we  may argue as in Proposition \ref{prop: restricted co-spectral radius to invariant means} to see that  $\rho_{E}(\cR/\cS,\nu)=\|\lambda_{\cR|_{E}/\cS|_{E}}(\nu)\|_{B(L^{2}(\cR|_{E}/\cS|_{E}))}$,  which implies the result.

(\ref{item: co-spectral radius condition nonergodic part}) implies (\ref{item:inv mean restricting well}): This is similar to the argument in \cite[Lemma 2.2]{BHAExtension}. Fix a countable $\Lambda\leq [\cR]$ which generates $\cR$. For a countable $\Gamma\leq [\cR]$ which contains $\Lambda$, and a finite tuple $E=(E_{1},\cdots,E_{k})$ of $\cR$-invariant measurable sets, let $\cM_{\Gamma,E}$ be the set of means $m\in L^{\infty}(\cR/\cS)^{*}$ which are $\Gamma$-invariant, and which satisfy $m(1_{E_{j}})=\mu(E_{j})$ for all $j=1,\cdots,k$. Each $\cM_{\Gamma,E}$ is weak$^{*}$ compact, and Proposition \ref{prop: restricted co-spectral radius to invariant means} implies that $\cM_{\Gamma,E}\ne\varnothing$ for every $\Gamma,E$. 
Since the space of means is weak$^{*}$-compact, we have that $\bigcap_{\Gamma,E}\cM_{\Gamma,E}\ne \varnothing$. Thus there is a $[\cR]$-invariant mean $m\in L^{\infty}(\cR/\cS)^{*}$ with $m(1_{E})=\mu(E)$ for every $\cR$-invariant measurable set $E$. Theorem \ref{ETS agrees with measure on center} implies that $m|_{L^{\infty}(X,\mu)}=\int \cdot\,d\mu$.

\end{proof}

\begin{defn}
We say that $\cS$ is  \textbf{coamenable in $\cR$}  if any of the above equivalent conditions hold.    
\end{defn}



\section{coamenability and almost sure cospectral radius}\label{sec: pointwise cospectral radius}
Let $\cR$ be a discrete, probability measure-preserving equivalence on a standard probability space $(X,\mu)$, and let $\cS\leq \cR$. In this section, we investigate how coamenability relates to the pointwise defined cospectral radius. 
In general, if $\cS$ is coamenable in $\cR$, then it does not follow that $\rho^{\cS}_{\nu}=1$ almost everywhere.
\begin{example}\label{example: rel amen comes from small pieces}
Let $(X,\mu)$ be a standard probability space and let $\Gamma\actson (X,\mu)$ be a free, probability measure-preserving action (e.g. the Bernoulli action). Suppose that $\Gamma$ is a nonamenable group. Let $\cR$ be the orbit equivalence relation of the $\Gamma$ action, and fix a measurable $E\subseteq X$ with $0<\mu(E)<1$. Let $\cS$ be the subrelation of $\cR$ where $(x,y)\in \cS$ if and only if either
\begin{itemize}
    \item $(x,y)\in \cR\cap (E\times E)$,
    \item $(x,y)\in (E^{c}\times E^{c})$ and $y=x$.
\end{itemize}

\end{example}

In this example, following \cite[Example 6]{AFH} we have that $\rho^{\cS}_{\nu}=1$ almost everywhere on $E$, for every $\nu\in \Prob([\cR])$ which is finitely supported and symmetric. If $\theta\in \Prob(\Gamma)$ is symmetric with $\ip{\supp(\theta)}=\Gamma$, let $\nu_{\theta}$ be the pushforward of $\theta$ under the map $\Gamma\mapsto[\cR]$ given by $g\mapsto (x\mapsto gx)$. Then on $E^{c}$, $\rho^{\cS}_{\nu_{\theta}}=\|\lambda(\theta)\|_{B(\ell^{2}(\Gamma))}<1$ , by nonamenability. Hence, by Theorem \ref{thm: big TFAE thm}, we have that $\cS$ is coamenable in $\cR$, but $\|\rho^{\cS}_{\nu_{\theta}}1_{E^{c}}\|_{\infty}<1$. 

In fact, in this example we have that $\cS$ is coamenable in $\cR$, but  $\cS|_{E^{c}}$ is not coamenable in $\cR|_{E^{c}}$ (as we will show in Proposition \ref{prop: group case is an iff}). 
In order for coamenability to relate more tightly with the pointwise cospectral radius, we make the following definition. 

\begin{defn}
Let $(X,\mu)$ be a standard probability space and $\cS\leq \cR$ discrete, probability measure-preserving equivalence relations. We say that $\cS$ is \textbf{ everywhere coamenable in $\cR$} if for all measurable sets $E\subseteq X$ of positive measure we have that $\cS|_{E}$ is coamenable in $\cR|_{E}$.
\end{defn}

In this section, we will show that everywhere coamenability will imply that the cospectral radius is $1$ almost surely, as well as investigate other general implications for cospectral radii. 
A key tool will be operator theoretic reformulations of the essential supremum of the cospectral radius restricted to a subset. 

Since we will frequently have to compare equivalence relations restricted to subsets, it will be helpful to introduce the following notation.
Suppose that $\cS\leq \cR$ are discrete, probability measure-preserving equivalence relations on a standard probability space and that $E,F\subseteq X$ are measurable with $\mu(E\setminus F)=0$. Then for almost every $x\in E$, there is a unique well-defined map $\iota_{x}\colon [x]_{\cR|_{E}}/\cS|_{E}\to [x]_{\cR|_{F}}/\cS|_{F}$ so that $\iota_{x}([y]_{\cS|_{E}})=[y]_{\cS|_{F}}$ for every $y\in [x]_{\cR}\cap E$. Note that $\iota_{x}$ is injective. This depends upon $F$, but we will typically suppress it from the notation.

It will also be helpful for us to identify $\rho_{E}(\cR/\cS,\nu)$ with an operator norm and this requires us to use an action of the space of $\cS$-ergodic components on $L^{2}(\cR/\cS)$. Suppose that $f\in L^{\infty}(X)$ is  almost everywhere $\cS$-invariant, and choose a Borel $f_{0}\colon X\to \C$ with $f_{0}=f$ almost everywhere and with $f_{0}$ being $\cS$-invariant. For $\xi\in L^{2}(\cR/\cS)$, we define $\xi f_{0}$ by saying that $(\xi f_{0})(x,c)=\xi(x,c)f_{0}(y)$ where $y$ is any element of $c$. Since $f_{0}$ is $\cS$-invariant, this is independent of the choice of $y$. It is direct to check that $\xi f_{0}=\xi f_{1}$ $\mu_{\cR/\cS}$-almost everywhere if $f_{1}\colon X\to \C$ is Borel and has $f_{0}=f$ $\mu$-almost everywhere. Hence we can make sense of $\xi f$ as an element of $L^{2}(\cR/\cS)$ (i.e. $\xi f$ is well-defined modulo almost everywhere equality),  by saying that $\xi f$ is the equivalence class of $\xi f_{0}$ modulo almost everywhere equality, where $f_{0}\colon X\to \C$ is any Borel $\cS$-invariant function with $f_{0}=f$ almost everywhere. 
We let $R(f)\in B(L^{2}(\cR/\cS))$ be given by $R(f)\xi=\xi f$.
Note that for any $\gamma\in [\cR]$, we have that $(\xi f)(x,[\gamma(x)]_{\cS})=\xi(x,[\gamma(x)]_{\cS})f(\gamma(x))$ almost everywhere, so that $\xi f$ is measurable by Proposition \ref{prop:measurability exercise}.

The following allows us to reduce our considerations of everywhere coamenability to the case that $E$ is $\cS$-invariant.

\begin{prop}\label{prop: passing rel amen between subsets}
Let $(X,\mu)$ be a standard probability space and $\cS\leq \cR$ be discrete, probability measure-preserving equivalence relations on $(X,\mu)$. Suppose that $E\subseteq X$ is measurable with $\mu(E)>0$.
\begin{enumerate}[(i)]
    \item Suppose that $\cS|_{E}$ is coamenable in $\cR|_{E}$, and that $F\subseteq X$ is measurable with $\mu(E\setminus F)=0=\mu(F\setminus \cR E)$. Then
    there is a sequence $\widetilde{\xi}^{(n)}\in L^{2}(\cR|_{F}/\cS|_{F})$ such that
    \begin{itemize}
        \item $\|\widetilde{\xi}^{(n)}_{x}\|_{2}=1$ for almost every $x\in F$,
        \item $\lim_{n\to\infty}\|\widetilde{\xi}^{(n)}_{x}-\widetilde{\xi}^{(n)}_{y}\|_{2}=0$ for almost every $(x,y)\in \cR|_{F}$,
        \item $R(1_{\cS E})\widetilde{\xi}^{(n)}=\widetilde{\xi}^{(n)}$.
    \end{itemize} 
    In particular, $\cS|_{F}$ is coamenable in $\cR|_{F}$. \label{item: rel amen ascent via saturation}
    \item If $\cS|_{\cS E}$ is coamenable in $\cR|_{\cS E}$, then  $\cS|_{E}$ is coamenable in $\cR|_{E}$.  \label{item: rel amen descent via saturation}
\end{enumerate} 
\end{prop}

\begin{proof}
We start with a bit of setup. 
For every $x\in E$, we have for every $\widetilde{y}\in [x]_{\cR}\cap \cS E$ there is a $y\in [x]_{\cR}\cap E$ with $[y]_{\cS}=[\widetilde{y}]_{\cS}$.  Thus there is a well-defined map $j_{x}\colon [x]_{\cR|_{\cS E}}/\cS|_{\cS E}\to[x]_{\cR|_{E}}/\cS|_{E}$ such that $j_{x}([\widetilde{y}]_{\cS|_{\cS E}})=[y]_{\cS|_{E}}$ where $y\in [\widetilde{y}]_{\cS}\cap  E$. In the case that $F=\cS E$, a direct computation shows that $j_{x},\iota_{x}$ are almost surely inverse to each other (in particular both are bijective).

(\ref{item: rel amen descent via saturation}): We can find a sequence $\widetilde{\xi}^{(n)}\in L^{2}(\cR|_{\cS E}/\cS|_{\cS E})$ such that for almost every $x\in \cS E$ we have $\|\widetilde{\xi}^{(n)}_{x}\|_{2}=1$
\[\lim_{n\to\infty}\|\widetilde{\xi}^{(n)}_{y}-\widetilde{\xi}^{(n)}_{z}\|_{2}=0\]
for all $y,z\in [x]_{\cR}\cap \cS E$. Define $\xi^{(n)}\in L^{2}(\cR|_{E}/\cS|_{E})$ by $\xi^{(n)}(x,c)=\widetilde{\xi}^{(n)}(x,\iota_{x}(c))$. By bijectivity of $\iota_{x}$, it follows that  for almost every $x\in E$ we have $\|\xi^{(n)}_{x}\|_{2}=1,$ and
\[\lim_{n\to\infty}\|\xi^{(n)}_{y}-\xi^{(n)}_{z}\|_{2}=\lim_{n\to\infty}\|\widetilde{\xi}^{(n)}_{y}-\widetilde{\xi}^{(n)}_{z}\|_{2}=0\]
for all $y,z\in [x]_{\cR}\cap E$. Thus $\cS|_{E}$ is coamenable in $\cR|_{E}$.

(\ref{item: rel amen ascent via saturation}): We may find a countable set $\{\phi_{i}\}_{i\in I}\subseteq [[\cR]]$ with $\dom(\phi_{i})\subseteq E$ such that 
\[F= \bigsqcup_{i}\ran(\phi_{i}),\]
modulo null sets. Let $\xi^{(n)}\in L^{2}(\cR|_{E}/\cS|_{E})$ such that for almost every $x\in E$ we have $\|\xi^{(n)}_{x}\|_{2}=1$
\[\lim_{n\to\infty}\|\xi^{(n)}_{y}-\xi^{(n)}_{z}\|_{2}=0\]
for all $y,z\in [x]_{\cR}\cap E$. Define $\widetilde{\xi}^{(n)}\in L^{2}(\cR|_{F}/\cS|_{F})$ by
\[\widetilde{\xi}^{(n)}(y,c)=\xi^{(n)}(\phi_{i}^{-1}(y),\iota_{\phi_{i}^{-1}(y)}^{-1}(c)),\]
if $y\in \ran(\phi_{i})$, $c\in \Im(\iota_{\phi_{i}^{-1}(y)})$, otherwise $\widetilde{\xi}^{(n)}(y,c)=0$. By injectivity of $\iota_{y}$, we have that $\|\widetilde{\xi}^{(n)}_{y}\|_{2}=1$ for almost every $y\in F$, and for almost every $x\in F$ we have that for all $y,z\in [x]_{\cR|_{F}}$
\[\|\widetilde{\xi}^{(n)}_{y}-\widetilde{\xi}^{(n)}_{z}\|_{2}=\|\xi^{(n)}_{\phi_{i}^{-1}(y)}-\xi^{(n)}_{\phi_{j}^{-1}(z)}\|_{2},\]
where $y\in \ran(\phi_{i}),z\in \ran(\phi_{j})$. Thus
\[\lim_{n\to\infty}\|\widetilde{\xi}^{(n)}_{y}-\widetilde{\xi}^{(n)}_{z}\|_{2}=0\]
for almost every $(x,z)\in \cR|_{F}$. So $\cS|_{F}$ is coamenable in $\cR|_{F}$.

\end{proof}

\begin{lem}\label{lem: restricted co-spec radi is just restricted op}
Let $\cS\leq \cR$ be discrete, probability measure-preserving equivalence relations on a standard probability space $(X,\mu)$. Suppose that $\nu\in \Prob([\cR])$ is symmetric and generates $\cR$. Then for any $\cS$-invariant, measurable $E\subseteq X$ with positive measure:
\[\lim_{k\to\infty}\|R(1_{E})(\lambda_{\cR/\cS}(\id_{E})\lambda_{\cR/\cS}(\nu)^{2k}\lambda_{\cR/\cS}(\id_{E}))\|^{1/2k}=\rho_{E}(\cR/\cS,\nu)=\|R(1_{E})\lambda_{\cR/\cS}(\nu)\|.\]
    
\end{lem}

\begin{proof}
Let $\Gamma=\ip{\supp(\nu)}$. We adopt notation as in the discussion surrounding (\ref{eqn: co spectral radius is norm restricted to subset background}).
Since $\lambda_{\cR/\cS}(\id_{E})\zeta_{E}=\zeta_{E}$, and $R(1_{E})\zeta_{E}=\zeta_{E}$, and $R(1_{E})$ and $\lambda_{\cR/\cS}([[\cR]])$ commute, we  obtain that 
\[\rho_{E}(\cR/\cS,\nu)=\lim_{k\to\infty}\ip{\lambda_{\cR/\cS}(\nu)^{2k}\zeta_{E},\zeta_{E}}^{1/2k}\leq\liminf_{k\to\infty}\|R(1_{E})\lambda_{\cR/\cS}(\id_{E})\lambda_{\cR/\cS}(\nu)^{2k}\lambda_{\cR/\cS}(\id_{E})\|^{1/2k}.\]
Additionally, since $R(1_{E})$ commutes with $\lambda_{\cR/\cS}([[\cR]])$:
\begin{align*}
    \|R(1_{E})(\lambda_{\cR/\cS}(\id_{E})\lambda_{\cR/\cS}(\nu)^{2k}\lambda_{\cR/\cS}(\id_{E}))\|=\|\lambda_{\cR/\cS}(\id_{E})R(1_{E})\lambda_{\cR/\cS}(\nu)^{2k}\lambda_{\cR/\cS}(\id_{E})\|&\leq \|R(1_{E})\lambda_{\cR/\cS}(\nu)^{2k}\|\\
    &=\|(R(1_{E})\lambda_{\cR/\cS}(\nu))^{2k}\|\\
    &\leq \|R(1_{E})\lambda_{\cR/\cS}(\nu)\|^{2k}.
\end{align*}
Thus it suffices to show that $\|R(1_{E})\lambda_{\cR/\cS}(\nu)\|\leq \rho_{E}(\cR/\cS,\nu).$
Define $\xi\in L^{2}(\cR/\cS)$ by $\xi(x,c)=1_{x\in c}$. 
Note that for any measurable $F\subseteq X$, and any $\gamma\in \Gamma$, we have 
\[\ip{R(1_{E})\lambda_{\cR/\cS}(\nu))^{2n}1_{F}\lambda_{\cR/\cS}(\gamma)\xi,1_{F}\lambda_{\cR/\cS}(\gamma)\xi}^{1/2n}\leq \ip{(R(1_{E})\lambda_{\cR/\cS}(\nu))^{2n}\lambda_{\cR/\cS}(\gamma)\xi,\lambda_{\cR/\cS}(\gamma)\xi}^{1/2n},\]
where in the above we (as usual) identify $L^{\infty}(X)\subseteq L^{\infty}(\cR/\cS)$ via $f\mapsto ((x,c)\mapsto f(x))$. 
Moreover, by \cite[Lemma 3.4]{AFH}, we have that 
\[L^{2}(\cR/\cS)=\overline{\Span\{1_{F}\lambda_{\cR/\cS}(\gamma)\xi:\gamma\in \Gamma, F\subseteq X \text{ is measurable}\}}.\]
Thus, the same argument as in the proof of \cite[Theorem 3.2(i)]{AFH} (using \cite[Lemma 3.4]{AFH} , see also \cite[Equation 2.8]{HutchcroftIsing}) shows that
\begin{align*}
\|R(1_{E})\lambda_{\cR/\cS}(\nu)\|&=\sup_{\gamma\in \Gamma}\lim_{n\to\infty}\ip{(R(1_{E})\lambda_{\cR/\cS}(\nu))^{2n}\lambda_{\cR/\cS}(\gamma)\xi,\lambda_{\cR/\cS}(\gamma)\xi}^{1/2n}\\
&=\sup_{\gamma\in \Gamma}\lim_{n\to\infty}\ip{R(1_{E})\lambda_{\cR/\cS}(\delta_{\gamma^{-1}}*\nu*\delta_{\gamma})^{2n}\xi,\xi}^{1/2n}.
\end{align*}
 Let $\gamma^{*}(\cS)$ be the subrelation $\{(\gamma(x),\gamma(y)):(x,y)\in \cS\}$. Then 
\[\ip{R(1_{E})\lambda_{\cR/\cS}(\delta_{\gamma^{-1}}*\nu*\delta_{\gamma})^{2n}\xi,\xi}=\int_{E}p_{2n,\gamma(x),\gamma^{*}(\cS)}\,d\mu(x).\]
Since $\Gamma=\ip{\supp(\nu)}$, we can find $k>0$ such that $\nu^{*k}(\{\gamma\})>0$. Set $c=\nu^{*k}(\{\gamma\})$. As in the proof of Theorem \cite[Theorem 3.2(i)]{AFH} we have that 
\[p_{2n,\gamma(x),\gamma^{*}(\cS)}\geq c^{2}p_{2(n-k),x,\cS}.\]
This shows that 
\[\lim_{n\to\infty}\ip{R(1_{E})\lambda_{\cR/\cS}(\delta_{\gamma^{-1}}*\nu*\delta_{\gamma})^{2n}\xi,\xi}^{1/2n}\leq \lim_{n\to\infty}(\mu(E)\ip{\lambda_{\cR/\cS}(\nu)^{2(n-k)}\zeta_{E},\zeta_{E}})^{1/2n}=\rho_{E}(\cR/\cS,\nu).\]
    
\end{proof}

Note that $R(1_{E})$ is an orthogonal projection, and since $R(1_{E})$ commutes with $\lambda_{\cR/\cS}$, the subspace $R(1_{E})L^{2}(\cR/\cS)$ is $\lambda_{\cR/\cS}$-invariant. Thus we can view $\|R(1_{E})\lambda_{\cR/\cS}(\nu)\|$ as simply the norm of $\lambda_{\cR/\cS}(\nu)|_{R(1_{E})L^{2}(\cR/\cS)}$.

We have the following general application to equality of cospectral radii.

\begin{prop}\label{prop: co spectral radius descent}
Let $\cS_{1}\leq \cS_{2}\leq \cR$ be discrete, probability measure-preserving equivalence relations on a standard probability space $(X,\mu)$. Suppose that $E\subseteq X$ is a measurable, $\cS_{1}$-invariant set with $\mu(E)>0$. If $\cS_{1}|_{E}$ is coamenable in $\cS_{2}|_{E}$, then for any generating, symmetric, and countably supported $\nu\in \Prob([\cR])$ we have $\rho_{E}(\cR/\cS_{2},\nu)=\rho_{E}(\cR/\cS_{1},\nu)$.  
\end{prop}

\begin{proof}
We always have that $\rho_{E}(\cR/\cS_{1},\nu)\leq \rho_{E}(\cR/\cS_{2},\nu)$.  By $\cS_{2}$-invariance of $\rho^{\cS_{2}}_{\nu}$, we have that $\rho_{E}(\cR/\cS_{2},\nu)=\rho_{\cS_{2}E}(\cR/\cS_{2},\nu)$. Hence, by Lemma \ref{lem: restricted co-spec radi is just restricted op}, it suffices to show that
\[\|R(1_{\cS_{2}E})\lambda_{\cR/\cS_{2}}(\nu)\|\leq \|R(1_{E})\lambda_{\cR/\cS_{1}}(\nu)\|.\]

We need a bit of setup. Since $\cS_{1}|_{E}$ is coamenable in $\cS_{2}|_{E}$, by Proposition \ref{prop: passing rel amen between subsets} we may find a sequence $\xi^{(n)}\in L^{2}(\cS_{2}|_{\cS_{2}E}/\cS_{1}|_{\cS_{2}E})$ with $\|\xi^{(n)}_{x}\|_{2}=1$ for almost every $x\in \cS_{2} E$, and with $\|\xi^{(n)}_{x}-\xi^{(n)}_{y}\|_{2}\to_{n\to\infty}0$ for almost every $(x,y)\in \cS_{2}|_{\cS_{2}E}$, and $R(1_{E})\xi^{(n)}=\xi^{(n)}$. Apply Corollary \ref{cor:choice functions} to find a countable family $\{\sigma_{j}\}_{j\in J}$ of choice functions for $\cS_{2}|_{\cS_{2}E}\leq \cR|_{\cS_{2}E}$. 
Choose Borel $\{\gamma_{i}\}_{i\in I}$ in $[[\cR]]$ so that 
\[\cR E=\bigsqcup_{i\in I} \ran(\gamma_{i}),\]
modulo null sets and $\dom(\gamma_{i})\subseteq \cS_{2} E$ for all $i\in I$. 
For almost every $x\in X$ we have a well-defined surjection $\pi_{x}\colon [x]_{\cR}/\cS_{1}\to [x]_{\cR}/\cS_{2}$ such that $\pi_{x}([y]_{\cS_{1}})=[y]_{\cS_{2}}$. Note that the total map $\pi\colon \cR/\cS_{1}\to \cR/\cS_{2}$ given by $\pi(x,c)=(x,\pi_{x}(c))$ is measurable.

Given $\zeta\in L^{2}(\cR/\cS_{2})$ and $n\in \N$ we define $\widetilde{\zeta}^{(n)}\in L^{2}(\cR/\cS_{1})$ as follows. For  almost every $x\in \cR E$,  there is a unique $i_{x}\in I$ with $x\in \ran(\gamma_{i_{x}})$. 
For almost every $x\in X$ we have for all $c\in [x]_{\cR}/\cS_{1}$ either $c\subseteq \cS_{2}E$ or $c\cap E=\varnothing.$ If $c\subseteq \cS_{2} E$, then since 
 $[x]_{\cR}/\cS_{1}=[\gamma_{i_{x}}^{-1}(x)]_{\cR}/\cS_{1}$,
 there is a unique $j_{c,x}\in J$ with $c\subseteq [\sigma_{j_{c,x}}(\gamma_{i_{x}}^{-1}(x))]_{\cS_{2}}$. 
 We define $\widetilde{\zeta}^{(n)}$  
by saying that $\widetilde{\zeta}^{(n)}_{x}=0$ if either $x\notin \cR E$, or $c\cap E=\varnothing,$ and for $x\in \cR E$, $c\subseteq \cS_{2}E$ we have
\[\widetilde{\zeta}^{(n)}(x,c)=\zeta(x,\pi_{x}(c))\xi^{(n)}(\sigma_{j_{c,x}}(\gamma_{i_{x}}^{-1}(x)),c),\]
where $j_{c,x},i_{x}$ are as above. Note that the set of $(x,c)\in \cR/\cS_{1}$ with $x\in \cR E$ and $c\subseteq \cS_{2}E$ is measurable, by $\cS_{1}$-invariance of $\cS_{2} E$. Further,  for every $\gamma\in [\cR]$ 
\[\widetilde{\zeta}^{(n)}(x,[\gamma(x)]_{\cS_{1}})=\zeta(x,[\gamma(x)]_{\cS_{2}})\xi^{(n)}(\sigma_{j_{c,x}}(\gamma_{i_{x}}^{-1}(x)),[\gamma(x)]_{\cS_{1}})1_{\cR E}(x)1_{\cS_{2} E}(\gamma(x))\]
is a product of measurable functions. So $\widetilde{\zeta}^{(n)}$ is measurable by Proposition \ref{prop:measurability exercise}.
Suppose that $\zeta_{1},\zeta_{2}\in L^{2}(\cR|_{\cR E}/\cS_{2}|_{\cR E})$, we then claim that
\begin{equation}\label{eqn: crucial limiting ip formula}
\lim_{n\to\infty}\ip{\lambda_{\cR/\cS_{1}}(\psi)\widetilde{\zeta}_{1}^{(n)},\widetilde{\zeta}^{(n)}_{2}}=\ip{\lambda_{\cR/\cS_{2}}(\psi)R(1_{\cS_{2}E})\zeta_{1},R(1_{\cS_{2} E})\zeta_{2}} \textnormal{ for all $\psi\in [\cR]$.}
\end{equation}
If $x\in \cR E$, then for each $b\in [\gamma_{i_{x}}^{-1}(x)]_{\cR|_{\cS_{2}E}}/\cS_{2}|_{\cS_{2}E}$, and for all $c\in \pi_{x}^{-1}(\{b\})$ we have that $[\sigma_{j_{c,x}}(\gamma_{i_{x}}^{-1}(x))]_{\cS_{2}}=b$, so we may find a $k_{b,x}\in J$ with $j_{c,x}=k_{b,x}$ for all $c\in \pi_{x}^{-1}(\{b\})$. For almost every $x\in X$, we may let $r_{x}$ be such that $\psi^{-1}(x)\in \ran(\gamma_{r_{x}})$. 
Similarly, we may find an $\ell_{b,x}\in J$ so that 
$j_{c,\psi^{-1}(x)}=\ell_{b,x}$ for all $c\in \pi_{\psi^{-1}(x)}^{-1}(\{b\})$. Thus for all $\psi\in [\cR]$, 
\[\ip{\lambda_{\cR/\cS_{1}}(\psi)\widetilde{\zeta}_{1}^{(n)},\widetilde{\zeta}^{(n)}_{2}}=\]
\[\int_{\cR/\cS_{2}}\overline{[R(1_{\cS_{2}E})\zeta_{2}](x,b)}[R(1_{\cS_{2}E})\zeta_{1}](\psi^{-1}(x),b)\ip{\xi^{(n)}_{\sigma_{\ell_{b,x}}(\gamma_{r_{x}}^{-1}(\psi^{-1}(x)))},\xi^{(n)}_{\sigma_{k_{b,x}}(\gamma_{i_{x}}^{-1}(x))}}_{\ell^{2}(b/\cS_{1})}\,d\mu_{\cR/\cS_{2}}(x,b).\]
Note that 
\[|[R(1_{\cS_{2} E})\overline{\zeta_{2}](x,b)}[(R1_{\cS_{2}E})\zeta_{1}](\psi^{-1}(x),b)\ip{\xi^{(n)}_{\sigma_{\ell_{b,x}}(\gamma_{r_{x}}^{-1}(\psi^{-1}(x)))},\xi^{(n)}_{\sigma_{k_{b,x}}(\gamma_{i_{x}}^{-1}(x))}}_{\ell^{2}(b/\cS_{1})}|\leq |\zeta_{2}(x,b)||\zeta_{1}(\psi^{-1}(x),b)|,\]
and 
\[\int_{\cR/\cS_{2}}|\zeta_{2}(x,b)||\zeta_{1}(\psi^{-1}(x),b)|\,d\mu_{\cR/\cS_{2}}(x,b)\leq\]
\[\left(\int_{\cR/\cS_{2}}|\zeta_{2}(x,b)|^{2}\,d\mu_{\cR/\cS_{2}}(x,b)\right)^{1/2}\left(\int_{\cR/\cS_{2}}|\zeta_{1}(\psi^{-1}(x),b)|^{2}\,d\mu_{\cR/\cS_{2}}(x,b)\right)^{1/2}=\|\zeta_{1}\|_{2}\|\zeta_{2}\|_{2}.\]
Thus, by the dominated convergence theorem, it suffices to show that 
\[\lim_{n\to\infty}\ip{\xi^{(n)}_{\sigma_{\ell_{b,x}}(\gamma_{r_{x}}^{-1}(\psi^{-1}(x))},\xi^{(n)}_{\sigma_{k_{b,x}}(\gamma_{i_{x}}^{-1}(x))}}_{\ell^{2}(b/\cS_{1})}=1, \textnormal{ for almost every $(x,b)\in \cR/\cS_{2}$ with $b\subseteq \cS_{2} E$.}\]

Our choice of $\xi^{(n)}$ implies that for almost every $x\in \ran(\gamma_{i})$ we have, by Cauchy-Schwartz,
\[|\ip{\xi^{(n)}_{\sigma_{\ell_{b,x}}(\gamma_{r_{x}}^{-1}(\psi^{-1}(x)))},\xi^{(n)}_{\sigma_{k_{b,x}}(\gamma_{i}^{-1}(x))}}_{\ell^{2}(b/\cS_{1})}-1|\leq \|\xi^{(n)}_{\sigma_{\ell_{b,x}}(\gamma_{r_{x}}^{-1}(\psi^{-1}(x)))}-\xi^{(n)}_{\sigma_{k_{b,x}}(\gamma_{i}^{-1}(x))}\|_{\ell^{2}(b/\cS_{1})}\to_{n\to\infty}0.\]
This completes the proof of (\ref{eqn: crucial limiting ip formula}).

Having shown (\ref{eqn: crucial limiting ip formula}), since $\lambda_{\cR/\cS_{j}}(\nu)=\sum_{\psi\in [\cR]}\nu(\{\psi\})\lambda_{\cR/\cS_{j}}(\psi)$, $\sum_{\psi\in [\cR]}\nu(\{\psi\})=1$, and $\nu(\{\psi\})\geq 0,$ it follows that 
\[\lim_{n\to\infty}\ip{\lambda_{\cR/\cS_{1}}(\nu)\widetilde{\zeta}_{1}^{(n)},\widetilde{\zeta}_{2}^{(n)}}=\ip{\lambda_{\cR/\cS_{2}}(\nu)R(1_{\cS_{2} E})\zeta_{1},R(1_{\cS_{2} E})\zeta_{2}} \textnormal{ for all $\zeta_{1},\zeta_{2}\in L^{2}(\cR/\cS_{2})$.}\]
Recall that $R(1_{E})\xi^{(n)}=\xi^{(n)}$, so $R(1_{E})\widetilde{\zeta}_{j}=\widetilde{\zeta}_{j}$ for $j=1,2$ and all $\zeta_{1},\zeta_{2}\in L^{2}(\cR/\cS_{2})$. Since $R(1_{E})$ commutes with $\lambda_{\cR/\cS_{1}}$, and $R(1_{E})^{2}=R(1_{E})=R(1_{E})^{*}$ it follows that
\[\lim_{n\to\infty}\ip{R(1_{E})\lambda_{\cR/\cS_{1}}(\nu)\widetilde{\zeta}_{1}^{(n)},\widetilde{\zeta}_{2}^{(n)}}=\ip{R(1_{\cS_{2} E})\lambda_{\cR/\cS_{2}}(\nu)\zeta_{1},\zeta_{2}} \textnormal{ for all $\zeta_{1},\zeta_{2}\in L^{2}(\cR/\cS_{2})$.}\]
This directly implies that 
\[\|R(1_{\cS_{2} E})\lambda_{\cR/\cS_{2}}(\nu)\|\leq \|R(1_{E})\lambda_{\cR/\cS_{1}}(\nu)\|.\]

\end{proof}

We will use the above Proposition to relate coamenability of a subrelation to the cospectral radius being almost surely $1$, for now we note the following general consequence.

\begin{cor}\label{cor: equality of co-spectral radii}
Let $(X,\mu)$ be a standard probability space and $\cR$ be a discrete, probability measure-preserving equivalence relations.
\begin{enumerate}[(i)]
    \item Suppose that $\cS_{1}\leq \cS_{2}\leq \cR$. If $E\subseteq X$ is $\cS_{1}$-invariant and has positive measure, and if $\cS_{1}|_{E}$ is everywhere coamenable in $\cS_{2}|_{E}$, then $\rho^{\cS_{2}}_{\nu}1_{E}=\rho^{\cS_{1}}_{\nu}1_{E}$ almost everywhere for every countably supported $\nu\in \Prob([\cR])$ which is symmetric and generates $\cR$. \label{item: equality of co-spectral radiii}
    \item If $\cS$ is coamenable in $\cR$, and $\cS$ is hyperfinite, then $\cR$ is hyperfinite. \label{item: easy hyperfinite condition}
\end{enumerate}

\end{cor}

\begin{proof}
(\ref{item: equality of co-spectral radiii}): We always have that $\rho^{\cS_{2}}_{\nu}\geq \rho^{\cS_{1}}_{\nu}$ almost everywhere. To prove the reverse inequality holds almost surely, since $\rho^{\cS_{2}}_{\nu}$, $\rho^{\cS_{1}}_{\nu}$ are almost surely $\cS_{1}$-invariant, it suffices to show that $\|\rho^{\cS_{2}}_{\nu}1_{F}\|_{\infty}\leq \|\rho^{\cS_{1}}_{\nu}1_{F}\|_{\infty}$ for every $\cS_{1}$-invariant measurable set $F\subseteq E$ with $\mu(F)>0$. This follows from Proposition \ref{prop: co spectral radius descent} and (\ref{eqn: co spectral radius is norm restricted to subset background}).

(\ref{item: easy hyperfinite condition}):  Note that any one of parts (\ref{item: Folner sets in invariant sets}), (\ref{item: rel amen in terms of Markov operator}), (\ref{item: co-spectral radius condition nonergodic part}) of Theorem \ref{thm: big TFAE thm} imply that $\cS|_{E}$ is coamenable in $\cR|_{E}$ for any $\cR$-invariant measurable set $E$ (since any $\cR|_{E}$-invariant measurable subset of $E$ is an $\cR$-invariant subset of $X$). Applying Proposition \ref{prop: co spectral radius descent} with $\cS_{1}$ being the trivial relation and $\cS_{2}=\cS$, and $E$ being any $\cR$-invariant subset of $X$ of positive measure, we see that
\[\rho_{E}(\cR,\nu)=\rho_{E}(\cR/\cS,\nu)=1\]
for every probability measure $\nu$ whose support generates $\cR$, and for every $\cR$-invariant subset $E$ of $X$. This implies that for every countable $\Gamma\leq [\cR]$ we have that $\lambda_{\cR|_{E}}|_{\Gamma}$ has almost invariant vectors for every $\cR$-invariant $E\subseteq X$ with $\mu(E)>0$. Theorem \ref{thm: big TFAE thm} thus implies that there is a $[\cR]$-invariant mean $m\in L^{\infty}(\cR)^{*}$ with $m|_{L^{\infty}(X)}=\int \cdot\,d\mu$ (see also \cite[Lemma 2.2]{BHAExtension}, as well as \cite[Lemma 2.2]{HaagerupAmenable}). 
It follows by \cite{KaimanovichAmenability} that $\cR$ is amenable (see also our Theorem \ref{thm: big TFAE thm}), and thus  \cite{CFW} implies that $\cR$ is hyperfinite. 

\end{proof}

The assumption in (\ref{item: equality of co-spectral radiii}) that $\ip{\supp(\nu)}$ generates is necessary. To see this, note that by \cite{MonodPopa} there are finitely generated groups $\Delta\leq \Lambda\leq \Gamma$ so that $\Delta$ is coamenable in $\Gamma$, but not in $\Lambda$ (see also \cite[Section 4]{monodcomments} for more examples). 
Let $\Gamma\actons (X,\mu)$ be a free probability measure-preserving action, and view $\Gamma\leq [\cR_{\Gamma,X}]$. Fix a finite, symmetric generating set $S$ for $\Lambda$. Set $\cR=\cR_{\Gamma,X}=\cS_{2}$ and $\cS_{1}=\cR_{\Delta,X}$. As we will later show (see Proposition \ref{prop: group case is an iff}) $\cS_{1}$ is everywhere coamenable in $\cR$. However, if $u_{S}$ is the uniform measure on $S$, then we have that $p_{2n,x,\cS_{1}}=\ip{\lambda_{\Lambda/\Delta}(u_{S})^{2n}\delta_{\Delta},\delta_{\Delta}}_{B(\ell^{2}(\Lambda/\Delta))}$ for every $n\in \N$. This implies, e.g. by  \cite[Lemma 10.1]{Woess}, that for almost every $x\in X$ we have:
\[\rho^{\cS_{1}}_{u_{S}}(x)=\lim_{n\to\infty}\ip{\lambda_{\Lambda/\Delta}(u_{S})^{2n}\delta_{\Delta},\delta_{\Delta}}_{\ell^{2}(\Lambda/\Delta)}=
\|\lambda_{\Lambda/\Delta}(u_{S})\|_{B(\ell^{2}(\Lambda/\Delta))}<1,\]
the last step following from the fact that $\Delta$ is not coamenable in $\Lambda$ and Theorem \ref{thm: folklore 2}.

Note that we will later show that   $\cS$ is coamenable in $\cR$ if and only if the von Neumann algebra of $\cS$ is coamenable in the von Neumann algebra of  $\cR$ (see Theorem \ref{thm: two different version of rel amen are the same}).
So Corollary \ref{cor: equality of co-spectral radii} (\ref{item: easy hyperfinite condition}) also follows from \cite[Theorem 3.2.4]{PopaCorr}.

Proposition \ref{prop: co spectral radius descent} will quickly imply that if $\cS$ is everywhere coamenable in $\cR$, then $\rho^{\cS}_{\nu}=1$ almost everywhere for every countably supported $\nu\in \Prob([\cR])$ which is symmetric and generates $\cR$. For the converse, we will need that if the pointwise defined cospectral radius is large on a subset, then there is a sequence of invariant vectors which are fiberwise supported on equivalence classes contained in that subset. Precisely, we have the following proposition.

\begin{prop}\label{prop: almost invariant vectors from restricted co-spectral raidus}
Let $(X,\mu)$ be a standard probability space and $\cS\leq \cR$ be discrete, probability measure-preserving equivalence relations on $(X,\mu)$. Suppose that $\nu\in \Prob([\cR])$, and that $E\subseteq X$ is measurable and $\cS$-invariant with $\mu(E)>0$. If $\rho_{E}(\cR/\cS,\nu)=1$, then there is a sequence $\xi^{(n)}\in L^{2}(\cR/\cS)$ with  $\|\xi^{(n)}\|_{2}=1$, $R(1_{E})\xi^{(n)}=\xi^{(n)}$ and $\|\lambda_{\cR/\cS}(\gamma)\xi^{(n)}-\xi^{(n)}\|_{2}\to_{n\to\infty}0$ for every $\gamma\in \ip{\supp(\nu)}$.   
\end{prop}

\begin{proof}
Set $\Gamma=\ip{\supp(\nu)}$, and $\theta=\frac{1}{2}\nu+\frac{1}{2}\delta_{e}$. Note that 
\[\ip{\lambda_{\cR/\cS}(\theta)^{2k}\zeta_{E},\zeta_{E}}=\sum_{j=0}^{2k}2^{-2k}\binom{2k}{j}\ip{\lambda_{\cR/\cS}(\nu)^{j}\zeta_{E},\zeta_{E}}.\]
Let $r_{k}=2\lfloor{k/2\rfloor}$. Then:
\[\rho_{E}(\cR/\cS,\theta)=\lim_{k\to\infty}\ip{\lambda_{\cR/\cS}(\theta)^{2k}\zeta_{E},\zeta_{E}}^{1/2k}\geq \lim_{k\to\infty}\frac{1}{2}\binom{2k}{r_{k}}^{1/2k}\ip{\lambda_{\cR/\cS}(\nu)^{r_{k}}\zeta_{E},\zeta_{E}}^{1/2k}=\rho_{E}(\cR/\cS,\nu)=1,\]
the second-to-last formula following by Stirling's formula and the definition of $\rho_{E}(\cR/\cS,\nu)$. So if we replace $\nu$ with $\theta,$
then we still have that $\rho_{E}(\cR/\cS,\nu)=1$, and now $\nu$ is lazy. So we may, and will, assume that $\nu$ is lazy. 

Set $a_{k}=\|\lambda_{\cR/\cS}(\nu)^{k}\zeta_{E}\|_{2}^{2}$. Since $\|\zeta_{E}\|_{2}=1$ and $\|\lambda_{\cR/\cS}(\nu)\|\leq 1$, the spectral theorem implies that there is a $\eta\in \Prob([-1,1])$ with
\[\|\lambda_{\cR/\cS}(\nu)^{k}\zeta_{E}\|_{2}^{2}=\ip{\lambda_{\cR/\cS}(\nu)^{2k}\zeta_{E},\zeta_{E}}=\int t^{2k}\,d\eta(t), \text{ for all $k\in \N\cup\{0\}$}.\]
Since $\eta\in \Prob([-1,1])$ the above formula shows that $a_{k}$ is decreasing. By the usual comparison between the ratio test and the root test \cite[Theorem 3.37]{BabyRudin}  we have 
\[\limsup_{k\to\infty}\frac{a_{2k+2}}{a_{2k}}\geq \lim_{k\to\infty} \ip{\lambda_{\cR/\cS}(\nu)^{4k}\zeta_{E},\zeta_{E}}^{1/k}=\rho_{E}(\cR/\cS,\nu)^{4}=1.\]
And since $a_{k}$ is decreasing, we thus have $\limsup_{k\to\infty}\frac{a_{2k+2}}{a_{2k}}=1$. In particular, we may find a subsequence $a_{n_{k}}$ with $\frac{a_{n_{k}+2}}{a_{n_{k}}}\to 1$. Since $a_{k}$ is decreasing, it follows that $\frac{a_{n_{k}+1}}{a_{n_{k}}}\to 1$. 
Set $\xi^{(k)}=\frac{1}{a_{n_{k}}^{1/2}}\lambda_{\cR/\cS}(\nu)^{n_{k}}\zeta_{E}$. Then $\|\xi^{(k)}\|_{2}=1$, and 
\[\|\lambda_{\cR/\cS}(\nu)^{2}\xi^{(k)}-\xi^{(k)}\|_{2}^{2}=\frac{a_{n_{k}+2}+a_{n_{k}}-2a_{n_{k}+1}}{a_{n_{k}}}.\]
So our choice of $a_{n_{k}}$ implies that $\|\lambda_{\cR/\cS}(\nu)^{2}\xi^{(k)}-\xi^{(k)}\|_{2}\to 0$, and $R(1_{E})\xi^{(k)}=\xi^{(k)}$ since $R(1_{E})$ commutes with $\lambda_{\cR/\cS}$, and $R(1_{E})\zeta_{E}=\zeta_{E}$.   Since $\nu$ is lazy, we have that $\ip{\supp(\nu^{*2})}=\ip{\supp(\nu)}=\Gamma$, and so the result follows.  
\end{proof}

\begin{thm}\label{thm: everywhere rel man TFAE thm}
 Let $(X,\mu)$ be a standard probability space and $\cS\leq \cR$ discrete, probability measure-preserving equivalence relations. Then the following are equivalent:
\begin{enumerate}[(i)]
    \item $\cS$ is everywhere coamenable in $\cR$, \label{item: everywhere rel amen}
    \item for every $\cS$-invariant  measurable $E\subseteq X$ with $\mu(E)>0$, we have $\cS|_{E}$ is coamenable in $\cR|_{E}$, \label{item: everywhere rel amen via invariant sets}
    \item for every $\cS$-invariant measurable set $E$ of positive measure, there is a $[\cR]$-invariant mean $m\colon L^{\infty}(\cR|_{\cR E}/\cS|_{\cR E})\to \C$ so that $m(R(1_{E})f)=m(f)$ for all $f\in L^{\infty}(\cR|_{\cR E}/\cS|_{\cR E})$, and $m|_{L^{\infty}(\cR E)}=\frac{1}{\mu(\cR E)}\int_{\cR E} \cdot\,d\mu.$ \label{item: everywhere rel amen via means}
    \item for every $\cS$-invariant measurable set $E$ with $\mu(E)>0$, we have that $\lambda_{\cR|_{E}/\cS|_{E}}$ has almost invariant vectors. \label{item: everywhere rel amen restricted almost invariant vectors}
    \item for every $\cS$-invariant measurable sets $E$ of positive measure, there is a sequence $\xi^{(n)}\in L^{2}(\cR|_{\cR E}/\cS|_{\cR E})$ with $\|\xi^{(n)}_{x}\|_{2}=1$ for almost every $x\in X$,  and $R(1_{E})\xi^{(n)}=\xi^{(n)}$, and with $\|\xi^{(n)}_{x}-\xi^{(n)}_{y}\|_{2}\to_{n\to\infty}0$ for almost every $(x,y)\in \cR$.  \label{item: everywhere rel amen via sequences}
    \item for every $\nu\in \Prob([\cR])$ such that $\supp(\nu)$ generates $\cR$, we have that $\rho^{\cS}_{\nu}=1$ almost everywhere. \label{item: everywhere rel amen via co-spectra radisu }    
\end{enumerate}
\end{thm}

\begin{proof}

(\ref{item: everywhere rel amen via invariant sets}) implies (\ref{item: everywhere rel amen}): This follows from Proposition \ref{prop: passing rel amen between subsets}.

(\ref{item: everywhere rel amen}) implies (\ref{item: everywhere rel amen via co-spectra radisu }):  Taking $\cS_{2}=\cR$, $\cS_{1}=\cS$ in Proposition \ref{prop: co spectral radius descent} implies that $\|\rho^{\cS}_{\nu}1_{E}\|_{\infty}=\rho_{E}(\cR/\cS,\nu)=1$ for every measurable $E\subseteq X$, and since $\rho^{\cS}_{\nu}\geq 0$, this implies that $\rho^{\cS}_{\nu}=1$ almost everywhere. 

(\ref{item: everywhere rel amen via co-spectra radisu }) implies (\ref{item: everywhere rel amen via means}):  For a measurable $\cR$-invariant set $F$ of positive measure, we have a natural identification of $L^{2}(\cR|_{F}/\cS|_{F})$ into $L^{2}(\cR/\cS)$, since $[x]_{\cR|_{F}}=[x]_{\cR}$ for almost every $x\in F$ (by $\cR$-invariance). 
As in the proof of Theorem \ref{thm: big TFAE thm} (\ref{item: co-spectral radius condition nonergodic part}) implies (\ref{item:inv mean restricting well}), it suffices to show that for every finite $\Phi\subseteq [\cR]$ and every partition modulo null sets $\cR E=\bigsqcup_{j=1}^{k}F_{j}$ where $F_{j}$ is $\cR$-invariant, and $\mu(F_{j})>0$ there is a mean $m\colon L^{\infty}(\cR|_{\cR E}/\cS|_{\cR E})\to \C$ such that
\begin{itemize}
    \item $m(1_{F_{j}})=\mu(F_{j})$, $j=1,\cdots,k$,
    \item and $m(R(1_{E})f)=m(f)$ for all $f\in L^{\infty}(\cR|_{\cR E}/\cS|_{\cR E})$,
    \item $m(f\circ \gamma)=m(f)$ for all $\gamma\in \Phi$ and all $f\in L^{\infty}(\cR|_{\cR E}/\cS|_{\cR E})$. 
\end{itemize}
Let $\nu\in \Prob([\cR])$ be  lazy, countably supported, symmetric, with  $\supp(\nu)$ generating $\cR$, and $\Phi\subseteq \supp(\nu)$. Note that for all $j=1,\cdots,k$ we have that $\mu(F_{j}\cap E)>0$, since $F_{j}$ are $\cR$-invariant sets of positive measure and are subsets of $\cR E$. For $j=1,\cdots,k$, let $\cH_{j}=R(1_{F_{j}\cap E})L^{2}(\cR/\cS)$, and note that $\cH_{j}\subseteq L^{2}(\cR|_{F_{j}}/\cS|_{F_{j}})$, and that $\cH_{j}$ is $\lambda_{\cR/\cS}$-invariant. Since $\mu(F_{j}\cap E)>0$,
\[\|\lambda_{\cR/\cS}(\nu)|_{\cH_{j}}\|=\|R(1_{F_{j}\cap E})\lambda_{\cR/\cS}(\nu)\|=\rho_{F_{j}\cap E}(\cR/\cS,\nu)=\|\rho^{\cS}_{\nu}1_{F_{j}\cap E}\|_{\infty}=1,\]
the second equality following by Lemma \ref{lem: restricted co-spec radi is just restricted op}, the third by (\ref{eqn: co spectral radius is norm restricted to subset background}), and the last equality by hypothesis. By Theorem \ref{thm: folklore 2}, for $j=1,\cdots,k$ we may find a sequence $\xi^{(n)}_{j}\in L^{2}(\cR|_{F_{j}}/\cS|_{F_{j}})$ with 
\begin{itemize}
\item $\|\xi^{(n)}_{j}\|_{L^{2}(\cR|_{F_{j}}/\cS|_{F_{j}})}=1$, for all $j=1,\cdots,k$ and $n\in \N$,
    \item $\|\lambda_{\cR/\cS}(\gamma)\xi^{(n)}_{j}-\xi^{(n)}_{j}\|_{2}\to_{n\to\infty}0$, for all $j=1,\cdots,k$, and $\gamma\in\ip{\supp(\nu)}$,
    \item $R(1_{F_{j}\cap E})\xi^{(n)}_{j}=\xi^{(n)}_{j}$, for all $j=1,\cdots,k$ and $n\in \N$,
\end{itemize}
Define $\xi^{(n)}\in L^{2}(\cR|_{\cR E}/\cS|_{\cR E})$ by $\xi^{(n)}_{x}=\xi^{(n)}_{j,x}$ for $x\in F_{j}$. Then 
\begin{itemize}
\item $\|1_{F_{j}}\xi^{(n)}\|_{2}=\mu(F_{j})^{1/2}$, for all $j=1,\cdots,k$ and $n\in \N$,
    \item $\|\lambda_{\cR/\cS}(\gamma)\xi^{(n)})-\xi^{(n)}\|_{2}\to_{n\to\infty}0$, for all  $\gamma\in\ip{\supp(\nu)}$,
    \item $R(1_{E})\xi^{(n)}=\xi^{(n)}$.
\end{itemize}
Define means $m_{n}\colon L^{\infty}(\cR|_{\cR E}/\cS|_{\cR E})\to \C$ by $m_{n}(f)=\ip{M_{f}\xi^{(n)},\xi^{(n)}}$, and let $m$ be any mean which is a weak$^{*}$ cluster point of the $m_{n}$'s. Then
\begin{itemize}
\item $m(1_{F_{j}})=\mu(F_{j})$ for all $j=1,\cdots,k$,
    \item $m(f\circ \gamma)=m(f)$ for all $f\in L^{\infty}(\cR|_{\cR E}/\cS|_{\cR E})$ and all $\gamma\in \ip{\supp(\nu)}$.
    \item $m(R(1_{E})f)=m(f)$ for $f\in L^{\infty}(\cR|_{\cR E}/\cS|_{\cR E})$.
\end{itemize}

(\ref{item: everywhere rel amen via means}) implies (\ref{item: everywhere rel amen via sequences}):
Let $\Gamma\leq [\cR]$ be countable and generate $\cR$. As in the proof Theorem \ref{thm: big TFAE thm} (\ref{item:inv mean restricting well}) implies (\ref{item: Rieter sequence on subgroup}),  applying Hahn-Banach separation may find a sequence $\mu^{(n)}\in L^{1}(\cR|_{\cR E}/\cS|_{\cR E})$ so that 
\begin{itemize}
    \item $\mu^{(n)}_{x}\in \Prob([x]_{\cR|_{\cR E}}/\cS|_{\cR E})$ for almost every $x\in \cR E$,
    \item for almost every $x\in X$ one has $\mu^{(n)}_{x}(c)=0$ if $c\in [x]_{\cR|_{\cR E}}/\cS|_{\cR E}$ and $c\cap E=\varnothing$,
    \item  and $\|\lambda_{\cR/\cS}(\gamma)\mu^{(n)}-\mu^{(n)}\|_{1}\to_{n\to\infty}0$, for all $\gamma\in \Gamma$.
\end{itemize}
 Since $\Gamma$ is countable, we may pass to a subsequence  and assume that $\|\mu^{(n)}_{\gamma(x)}-\mu^{(n)}_{x}\|_{1}\to_{n\to\infty}$ for all $\gamma\in \Gamma$ and almost every $x\in X$. Since $\Gamma$ generates $\cR$, we have that 
\[\|\mu^{(n)}_{y}-\mu^{(n)}_{x}\|_{1}\to_{n\to\infty}0 \text{ for almost every $(x,y)\in \cR$.}\]
Setting $\xi^{(n)}_{x}=(\mu^{(n)}_{x})^{1/2}$, it
follows that 
\[\|\xi^{(n)}_{x}-\xi^{(n)}_{y}\|_{2}\to_{n\to\infty}0\]
for almost every $(x,y)\in \cR$.

(\ref{item: everywhere rel amen via sequences}) implies (\ref{item: everywhere rel amen restricted almost invariant vectors}):
Suppose that $E\subseteq X$ is $\cS$-invariant with $\mu(E)>0$, and let $\xi^{(n)}\in L^{2}(\cR|_{\cR E}/\cS|_{\cS E})$ be as in (\ref{item: everywhere rel amen via sequences}). Note that for almost every $x\in E$, we have that $\xi^{(n)}(x,c)=0$ if $c\notin\Im(\iota_{x})$, since $R(1_{E})\xi^{(n)}=\xi^{(n)}$. Thus if we define $\zeta^{(n)}\in L^{2}(\cR|_{E}/\cS|_{E})$ by $\zeta^{(n)}(x,c)=\xi^{(n)}(x,\iota_{x}(c))$, then we still have that $\|\zeta^{(n)}_{x}\|_{2}=1$ for almost every $x\in E$, and $\|\zeta^{(n)}_{x}-\zeta^{(n)}_{y}\|_{2}\to_{n\to\infty}0$ for almost every $(x,y)\in \cR|_{E}$. By Theorem \ref{thm: DCT argument}, it follows that $\zeta^{(n)}$ is a sequence of almost invariant unit vectors for $\lambda_{\cR|_{E}/\cS|_{E}}$.

 (\ref{item: everywhere rel amen restricted almost invariant vectors}) implies (\ref{item: everywhere rel amen via invariant sets}): If $E\subseteq X$ is measurable and $\cS$-invariant, then any measurable $\cR|_{E}$-invariant subset of $E$ is necessarily $\cS$-invariant (as a subset of $X$). Thus the desired result follows from Theorem \ref{thm: big TFAE thm} (\ref{item: almost invariant vectors nonergodic part inv sets}).

\end{proof}

\section{Simplification of the cospectral radius condition under spectral gap}\label{sec:cospectral radius}

As alluded to in the discussion following Theorem \ref{thm: main theorem intro} results of Kaimanovich \cite{VKLeafcounterexample} imply  that in Theorem \ref{thm: big TFAE thm} (\ref{item: co-spectral radius condition nonergodic part}) we cannot in general replace ``every generating probability measure" with ``some generating probability measure". In this section, we present cases where such a replacement can be done. These cases  rely on a spectral gap assumption of the action of the group generated by the support of the probability measure. 

\begin{prop}\label{prop: spectral gap simplify co-amenability}
Let $\cS\leq \cR$ be an inclusion of probability measure-preserving equivalence relations over a standard probability space $(X,\mu)$ with $\cR$ ergodic. Suppose that $\nu\in \Prob([\cR])$ is countably supported and generates $\cR$. Assume moreover that $\ip{\supp(\nu)}\actson (X,\mu)$ has spectral gap. Then $\cS$ is coamenable in $\cR$ if and only if $\rho(\cR/\cS,\nu)=1$.
    
\end{prop}

\begin{proof}
One direction follows from Theorem \ref{thm: big TFAE thm}. Hence we focus on proving that if $\rho(\cR/\cS,\nu)=1$, then $\cS$ is coamenable in $\cR$. Recall that $\alpha\colon [\cR]\to \cU(L^{2}(X,\mu))$ is given by $(\alpha(\gamma)\xi)(x)=\xi(\gamma^{-1}(x)).$
Note that $1=\rho(\cR/\cS,\nu)=\|\lambda_{\cR/\cS}(\nu)\|$, hence we may find a sequence $\xi^{(n)}\in L^{2}(\cR/\cS)$ with $\|\xi^{(n)}\|_{2}=1$ such that 
\[\|\lambda_{\cR/\cS}(\gamma)\xi^{(n)}-\xi^{(n)}\|_{2}\to_{n\to\infty}0, \textnormal{ for all $\gamma\in \ip{\supp(\nu)}$.}\]
Define $f_{n}\in L^{2}(X,\mu)$ by $f_{n}(x)=\|\xi^{(n)}_{x}\|_{2}$. Then, by the reverse triangle inequality,
\[\|\alpha(\gamma)f_{n}-f_{n}\|_{2}\leq \|\lambda_{\cR/\cS}(\gamma)\xi^{(n)}-\xi^{(n)}\|_{2}\to_{n\to\infty}0. \]
Thus by spectral gap, there is a $\lambda_{n}\in [0,+\infty]$ with $\|f_{n}-\lambda_{n}\|_{2}\to_{n\to\infty}0$. Since $\|f_{n}\|_{2}=1$, it follows that $|1-\lambda_{n}|\to_{n\to\infty}0$ so $\|f_{n}-1\|_{2}\to_{n\to\infty}0$. Since $f_{n}(x)=\|\xi^{(n)}_{x}\|_{2}$, this implies that there is a $\zeta^{(n)}\in L^{2}(\cR/\cS)$ with $\|\xi^{(n)}-\zeta^{(n)}\|_{2}\to_{n\to\infty}0$ and $\|\zeta^{(n)}_{x}\|_{2}=1$ for almost every $x\in X$. Since $\xi^{(n)}$ is almost $\ip{\supp(\nu)}$-invariant, we have that $\zeta^{(n)}$ is almost $\ip{\supp(\nu)}$-invariant. Thus Theorem \ref{thm: big TFAE thm} implies that $\cS$ is coamenable in $\cR$.

\end{proof}

We would like to obtain a slight refinement of the above result, which will allow us to deduce coamenability of an inclusion of relations restricted to a subset from the cospectral radius of a single measure, under a spectral gap assumption.

\begin{prop}\label{prop: spectral gap compress subreln}
Let $\cS\leq \cR$ be measure-preserving, discrete equivalence relations over a standard probability space $(X,\mu)$. Suppose that $\nu\in \Prob([\cR])$ is countably supported and generates $\cR$, and that $\ip{\supp(\nu)}\actons (X,\mu)$ has spectral gap. Given an $\cS$-invariant $E\subseteq X$ with positive measure, we have that $\cS|_{E}$ is coamenable in $\cR|_{E}$ if and only if $\rho_{E}(\cR/\cS,\nu)=1$.
\end{prop}

\begin{proof}

One implication follows from Proposition \ref{prop: co spectral radius descent}, taking $\cS_{2}=\cR$,$\cS_{1}=\cS$. 

For the other implication, set $\Gamma=\ip{\supp(\nu)}$. By Theorem \ref{thm: folklore 2} and Lemma \ref{lem: restricted co-spec radi is just restricted op}, we may find a sequence $\xi^{(n)}\in L^{2}(\cR/\cS)$ with $\|\xi^{(n)}\|_{2}=1$, $R(1_{E})\xi^{(n)}=\xi^{(n)}$, and which is almost $\lambda_{\cR/\cS}|_{\Gamma}$-invariant. Since $\Gamma\actson (X,\mu)$ has spectral gap,  as in the proof of Proposition \ref{prop: spectral gap simplify co-amenability}, we find that $\|f_{k}-1\|_{2}\to 0$, where $f_{k}(x)=\|\xi_{x}^{(k)}\|_{2}$. 
Define $v_{k}\in L^{2}(\cR/\cS)$ by
\[v_{k}(x,c)=\begin{cases}
    \frac{\xi^{(k)}(x,c)}{\|\xi^{(k)}_{x}\|_{2}},& \textnormal{ if $\|\xi^{(k)}_{x}\|\geq 1/2$},\\
    1_{x\in c},& \textnormal{ if $\|\xi^{(k)}_{x}\|< 1/2$}
\end{cases}.\]
Since $\|f_{k}-1\|_{2}\to 0$, we have  $\|v_{k}-\xi^{(k)}\|_{2}\to 0$. Moreover $\|v_{k,x}\|_{2}=1$ and $\|\lambda_{\cR/\cS}(\nu)v_{k}-v_{k}\|_{2}\to 0$, since $\|\lambda_{\cR/\cS}(\nu)\xi^{(k)}-\xi^{(k)}\|_{2}\to 0$. Thus passing to a further subsequence, we may assume that $\|v_{k,\gamma(x)}-v_{k,x}\|_{2}\to 0$ for all $\gamma\in \Gamma$, and almost every $x\in X$. 
For $x\in E$, observe that there is a well-defined map $\iota_{x}\colon [x]_{\cR|_{E}}/\cS|_{E}\to [x]_{\cR}/\cS$ which sends $\iota_{x}([y]_{\cS|_{E}})=[y]_{\cS}$. 
Since $R(1_{E})\xi^{(k)}=\xi^{(k)}$, if we define $w_{k}\in L^{2}(\cR|_{E}/\cS|_{E})$ by $w_{k}(x,c)=v_{k}(x,\iota_{x}(c))$, then we still have $\|w_{k,x}\|_{2}=1$ for almost every $x\in E$. Since $\|w_{k,\gamma(x)}-w_{k,x}\|_{2}\to 0$ for almost every $x\in E$ and every $\gamma\in \Gamma$, it follows as in the proof of Theorem  \ref{thm: big TFAE thm} (\ref{item:almost invariant fiberwise nonergodic intro part 2}) implies (\ref{item:almost invariant fiberwise norm 1 nonergodic part}) that $\cS|_{E}$ is coamenable in $\cR|_{E}$.

\end{proof}

\subsection{Connections to percolation theory}
We follow much of the presentation from \cite{Gaboriau-Lyons, SolidErg}.
Let $\Gamma$ be a finitely generated group with finite, symmetric generating set $S$ with $e\notin S$. Let $\cG=(\Gamma,E)$ be the associated Cayley graph, where $E=\{(\gamma,\gamma s):\gamma\in \Gamma,s\in S\}$. We define an equivalence relation $\thicksim$ on $E$ by saying that $(a,b)\thicksim (b,a)$ for $(a,b)\in E$. We abuse notation and regard $\{0,1\}^{E/\thicksim}$ as all functions $x\colon E\to \{0,1\}$ such that $x(a,b)=x(b,a)$ for $(a,b)\in E$. For $x\in \{0,1\}^{E/\thicksim}$, we have a natural subgraph of $\cG$ given by $\cG_{x}=(\Gamma,x^{-1}(\{1\}))$.
For $p\in (0,1)$, let $\mu_{p}=((1-p)\delta_{\{0\}}+p\delta_{\{1\}})^{\otimes E/\thicksim}$. The random variable $x\mapsto \cG_{x}$ on the probability space $(\{0,1\}^{E/\thicksim},\mu_{p})$ is known as \emph{$p$-Bernoulli edge percolation}. Note that $\Gamma$ has a natural probability measure-preserving action $\Gamma\actson\{0,1\}^{E/\thicksim}$ induced by the action $\Gamma\actson E$, namely, 
\[(\gamma x)(a,b)=x(\gamma^{-1}a,\gamma^{-1}b).\]
We set $\cR=\cR_{\Gamma,\{0,1\}^{E/\thicksim}}$.
Typical questions in percolation theory are about the connected components in $\cG_{x}$ and this can be encoded in terms of the subrelation
\[\cS=\{(x,\gamma x):x\in \{0,1\}^{E/\thicksim}, \gamma\in\Gamma, \textnormal{ and e  and $\gamma^{-1}$ are connected in $\cG_{x}$}\}.\]

Often one is interested in the structure of infinite components. We can study this by consider 
\[U_{\infty}=\{x\in X: [x]_{\cS} \textnormal{ is infinite}\}.\]
For instance, a common occurrence is to ask when there are infinitely many infinite connected components $\mu_{p}$-almost everywhere. This is equivalent to the condition that for almost every $x\in U_{\infty}$ we have that $[x]_{\cR|_{U_{\infty}}/\cS|_{U_{\infty}}}$ is infinite. 

If $\nu\in \Prob(\Gamma)$ is a finitely supported, symmetric probability measure with $\supp(\nu)=S$, then we abuse notation and consider $\nu$ as a probability measure on $[\cR]$  by identifying $\Gamma$ with its image in $[\cR]$.

Following the work in \cite{AFH}, Gabor Pete asked the authors of \cite{AFH} the following question.

\begin{question}
 Suppose that $p\in (0,1)$ and that $\cG_{x}$ has infinitely many infinite connected components $\mu_{p}$-almost everywhere. Suppose that $\nu\in \Prob(\Gamma)$ is finitely supported, symmetric, and $\ip{\supp(\nu)}=\Gamma.$
 Do we have that $\rho^{\cS}_{\nu}(x)<1$ for $\mu_{p}$-almost every $x\in U_{\infty}$?
\end{question}

One consequence on our work on spectral gap in the preceding subsection is the following.

\begin{cor}
Assume the above setup, and fix $p\in (0,1)$, and $\nu\in \Prob(\Gamma)$ is symmetric with $\ip{\supp(\nu)}=\Gamma$. Then the following are equivalent:
\begin{enumerate}[(i)]
    \item $\rho^{\cS}_{\nu}(x)<1$ for $\mu_{p}$-almost every $x\in U_{\infty}$,
    \item $\cR|_{U_{\infty}}$ is not amenable relative to $\cS|_{U_{\infty}}$. 
\end{enumerate}
\end{cor}

\begin{proof}
Note that if $\Gamma$ is amenable, then by \cite[Proposition 3.4]{ZimmerAMen1} and \cite[Theorem 2.1]{ZimmerAmenPoisson} $\cR$ is amenable, and so $\cR$ is everywhere amenable relative to $\cS$, and thus by Theorem \ref{thm: everywhere rel man TFAE thm} we have that  $\rho^{\cS}_{\nu}(x)=1$ for $\mu_{p}$-almost every $x\in U_{\infty}$.
So we may, and will, assume that $\Gamma$ is nonamenable. 

By the indistinguishability theorem of Lyons-Schramm \cite{LyonsSchramm} and \cite[Proposition 5]{Gaboriau-Lyons}, we have that $\cS|_{U_{\infty}}$ is ergodic. Since $\rho^{\cS}_{\nu}$ is almost surely constant on $\cS$-classes, we have that $\rho^{\cS}_{\nu}|_{U_{\infty}}$ is almost surely constant. Since $\Gamma$ is nonamenable, and $\Gamma\actson (\{0,1\}^{E/\thicksim},\mu_{p})$ is a nontrivial Bernoulli action, we have that $\ip{\supp(\nu)}=\Gamma\actson (\{0,1\}^{E/\thicksim},\mu_{p})$ has spectral gap. The desired result now follows from Proposition \ref{prop: spectral gap compress subreln}.
    
\end{proof}

Note that \cite{HPMonotone} gives conditions which guarantee that a percolation has infinitely many infinite connected components almost surely. The argument there goes by showing that once there are at least two infinite connected components, there are infinitely many infinite connected components. It may be that a finer analysis of that argument gives something akin to a ``paradoxical decomposition" in the space $(\cR|_{U_{\infty}}/\cS|_{U_{\infty}},\mu_{p,\cR|_{U_{\infty}/\cS_{U_{\infty}}}})$ which allows one to prove that $\cR|_{U_{\infty}}$ is not amenable relative to $\cS|_{U_{\infty}}$.


\section{Coamenability in terms of von Neumann algebras} \label{sec: vNa stuff}

In this section, we explain how the definition of coamenability for relations given by Theorem \ref{thm: big TFAE thm} is equivalent to the one implicitly defined in \cite{MonodPopa, PopaCorr} via von Neumann algebras. To do this, we will recall some general background on von Neumann algebras.

\subsection{Preliminaries on von Neumann algebras}
 Let $\cH$ be a Hilbert space. Recall that the \emph{strong operator topology} (often abbreviated SOT) on $B(\cH)$ is defined by saying that for $T\in B(\cH)$, the sets
 \[\mathcal{O}_{F,\varepsilon}(T)=\bigcap_{\xi\in F}\{S\in B(\cH):\|(S-T)\xi\|<\varepsilon\}\]
ranging over finite subsets $F$ of $\cH$ and $\varepsilon>0$ form a neighborhood basis at $T$. One typically thinks of the SOT in terms of what it means to converge in SOT: if $T_{n}$ is a sequence in $B(\cH)$, then $T_{n}\to T$ in SOT if and only if $\|(T_{n}-T)\xi\|\to_{n\to\infty}0$ for all $\xi\in \cH$ (a similar formulation works more generally for nets). We also say $T_{n}\to T$ strongly if $T_{n}\to T$ in SOT. While the SOT is not metrizable on $B(\cH)$ if $\cH$ is infinite-dimensional, it is on the unit ball by  \cite[Proposition IX.1.3]{Conway}. Thus on norm bounded sets we can detect continuity, closures etc. via convergence of sequences.

\begin{defn}
 Let $\cH$ be a Hilbert space. A \emph{von Neumann algebra} is a $*$-subalgebra $M\subseteq B(\cH)$, with $1\in M$ and which is closed in the strong operator topology.    
\end{defn}

For a subset $E\subseteq B(\cH)$, we let $W^{*}(E)$ be the von Neumann algebra \emph{generated by $E$}
(i.e. the smallest von Neumann subalgebra of $B(\cH)$ containing $E$). Key examples of von Neumann algebras include $B(\cH)$ as well as $L^{\infty}(Y,\nu)$ acting on $L^{2}(Y,\nu)$ via multiplication operators, where $(Y,\nu)$ is a sigma-finite measure space (see \cite[Example IX.7.2]{Conway}).   A core example of interest for us is the following.

Let $\cE$ be a discrete, measure-preserving equivalence relation on a standard sigma-finite space $(Y,\nu)$. For $f\in L^{\infty}(Y,\nu)$ we define $M_{f}\in B(L^{2}(\cE))$ by $(M_{f}\xi)(x,y)=f(x)\xi(x,y)$.
We let $L(\cE)$ be the von Neumann algebra generated by $\lambda_{\cE}([\cE])\cup \{M_{f}:f\in L^{\infty}(Y,\nu)\}$. We typically identify $[\cE]$, $L^{\infty}(Y,\nu)$ with their image inside $L(\cE)$. We also sometimes write $\lambda$ instead of $\lambda_{\cE}$ if it is clear from context.  Note that if $\Gamma\leq [\cE]$ generates $\cE$, and $\sigma\in [[\cR]]$, then modulo null sets we may write $\ran(\sigma)=\bigsqcup_{\gamma\in \Gamma}E_{\gamma}$, where $E_{\gamma}\subseteq\{y\in Y:\sigma^{-1}(y)=\gamma^{-1}(y)\}$. Then $\lambda(\sigma)=\sum_{\gamma\in \Gamma}M_{1_{E_{\gamma}}}\lambda(\gamma)$ with the sum converging in SOT. Thus in this case we have $L(\cE)=W^{*}(L^{\infty}(Y,\nu)\cup \Gamma)$. We will mention an explicit description of the operators in $L(\cE)$ given in the next subsection, which will require the theory of direct integrals. 

For a von Neumann algebra $M\subseteq B(\cH)$, we define its \emph{commutant} \[M'=\{T\in B(\cH):Tx=xT \textnormal{ for all $x\in M$}\}.\] Then $M'$ is a von Neumann algebra, and von Neumann's bicommutant theorem shows that $M=M''$ (see \cite[Theorem IX.6.4]{Conway}). 
Recall that an operator $x\in B(\cH)$ is positive, denoted $x\geq 0$, if $x=x^{*}$ and $\ip{x\xi,\xi}\geq 0$ for all $\xi\in \cH$. A linear map $T\colon M\to N$ is \emph{positive} if $x\geq 0$ implies $T(x)\geq 0$. A \emph{state} on $M$ is a positive, linear functional $\varphi$ with $\varphi(1)=1$.
An important continuity notion for linear maps $T\colon M\to N$ between von Neumann algebras is the \emph{normality} of $T$. A linear functional $\varphi\colon M\to \C$ is \emph{normal}, if $\varphi|_{\{x\in M:\|x\|\leq 1\}}$ is SOT-continuous. We say that a linear map $T\colon M\to N$ is \emph{normal} if $\varphi\circ T$ is normal for all normal $\varphi\colon N\to \C$. Note that is implied by saying that $T|_{\{x\in M:\|x\|\leq 1\}}$ is SOT-SOT continuous. These two concepts are  equivalent if $T$ is a $*$-homomorphism \cite[Proposition 7.1.15]{KadisonRingroseII}, but not in general. However, in all cases in this paper, we will check normality by showing the stronger condition that $T|_{\{x\in M:\|x\|\leq 1\}}$ is SOT-SOT continuous. 

A good example for intuition on normality is the following. If $M=L^{\infty}(X,\mu)$, then a \emph{normal} state precisely corresponds to a probability measure on $X$ which is absolutely continuous with respect to $\mu$, whereas a general state on $M$ is a \emph{finitely additive} probability measure which is absolutely continuous with respect to $\mu$. 

In the special case of a probability measure-preserving equivalence relation, the algebra $L(\cE)$ has additional structure. 

\begin{defn}
A \emph{tracial von Neumann algebra} is a pair $(M,\tau)$ where $M$ is a von Neumann algebra and $\tau\colon M\to \C$ is a linear functional which is
\begin{itemize}
    \item (tracial) $\tau(xy)=\tau(yx)$ for $x,y\in M$,
    \item (a state) $x\geq 0$ implies that $\tau(x)\geq 0$, and $\tau(1)=1$,
    \item (faithful) $\tau(x^{*}x)=0$ if and only if $x=0$,
    \item (normal) $\tau|_{\{x\in M:\|x\|\leq 1\}}$ is SOT continuous. 
\end{itemize}
\end{defn}

The essential idea is that $(M,\tau)$ should be thought of as a ``noncommutative probability space", e.g. the case where $M$ is commutative corresponds exactly to a probability space $(X,\mu)$. 

For a discrete, pmp equivalence relation $\cR$ over a standard probability space $(X,\mu)$, we have a faithful, normal, tracial state $\tau\colon L(\cR)\to \C$ given by $\tau(x)=\ip{x 1_{\Delta},1_{\Delta}}$, where $\Delta=\{(x,y)\in \cR:y=x\}$. 

For a tracial von Neumann algebra $(M,\tau)$, we have an inner product on $M$ given by $\ip{x,y}=\tau(y^{*}x)$, and  we set $\|x\|_{2}=\ip{x,x}^{1/2}$. We let $L^{2}(M,\tau)$ be the completion of $M$ under this inner product. For $x,y\in M$ we have 
\[\|xy\|_{2}\leq \min(\|x\|\|y\|_{2},\|x\|_{2}\|y\|),\]
see \cite[V.2, equation (8)]{Taka}.
We thus have  natural left and right actions $\rho,\lambda(x)$ of $x\in M$ on $L^{2}(M,\tau)$ which are defined densely by $\lambda(x)y=xy$, and $\rho(x)y=yx$, for $y\in M$. Note that $\rho$ is a $*$-anti-homomorphism.  Both $\lambda$ and $\rho$ are normal \cite[Theorem 2.22]{Taka}. We usually view $M$ as represented on $L^{2}(M,\tau)$ and ignore whatever Hilbert space $M$ was originally represented on. Thus we will typically use $x$ instead of $\lambda(x)$. One important fact we will use later is that $\rho(M)'=M$ and $M'=\rho(M)$ when acting on $L^{2}(M,\tau)$, see \cite[Theorem 2.22]{Taka}. In the case of $L(\cR)$, we have a natural identification of $L^{2}(L(\cR))$ with $L^{2}(\cR)$ which sends $x\in L(\cR)$ to $x1_{\Delta}$. In this case, the right action $\rho$ can be concretely identified on $L^{2}(\cR)$ as:
\[(\rho(\gamma)\xi)(x,y)=\xi(x,\gamma(y)) \textnormal{ for $(x,y)\in \cR,\gamma\in [\cR],\xi\in L^{2}(\cR)$},\]
\[(\rho(f)\xi)(x,y)=f(y)\xi(x,y) \textnormal{ for $f\in L^{\infty}(X,\mu),\xi\in L^{2}(\cR)$}.\]

\subsection{Preliminaries on direct integrals}

\begin{defn}Let  $(Y,\nu)$ be a $\sigma$-finite measure space, then a \emph{measurable field of Hilbert spaces over $Y$} is a family $(\mathcal{H}_{y})_{y\in Y}$ of separable Hilbert spaces, together with a family $\Meas(\mathcal{H}_{y})\subseteq \prod_{y\in Y}\mathcal{H}_{y}$  so that:
\begin{itemize}
\item for every $(\xi_{y})_{y},(\eta_{y})_{y}\in \Meas(\mathcal{H}_{x})$ we have that $y\mapsto \ip{\xi_{y},\eta_{y}}$ is measurable,
\item if $\eta=(\eta_{y})_{y\in Y}\in \prod_{y\in Y}\mathcal{H}_{y}$ and $y\mapsto \ip{\xi_{y},\eta_{y}}$ is measurable for all $\xi=(\xi_{y})_{y}\in \Meas(\mathcal{H}_{y}),$ then $\eta\in \Meas(\mathcal{H}_{y}),$
\item there is a sequence $(\xi^{(n)})_{n=1}^{\infty}$ with $\xi^{(n)}=(\xi^{(n)}_{y})_{y\in Y}$ in $\Meas(\mathcal{H}_{y})$ so that $\mathcal{H}_{y}=\overline{\Span\{\xi^{(n)}_{y}:n\in \N\}}$ for almost every $y\in Y.$
\end{itemize}
The direct integral, denoted $\int_{Y}^{\oplus} \mathcal{H}_{y}\,d\nu(y),$ is defined to be all $\xi\in \Meas(\mathcal{H}_{y})$ so that $\int_{Y}\|\xi_{y}\|^{2}\,d\nu(y)<\infty,$ where we identify two elements of $\Meas(\mathcal{H}_{y})$ if they agree outside a set of measure zero. We put an inner product on $\int_{Y}^{\oplus}\mathcal{H}_{y}\,d\nu(y)$ by
\[\ip{\xi,\eta}=\int_{Y}\ip{\xi_{y},\eta_{y}}\,d\nu(y),\]
and this gives $\int_{Y}^{\oplus}\mathcal{H}_{y}\,d\nu(y)$ the structure of a Hilbert space.
Suppose that $(T_{y})_{y\in Y}\in \prod_{y\in Y}B(\cH_{y})$. We then say that $(T_{y})_{y\in Y}$ is a \emph{measurable field of operators} if for any $(\xi_{y})_{y\in Y}\in \Meas(\mathcal{H}_{y})$, we have that $(T_{y}\xi_{y})_{y\in Y}\in \Meas(\mathcal{H}_{y})$. Suppose that $(T_{y})_{y\in Y}$ is a measurable field of operators and $\esssup_{y\in Y}\|T_{y}\|<+\infty$. We may then define $\int_{Y}^{\bigoplus}T_{y}\,d\nu(y):=T$ by $(T\xi)_{y}=T_{y}\xi_{y}$.
\end{defn}

By \cite[Section IV.7 Equation (6)]{Taka}, it follows that 
\[\left\|\int_{Y}^{\bigoplus}T_{y}\,d\nu(y)\right\|=\esssup_{y\in Y}\|T_{y}\|.\]

We list a few examples that are most relevant for us.
\begin{example}\label{example: equivalence reln direct integral}
 Let $\cE$ be a discrete, measure-preserving equivalence relation on a standard $\sigma$-finite space $(Y,\nu)$. For $y\in Y$, set $\cH_{y}=\ell^{2}([y]_{\cE})$. We turn $\cH_{y}$ into a measurable field by declaring that $(\xi^{y})_{y\in Y}\in \Meas(\cH_{y})$ if $(x,y)\mapsto \xi^{y}(x)$ is $\nu_{\cE}$-measurable. In this case, we have a concrete isomorphism
 \[L^{2}(\cE)\cong \int_{Y}^{\bigoplus}\ell^{2}([y]_{\cE})\,d\nu(y),\]
 given by sending $\xi\in L^{2}(\cE)$ to $(\xi^{y})_{y\in Y}$ where $\xi^{y}(x)=\xi(x,y)$.
\end{example}

\begin{example}
 Let $\cS\leq \cR$ be discrete, probability measure-preserving equivalence relations on a standard probability space $(X,\mu)$. For $x\in X$, we let $\cH_{x}=\ell^{2}([x]_{\cR}/\cS)$. We turn $\cH_{x}$ into a measurable field by saying that $(\xi_{x})_{x\in X}\in \Meas(\ell^{2}([x]_{\cR}/\cS))$ if for every $\gamma\in [\cR]$ the map $x\mapsto \xi_{x}([\gamma(x)]_{\cS})$ is $\mu$-measurable. In this case, we have a concrete isomorphism
 \[L^{2}(\cR/\cS)\cong \int_{X}\ell^{2}([x]_{\cR}/\cS)\,d\mu(x)\]
 by sending $\xi\in L^{2}(\cR/\cS)$ to the field $(\xi_{x})_{x\in X}$ given by $\xi_{x}(c)=\xi(x,c)$. This example also appears in \cite[discussion following Definition 3.1]{AFH}.
\end{example}

A toy example can be given by choosing $X=\{1,\cdots,n\}$ for some $n\in \N$ and letting $\mu$ be the uniform measure on $\{1,\cdots,n\}$. If we then set $\cH_{j}=\C$, we have that $\int_{\{1,\cdots,n\}}^{\bigoplus}\C\,d\mu(j)=\ell^{2}(n,\mu)$. In this case, an operator on $\ell^{2}(n,\mu)$ given as a direct integral exactly corresponds to a diagonal matrix. In general, we think of operators given in direct integral form as ``continuously diagonal operators". E.g. a matrix is diagonal if and only if it commutes with all diagonal operators. Similarly, an operator on a direct integral space $\int_{X}^{\bigoplus}\cH_{x}\,d\mu(x)$ can be written as a direct integral of operators if and only if it commutes with all multiplication operators $\int_{X}^{\bigoplus}f(x)\,d\mu(x)$ where $f\in L^{\infty}(X,\mu)$ (see \cite[Corollary  IV.8.16]{Taka}).

The following is well known, but since we are unable to find a proof in the literature, we include it for completeness.
\begin{prop}\label{prop: folklore subsequences direct integrals}
Let $(Y,\nu)$ be a $\sigma$-finite measure space, and let $(\cH_{y})_{y\in Y}$ be a measurable field of Hilbert spaces and set $\cH=\int_{Y}^{\bigoplus}\cH_{y}\,d\nu(y)$. Let $T_{n}=\int_{Y}^{\bigoplus}T_{n,y}\,d\nu(y)\in B(\cH)$ and suppose that $T_{n}$ converges strongly as $n\to\infty$ to $T=\int_{Y}^{\bigoplus}T_{y}\in B(\cH)$. Then there is a strictly increasing sequence of integers $(n_{k})_{k}$ so that for almost every $y\in Y$ we have $T_{n_{k},y}\to_{k\to\infty}T_{y}$ strongly.     
\end{prop}

\begin{proof}
We first handle the case where $\nu$ is a probability measure.
A common application of the principle of uniform boundedness shows that there is a $C>0$ with $\|T_{n}\|\leq C$ for all $n\in \N$. Hence for almost every $y\in Y$  we have that $\|T_{n,y}\|\leq C$ for all $n\in \N$.
Let $\xi^{(j)}$ be a sequence of measurable vector fields with $\overline{\Span\{\xi^{(j)}_{y}:j\in \N\}}=\cH_{y}$ for almost every $y\in Y$. We may, and will, assume that $\|\xi^{(y)}_{j}\|=1$ almost everywhere.
Since $\|(T_{n}-T)\xi^{(j)}\|^{2}=\int_{Y}\|(T_{n,y}-T_{y})\xi_{j}^{(y)}\|^{2}\,d\nu(y)\to_{n\to\infty}0$ for every $j\in \N$, we may choose a strictly increasing sequence $(n_{k})_{k=1}^{\infty}$
with
\[\nu\left(\bigcup_{j=1}^{k}\{y:\|T_{n_{k},y}\xi^{(j)}_{y}-T_{y}\xi^{(j)}_{y}\|\geq 2^{-k}\}\right)<2^{-k}.\]
Thus Borel-Cantelli implies that for almost every $y\in Y$ we have that:
\begin{itemize}
    \item $\|T_{n,y}\|\leq C$, for all $n\in \N$,
    \item $\|(T_{n_{k},y}-T_{y})\xi^{(j)}_{y}\|\to_{k\to\infty}0$, for all $j\in \N$,
    \item $\Span\{\xi^{(j)}_{y}:j\in \N\}$ is norm dense in $\cH_{y}$.
\end{itemize}
If $y$ is such that the first bullet point holds, then $\{\zeta\in \cH_{y}:\|T_{n_{k},y}\zeta-T_{y}\zeta\|\to_{k\to\infty}0\}$ is a closed linear subspace of $\cH_{y}$. So if $y$ is such that all bullet points above hold, then $T_{n_{k},y}\to T_{y}$ strongly. This handles the case that $\nu$ is a probability measure.

If $\nu$ is not a probability measure, then since $\nu$ is $\sigma$-finite, we may choose a probability measure $\lambda$ on $Y$ which is mutually absolutely continuous with $\nu$. Set $\widetilde{\cH}=\int_{Y}\cH_{y}\,d\lambda(y)$, and define a unitary $U\colon \cH\to \widetilde{\cH}$ by $(U\xi)_{y}=\frac{d\nu}{d\lambda}(y)^{1/2}\xi_{y}$. If $T_{n}\to T$ strongly as in the statement of the proposition, then $UT_{n}U^{*}\to UTU^{*}$ strongly, and $UT_{n}U^{*}=\int_{Y}^{\bigoplus}T_{n,y}\,d\lambda(y)$, $UTU^{*}=\int_{Y}^{\bigoplus}T_{y}\,d\lambda(y)$. Hence by the case of a probability measure, we may find a subsequence $T_{n_{k}}$ so that for $\lambda$-almost every $y$ we have $T_{n_{k},y}\to T_{y}$ strongly. Since $\lambda$ and $\nu$ are mutually absolutely continuous we have that for $\nu$-almost every $y$ that $T_{n_{k},y}\to T_{y}$ strongly.

\end{proof}

\subsection{The basic construction and von Neumann algebras of relations}\label{sec: basic construction}

Recall that given a discrete, measure-preserving equivalence relation  $\cE$ on a standard sigma-finite space $(Y,\nu)$, then by Example \ref{example: equivalence reln direct integral} we may view $L^{2}(\cE)$ as the direct integral of $\int \ell^{2}([y]_{\cE})\,d\nu(y)$.

\begin{thm}[Corollary XIII.2.8 of \cite{TakesakiIII}] \label{thm: direct integral rep of equiv ops Takesaki}
Let $\cE$ be a discrete, measure-preserving equivalence relation on a standard sigma-finite space $(Y,\nu)$. Given $T\in B(L^{2}(\cE))$, we have that $T\in L(\cE)$ if and only if we may write $T=\int_{Y}^{\bigoplus}T^{y}\,d\nu(y)$ and $T^{y_{1}}=T^{y_{2}}$ for almost every $(y_{1},y_{2})\in \cE$. 
\end{thm}
In the special case where $(Y,\nu)$ is a probability measure, we will prove something more general in Proposition \ref{prop: thanks FSZ}.

An important ingredient of coamenability for von Neumann algebras is Jones' basic construction. Let $(M,\tau)$ be tracial von Neumann algebra, and let $N\leq M$ be a von Neumann algebra. Let $e_{N}\in B(L^{2}(M))$ be the projection onto $L^{2}(N)$. The \emph{basic construction of $M$ over $N$}, denoted $\ip{M,e_{N}}$ is $W^{*}(M\cup\{e_{N}\})$. We typically call $e_{N}$ the \emph{Jones projection}. By \cite[Lemma XIX.2.12]{TakesakiIII} we have that $\ip{M,e_{N}}=\rho(N)'$.

Given an inclusion $\cS\leq \cR$ of discrete, probability measure-preserving equivalence relations on $(X,\mu)$, recall that we define a new equivalence relation $\widehat{\cS}$ on $\cR/\cS$ by saying that $((x,c),(x',c'))\in \widehat{\cS}$ if $(x,x')\in \cR$ and $c=c'$. We often abuse notation and regard $\widehat{\cS}=\{(x,x',c):(x,x')\in \cR, c\in [x]_{\cR}/\cS\}$.
The main goal of this section is to identify the basic construction associated to an inclusion $\cS\leq \cR$ with the von Neumann algebra of the relation $\widehat{\cS}$. The earliest reference claiming this fact to be true we could find was \cite[Section 1]{FSZ}, indicating that a proof would be given in \cite{SuthBasic}. Being unable to find \cite{SuthBasic} in the literature, we have elected to give a proof here. 
Our starting point is a description of operators in the basic construction as direct integrals of operators which are almost surely constant on the subrelation.

\begin{prop}\label{prop: thanks FSZ}
Let $\cS\leq \cR$ be discrete, probability measure-preserving equivalence relations on a standard probability space $(X,\mu)$. View $L^{2}(\cR)=\int_{X}^{\bigoplus}\ell^{2}([y]_{\cR})\,d\mu(y)$ by viewing $\xi\in L^{2}(\cR)$ as $(\xi^{y})_{y}$ where $\xi^{y}(x)=\xi(x,y)$. Given $T\in B(L^{2}(\cR))$, we have that $T\in \ip{L(\cR),e_{L(\cS)}}$ if and only if we may write $T=\int_{X}^{\bigoplus}T^{y}\,d\mu(y)$ where $T^{y}\in B(\ell^{2}([y]_{\cR}))$ and for almost every $(y_{1},y_{2})\in \cS$ we have $T^{y_{1}}=T^{y_{2}}$. 

\end{prop}

\begin{proof}
By \cite[Lemma XIX.2.12]{TakesakiIII} we have that $\ip{L(\cR),e_{L(\cS)}}$ is the commutant of the right action of $L(\cS)$. Note that for $f\in L^{\infty}(X,\mu)$ we have $\rho(f)=\int_{Y}^{\bigoplus}f(y)\,d\mu(y)$. Hence, if $T\in \ip{L(\cR),e_{L(\cS)}}$ then it commutes with all $\rho(f)$ and hence, by \cite[Corollay IV.8.16]{Taka} we may find a measurable field $(T^{y})_{y}$ with $T^{y}\in B(\ell^{2}([y]_{\cR}))$ and  $T=\int_{Y}^{\bigoplus}T^{y}\,d\mu(y)$. If $\gamma\in [\cS]$, then $(\rho(\gamma)\xi)^{y}=\xi^{\gamma(y)}$. Thus $\rho(\gamma)T\rho(\gamma^{-1})=\int_{X}^{\bigoplus}T^{\gamma(y)}\,d\mu(y)$. Choosing a countable subgroup of $[\cS]$ which generates $\cS$, we thus see that for almost every $y\in X$ we have $T^{z}=T^{y}$ for all $z\in [y]_{\cS}$. Thus $T^{y_{1}}=T^{y_{2}}$ for almost every $(y_{1},y_{2})\in \cS$.

Conversely, suppose that $T=\int_{X}^{\bigoplus}T^{y}\,d\mu(y)$ and $T^{y_{1}}=T^{y_{2}}$ for almost every $(y_{1},y_{2})\in \cS$. Then the same calculations as above show that $T$ commutes with $\rho(L^{\infty}(X,\mu)\cup [\cS])$, and so by normality of $\rho$ we have that $T$ commutes with $\rho(\overline{W^{*}(L^{\infty}(X,\mu)\cup [\cS])}^{SOT})=\rho(L(\cS))$. Thus $T\in \rho(L(\cS))'=\ip{L(\cR),e_{L(\cS)}}$.

\end{proof}

The crux of our argument connecting coamenability defined via von Neumann algebras to the versions of coamenability given in Theorem \ref{thm: big TFAE thm} will be identifying the von Neumann algebra of the basic construction with the von Neumann algebra of the relation $\widehat{\cS}$. It will also be important to identify the image of  $[\cR]$ under this identification, as well as the Jones projection $e_{L(\cS)}$.

Note that if $T\in L(\cR)$, then we may write $T=\int_{X}^{\bigoplus}T^{y}\,d\mu(y)$, where $T^{y_{1}}=T^{y_{2}}$ for almost every $(y_{1},y_{2})\in \cR$. We may then define $\iota(T)=\int_{\cR/\cS}^{\bigoplus}T^{y}\,d\mu_{\cR/\cS}(y,c)\in B(L^{2}(\widehat{\cS}))$, and observe that $\iota(T)\in L(\widehat{\cS})$ by Proposition \ref{prop: thanks FSZ}. Let $(\phi_{i})_{i}$ be choice functions for $\cR/\cS$, we have that 
\[\|\iota(T)\xi\|_{2}^{2}=\sum_{i\in I}\|T\xi_{i}\|_{2}^{2},\]
where $(\xi_{i})^{y}(x)=\xi(x,y,[\phi_{i}(y)]_{\cS})1_{\dom(\phi_{i})}(y)$. Thus $\iota$ is a normal $*$-homomorphism.
Note that $\iota(\lambda_{\cR}(\gamma))=\lambda_{\widehat{\cS}}(\gamma)$ for $\gamma\in [\cR]$. 
We now show that $L(\widehat{\cS})$ is identified with $\ip{L(\cR),e_{L(\cS)}}$ and that the identification sends  the Jones projection  with the operator
$\widehat{e}\in B(L^{2}(\widehat{\cS}))$ given by $(\widehat{e}\xi)(x,y,c)=1_{c}(x)\xi(x,y,c)$. 

\begin{prop}\label{prop: thanks S 2}
Let $\cS\leq \cR$ be discrete, probability measure-preserving equivalence relations on a standard probability space $(X,\mu)$. Let $\widehat{\cS}$ be the equivalence relation on $\cR/\cS$ define by $(x,c)\thicksim (x',c')$ if and only if $c=c'$ and $(x,x')\in \cR$. Then there is a unique, normal $*$-isomorphism $\Phi\colon \ip{L(\cR),e_{L(\cS)}}\to L(\widehat{\cS})$ such that $\Phi|_{L(\cR)}=\iota$, and $\Phi(e_{L(\cS)})=\widehat{e}$. 
    
\end{prop}

\begin{proof}
We regard $L^{2}(\cR)=\int_{X}^{\oplus}\ell^{2}([y]_{\cR})\,d\mu(y)$ by identifying $\xi\in L^{2}(\cR)$ with the measurable field $(\xi^{y})_{y}$ where $\xi^{y}(x)=\xi(x,y)$. Similarly, we identify $L^{2}(\widehat{\cS})=\int_{\cR/\cS}^{\oplus}\ell^{2}([y]_{\cR})\,d\mu_{\cR/\cS}(y,c)$ where $\xi\in L^{2}(\widehat{\cS})$ is identified with the measurable field $(\xi^{(y,c)})_{(y,c)}$ where $(\xi^{(y,c)})(x)=\xi(x,y,c)$. 
Given $T\in \ip{L(\cR),e_{L(\cS)}}$ we may write $T=\int_{X}^{\oplus}T^{y}\,d\mu(y)$ where for almost every $(y_{1},y_{2})\in\cS$ we have $T^{y_{1}}=T^{y_{2}}$. We define $\Phi(T)\in L(\widehat{\cS})$ by $\Phi(T)^{(y,c)}=T^{z}$ where $z$ is any element of $c$. Since $y\mapsto T^{y}$ is almost surely $\cS$-invariant, we have that $\Phi(T)^{(y,c)}$ is well-defined almost everywhere.
To show that $\Phi(T)^{(y,c)}$ is a measurable field of operators,
suppose that $(\xi^{(y,c)})_{(y,c)},(\eta^{(y,c)})_{(y,c)}\in \Meas(\ell^{2}([(y,c)]_{\widehat{\cS}}))$. Then it suffices to show that $(y,c)\mapsto \ip{\Phi(T)^{(y,c)}\xi^{(y,c)},\eta^{(y,c)}}$ is measurable.  By Proposition \ref{prop:measurability exercise} it suffices to show that for any $\gamma\in [\cR]$, we have that $y\mapsto \ip{\Phi(T)^{(y,[\gamma(y)]_{\cS})}\xi^{(y,[\gamma(y)]_{\cS})},\eta^{(y,[\gamma(y)]_{\cS})}}$ is measurable. But 
\[\ip{\Phi(T)^{(y,[\gamma(y)]_{\cS})}\xi^{(y,[\gamma(y)]_{\cS})},\eta^{(y,[\gamma(y)]_{\cS})}}=\ip{T^{\gamma(y)}\xi^{(y,[\gamma(y)]_{\cS})},\eta^{(y,[\gamma(y)]_{\cS})}}=f(\gamma(y)),\] where $f(y)=\ip{T^{y}\xi^{(\gamma^{-1}(y),[y]_{\cS})},\eta^{(\gamma^{-1}(y),[y]_{\cS})}}$. Since $\gamma$ is measure-preserving, we have that $f\circ \gamma$ is measurable if $f$ is.  Since $T^{y}$ is a measurable field of operators, it then suffices to show that $\xi_{\gamma},\eta_{\gamma}$ defined by $\xi_{\gamma}^{y}=\xi^{(\gamma^{-1}(y),[y]_{\cS})}$ and $\eta_{\gamma}^{y}=\eta^{(\gamma^{-1}(y),[y]_{\cS})}$ are measurable vector fields. But $\xi_{\gamma}^{y}(x)=\xi^{(\gamma^{-1}(y),[y]_{\cS})}(x)$, and so this follows from the measurability of $(x,y,c)\mapsto \xi^{(y,c)}(x)$, as well as the fact that $\gamma$ is measure preserving. The same applies to measurability of $\eta_{\gamma}^{y}$.
Having shown measurability of $\Phi(T)$, we see directly that $\Phi(T)^{(y,c)}$ is almost surely constant on $\widehat{\cS}$-classes, so $\Phi(T)\in L(\widehat{\cS})$.  

For $\widehat{T}\in L(\widehat{\cS})$, we may write $\widehat{T}=\int_{\cR/\cS}\widehat{T}^{(y,c)}\,d\mu_{\cR/\cS}(y,c)$ where $\widehat{T}^{(y_{1},c)}=\widehat{T}^{(y_{2},c)}$ for almost every $(y_{1},y_{2})\in \cR$. We define $\Psi(\widehat{T})\in \ip{L(\cR),e_{L(\cS)}}$ by $\Psi(\widehat{T})^{y}=\widehat{T}^{(y,[y]_{\cS})}$. To see that $\Psi(\widehat{T})^{y}$ is a measurable field of operators, let $(\xi^{y})_{y},(\eta^{y})_{y}\in \Meas(\ell^{2}([y]_{\cR}))$. Then
\[\ip{\Psi(\widehat{T})^{y}\xi^{y},\eta^{y}}=\ip{\widehat{T}^{(y,[y]_{\cS})}\widetilde{\xi}^{(y,[y]_{\cS})},\widetilde{\eta}^{(y,[y]_{\cS})}},\]
where $\widetilde{\xi}^{(y,c)}=1_{c}(y)\xi^{y}$,$\widetilde{\eta}^{(y,c)}=1_{c}(y)\eta^{y}$. Since $\widehat{T}^{(y,c)}$ is a measurable field of operators, it thus suffices to show that $\widetilde{\xi}^{(y,c)},\widetilde{\eta}^{(y,c)}$ are measurable vector fields. But $\widetilde{\xi}^{(y,c)}(x)=1_{c}(y)\xi^{y}(x)$, which is measurable by measurability of the vector field $(\xi^{y})$, and the same argument applies to $\widetilde{\eta}^{y}$. 
Moreover, $\Psi(T)^{y}$ is almost surely constant on $\cS$-classes, so that $\Psi(T)\in \ip{L(\cR),e_{L(\cS)}}$, by Proposition \ref{prop: thanks FSZ}.

Note that $\Psi(\Phi(T))^{y}=\Phi(T)^{(y,[y]_{\cS})}=T^{y}$ for $T\in L(\cS)$ and almost every $y\in X$. For $\widehat{T}\in \ip{L(\cR),e_{L(\cS)}}$ and almost every $(y,c)\in \cR/\cS$ we have $\Phi(\Psi(\widehat{T}))^{(y,c)}=\Psi(\widehat{T})^{z}=\widehat{T}^{(z,c)}$, where $z\in c$. But since $\widehat{T}$ is almost surely constant on $\widehat{\cS}$-classes, we have that $\widehat{T}^{(z,c)}=\widehat{T}^{(y,c)}$. Thus $\Phi,\Psi$ are $*$-homomorphisms which are mutually inverse to each other. It remains to check that $\Phi,\Psi$ are normal. One can do this by appealing to general theory which says that a $*$-isomorphism between von Neumann algebras is automatically normal (see \cite[Corollary 7.1.16]{KadisonRingroseII}), however we will check this directly. 

Since $\Phi,\Psi$ are $*$-homomorphisms, it suffices to show that they are SOT-SOT continuous on the unit ball. Suppose that $T_{n}\in \ip{L(\cR),e_{L(\cS)}}$ with $\|T_{n}\|\leq 1$ and $T_{n}=\int_{X}T_{n}^{y}\,d\mu(y)$ and $T_{n}\to T=\int_{X}T^{y}\,d\mu(y)\in \ip{L(\cR),e_{L(\cS)}}$ strongly. To show that $\Phi(T_{n})\to \Phi(T)$ strongly, it suffices to show that if $T_{n_{k}}$ is any subsequence then there is a further subsequence $T_{n_{k_{j}}}$ with $\Phi(T_{n_{k_{j}}})\to \Phi(T)$ strongly. So fix a subsequence $T_{n_{k}}$ of $T_{n}$. Passing to a further subsequence and applying Proposition \ref{prop: folklore subsequences direct integrals} we may assume that $T_{n_{k}}^{y}\to T^{y}$ almost surely. Thus we may find a conull, $\cR$-invariant subset $X_{0}\subseteq X$ so that for all $y\in X_{0}$ and all $k\in \N$:
\begin{itemize}
    \item $T_{n_{k}}^{y}\to T^{y}$ strongly,
    \item $\|T_{n_{k}}^{y}\|\leq 1,$
    \item $T_{n_{k}}^{y}=T_{n_{k}}^{z}$ for all $z\in [y]_{\cR}$.
\end{itemize} 
Given $\xi=(\xi^{(y,c)})_{(y,c)}\in L^{2}(\widehat{S})$ we have that $\|(\Phi(T_{n_{k}})^{(y,c)}-\Phi(T)^{(y,c)})\xi^{(y,c)}\|_{2}^{2}\leq 4\|\xi^{(y,c)}\|_{2}^{2}$, for almost every $(y,c)$. Note that $4\|\xi^{(y,c)}\|_{2}^{2}\in L^{1}(\cR/\cS).$ Further by $\cR$-invariance of $X_{0}$ and choice of $X_{0}$ we have $\|(\Phi(T_{n_{k}})^{(y,c)}-\Phi(T)^{(y,c)})\xi^{(y,c)}\|_{2}\to 0$ for all $y\in X_{0}$ and $c\in [y]_{\cR}/\cS$. Thus the dominated convergence theorem implies that $\|(\Phi(T_{n_{k}})-\Phi(T))\xi\|_{2}\to 0$. This establishes that $\Phi$ is SOT-continuous on the unit ball of $\ip{L(\cR),e_{L(\cS)}}$. The proof that $\Psi$ is SOT-continuous on the unit ball of $L(\widehat{\cS})$ is similar.  The facts that $\Phi|_{L(\cR)}=\iota$ and $\Phi(e_{L(\cS)})=\widehat{e}$ are direct computations.

\end{proof}

\subsection{Coamenability}


We now have sufficient background to recall the notion of coamenability for von Neumann algebras and describe what it means precisely for equivalence relations. 

If $\varphi\colon M\to \C$ is a state, then we define the \emph{centralizer} of $\varphi$ to be $M^{\varphi}=\{x\in M:\varphi(yx)=\varphi(xy) \textnormal{ for all $y\in M$}\}$.  If $N\leq M$, we say that $\varphi$ is \emph{$N$-central} if $N\subseteq M^{\varphi}$. 

\begin{defn}[Definition 3.2.1 and Theorem 3.2.3 of \cite{PopaCorr}, see also Definition 4 and Proposition 5 of \cite{MonodPopa}, as well as \cite{ADAmenableCorr}]
Let $(M,\tau)$ be a tracial von Neumann algebra and $Q\leq M$. We say that $Q$ is \emph{coamenable in $M$} if there is a  $Q$-central state $\varphi\colon \ip{M,e_{Q}}\to \C$ with $\varphi|_{M}=\tau$. We call such a $\varphi$ a \emph{hypertrace}.
\end{defn}

The idea here is that one should think of a hypertrace as being an invariant mean (e.g. it turns out that $\varphi$ is $Q$-central if and only if it is invariant under conjugation by unitaries in $Q$). The aim of this section is to show that $\cS$ is coamenable in $\cR$ if and only if $L(\cS)$ is coamenable in $L(\cR)$. 

Suppose that $\cE$ is a discrete, measure-preserving relation on a $\sigma$-finite measure space $(Y,\nu)$. We define $E_{L^{\infty}(Y)}\colon L(\cE)\to L^{\infty}(Y,\nu)$ by:
\[E_{L^{\infty}(Y)}(T)(y)=\ip{T^{y}\delta_{y},\delta_{y}},\]
where $T=\int_{Y}^{\bigoplus}T^{y}\,d\nu(y)$ is as in Theorem \ref{thm: direct integral rep of equiv ops Takesaki}.
The following is a direct computation.

\begin{lem}\label{lem:expectations restrict well}
Let $\cS\leq \cR$ be an inclusion of discrete, probability measure-preserving equivalence relations on a standard probability space $(X,\mu)$.
Let $E_{L^{\infty}(\cR/\cS)}\colon L(\widehat{\cS})\to L^{\infty}(\cR/\cS)$, $E_{L^{\infty}(X)}\colon L(\cR)\to L^{\infty}(X)$ be the maps described above. Then $E_{L^{\infty}(\cR/\cS)}|_{L(\cR)}=E_{L^{\infty}(X)}$. 
\end{lem}

We also need the following fact. This is a folklore result, but we include a proof for completeness.

\begin{prop}\label{prop:folklore centrality}
Let $(M,\tau)$ be a tracial von Neumann algebra with a normal embedding into a von Neumann algebra $N$. Suppose that $\varphi$ is a state on $N$ with $\varphi|_{M}=\tau$. Then $N^{\varphi}\cap M$ is a von Neumann subalgebra of $M$.
    
\end{prop}

\begin{proof}
Set $A=N^{\varphi}\cap M$. The fact that $A$ is a $*$-subalgebra is an exercise. To show that $A$ is SOT-closed, let $x\in \overline{A}^{SOT}$, and $y\in M$. Fix $\varepsilon>0$. Viewing $1\in L^{2}(M,\tau)$, we see that we may find an $a\in A$ with $\|x-a\|_{2}<\varepsilon$. Define an inner product $\ip{\cdot,\cdot}_{L^{2}(\varphi)}$ on $M$ by $\ip{c,d}_{L^{2}(\varphi)}=\varphi(d^{*}c)$. 
Since $\varphi$ is a state, this is a positive semi-definite inner product and so Cauchy-Schwarz implies:
\[|\varphi(ay)-\varphi(xy)|=|\ip{y,x^{*}-a^{*}}|\leq \|y\|_{L^{2}(\varphi)}\|x^{*}-a^{*}\|_{L^{2}(\varphi)}.\]
Since $\varphi|_{M}=\tau$, we see that 
\[\|x^{*}-a^{*}\|_{L^{2}(\varphi)}=\|x^{*}-a^{*}\|_{L^{2}(\tau)}=\|x-a\|_{L^{2}(\tau)}<\varepsilon,\]
with the second-to-last equality following by traciality. 
Similarly,
\[|\varphi(ya)-\varphi(yx)|\leq \|y\|_{L^{2}(\varphi)}\varepsilon.\]
Since $a\in N\cap M^{\varphi}$, it follows that 
\[|\varphi(xy)-\varphi(yx)|\leq 2\varepsilon\|y\|_{L^{2}(\varphi)}.\]
Since this is true for every $\varepsilon>0$ and $y\in M$ was arbitrary, we see that $x\in A$.

\end{proof}

We can now connect coamenability defined via Theorem \ref{thm: big TFAE thm} to that defined via von Neumann algebras.

\begin{thm}\label{thm: two different version of rel amen are the same}
Let $\cS\leq \cR$ be discrete, probability measure-preserving equivalence relations on $(X,\mu)$. Then $\cS$ is coamenable in $\cR$ if and only if $L(\cS)$ is coamenable in $L(\cR)$.
\end{thm}


\begin{proof}
Throughout, we use the identification $\ip{L(\cR),e_{L(\cS)}}\cong L(\widehat{\cS})$ given by Proposition \ref{prop: thanks S 2}.
Suppose that $L(\cS)$ is coamenable in $L(\cR)$, and
let $\varphi\colon L(\widehat{\cS})\to \C$ be a hypertrace. Set $m=\varphi|_{L^{\infty}(\cR/\cS)}$. Then for $f\in L^{\infty}(\cR/\cS)$ and $\gamma\in [\cR]$ we have
\[m(f\circ \gamma^{-1})=\varphi(\lambda_{\widehat{\cS}}(\gamma)f\lambda_{\widehat{\cS}}(\gamma)^{-1})=\varphi(f)=m(f).\]
For $f\in L^{\infty}(X)$, since $L^{\infty}(X)\subseteq L(\cR)$, we have $m(f)=\varphi(f)=\tau(f)=\int f\,d\mu$. Thus $m$ is a $[\cR]$-invariant mean with $m|_{L^{\infty}(X)}=\int \cdot\,d\mu$. 

Conversely, let $m\colon L^{\infty}(\cR/\cS)\to \C$ be a $[\cR]$-invariant mean with $m|_{L^{\infty}(X)}=\int \cdot\,d\mu$. Define $\varphi\colon L(\widehat{\cS})\to \C$ by $\varphi=m\circ E_{L^{\infty}(\cR/\cS)}$. Then for any $a\in L(\cR)$, Lemma \ref{lem:expectations restrict well} implies
\[\varphi(a)=m(E_{L^{\infty}(\cR/\cS)}(a))=m(E_{L^{\infty}(X)}(a))=\int E_{L^{\infty}(X)}(a)\,d\mu=\tau(a).\]

Thus it suffices to show that $\varphi$ is $L(\cR)$-central. It is direct to show that $\varphi$ is $L^{\infty}(X)$-central. We now check that $\varphi$ is $[\cR]$-central. For this, let $\gamma\in [\cR]$. Then for all $a\in L(\widehat{\cS})$: 
\[E_{L^{\infty}(\cR/\cS)}(\lambda_{\widehat{\cS}}(\gamma)a\lambda_{\widehat{\cS}}(\gamma)^{-1})=E_{L^{\infty}(\cR/\cS)}(a)\circ \gamma^{-1}.\]
Thus $m$ being $[\cR]$-invariant shows that 
\[\varphi(a)=\varphi(\lambda_{\widehat{\cS}}(\gamma)a\lambda_{\widehat{\cS}}(\gamma)^{-1})\]
for all $a\in L(\widehat{\cS})$.
So $L(\widehat{\cS})^{\varphi}\supseteq L^{\infty}(X)\cup [\cR]$. 
By Proposition \ref{prop:folklore centrality}, it follows that $\varphi$ is $L(\cR)$-central. 
    
\end{proof}


\section{General properties of coamenability}\label{sec; general prop}

In this section, we establish general permanence properties of coamenability. 
We start by connecting coamenability of groups to \emph{everywhere} coamenability of relations. To do this it will be helpful to use the following version of the mean ergodic theorem for coamenable inclusions. This result is surely folklore, but we are unable to find a proof in the literature and include it for completeness.

\begin{prop}[Weak Mean Ergodic Theorem for Coamenable Inclusions of Groups] \label{prop: weak mean ergodic theorem}
Let $\Gamma$ be a countable, discrete, group and $\Lambda\leq \Gamma$ be a coamenable subgroup. Let $F_{n}\subseteq \Gamma/\Lambda$ be a F\o lner sequence for the action of $\Gamma$ on $\Gamma/\Lambda$. Suppose that $\pi\colon \Gamma\to \cU(\cH)$, and that $\xi\in \Fix_{\Lambda}(\cH)$. Then
\[\frac{1}{|F_{n}|}\sum_{g\Lambda\in F_{n}}\pi(g)\xi\to_{n\to\infty}P_{\Fix_{\Gamma}(\cH)}(\xi),\]
weakly. Namely,
\[\ip{\frac{1}{|F_{n}|}\sum_{g\Lambda\in F_{n}}\pi(g)\xi,\zeta}\to_{n\to\infty}\ip{P_{\Fix_{\Gamma}(\cH)}(\xi),\zeta}\]
for all $\zeta\in\cH$.
\end{prop}

We remark that, since $\xi$ is $\Lambda$-fixed, if $g_{1}\Lambda=g_{2}\Lambda$, then $\pi(g_{1})\xi=\pi(g_{2})\xi$. Hence we may unambiguously make sense of $\pi(g)\xi$ for $g\Lambda\in \Gamma/\Lambda$.

\begin{proof}
To ease notation, set $\xi_{n}=\frac{1}{|F_{n}|}\sum_{g\Lambda\in F_{n}}\pi(g)\xi$. Since $\ip{\pi(x)v,w}=\ip{v,\pi(x^{-1})w}$ for all $v,w\in \cH$, we have that $\ip{\xi_{n},\zeta}=\ip{\xi,\zeta}$ for all $\zeta\in \Fix_{\Gamma}(\cH)$. Hence it suffices to show that $\ip{\xi_{n},\zeta}\to_{n\to\infty}0$ for all $\zeta\in \Fix_{\Gamma}(\cH)^{\perp}$. From the uniform estimate $\|\xi_{n}\|\leq \|\xi\|$, the set $\{\zeta\in \Fix_{\Gamma}(\cH)^{\perp}:\ip{\xi_{n},\zeta}\to_{n\to\infty}0\}$ is a norm-closed linear subspace. Hence it suffices to show that $\{\zeta\in \Fix_{\Gamma}(\cH)^{\perp}:\ip{\xi_{n},\zeta}\to_{n\to\infty}0\}$ has dense linear span in $\Fix_{\Gamma}(\cH)^{\perp}$.
Note that 
\[\Fix_{\Gamma}(\cH)^{\perp}=\left(\bigcap_{\gamma\in \Gamma}\ker(\pi(\gamma)-1)\right)^{\perp}=\overline{\sum_{\gamma\in \Gamma}\Im(\pi(\gamma)-1)}.\]
Thus it suffices to show that $\ip{\xi_{n},(\pi(\gamma)-1)\eta}\to_{n\to\infty}0$ for all $\eta\in \cH$,$\gamma\in\Gamma$. But:
\[|\ip{\xi_{n},(\pi(\gamma)-1)\eta}|=|\ip{(\pi(\gamma^{-1})-1)\xi_{n},\eta}|\leq \|(\pi(\gamma^{-1})-1)\xi_{n}\|\|\eta\|\leq \|\xi\|\|\eta\|\frac{|\gamma^{-1} F_{n}\Delta F_{n}|}{|F_{n}|}\to_{n\to\infty}0.\]
    
\end{proof}

\begin{prop}\label{prop: group case is an iff}
Let $\Lambda\leq\Gamma$ be discrete groups, and $\Gamma\actson (X,\mu)$ a probability measure-preserving action. 
\begin{enumerate}[(i)]
 \item If $\Lambda$ is coamenable in $\Gamma,$ then $R_{\Lambda,X}$ is everywhere coamenable in $R_{\Gamma,X}$.
 \label{item: everywhere co-amenability from subgroups}
    \item If $\Gamma\actson (X,\mu)$ is essentially free, and $\Lambda$ is not coamenable in $\Gamma,$ then for every measurable $E\subseteq X$ with positive measure, we have that $R_{\Lambda,X}|_{E}$ is not coamenable in $R_{\Gamma,X}|_{E}$.
    \label{item: nowhere co-amenability from subgroups}
    
\end{enumerate}
\end{prop}

\begin{proof}
(\ref{item: everywhere co-amenability from subgroups}): To ease notation, set $\cS=\cR_{\Lambda,X}$,$\cR=\cR_{\Gamma,X}$.
Let $F_{n}\subseteq \Gamma/\Lambda$ be a F\o lner sequence. Note that for almost every $x\in X$, we have that $gx\in c$ if and only if $(gh)^{-1}x\in c$ for all $g\in \Gamma,h\in \Lambda$ and for all $c\in [x]_{\cR}/\cS$. Thus for almost every $x\in X$ the set $(\Gamma/\Lambda)_{x,c}=\{g\Lambda\in \Gamma/\Lambda:g^{-1}x\in c\}$ is well defined. Moreover, for almost every $x\in X$ we have that 
\[\Gamma/\Lambda=\bigsqcup_{c\in [x]_{\cR}/\cS}(\Gamma/\Lambda)_{x,c}.\]
Furthermore:
\begin{equation}\label{eqn: lowkey annoying}
  (\Gamma/\Lambda)_{ax,c}=a[(\Gamma/\Lambda)_{x,c}], \textnormal{ for almost every $x\in X$ and all $a\in \Gamma,c\in [x]_{\cR}/\cS$.}  
\end{equation}

We define $\xi^{(n)}\in L^{2}(\cR/\cS)$ by $\xi^{(n)}(x,c)=\sqrt{\frac{|F_{n}\cap (\Gamma/\Lambda)_{x,c}|}{|F_{n}|}}$. 
Then $\|\xi^{(n)}_{x}\|_{2}=1$ for almost every $x\in X$. From (\ref{eqn: lowkey annoying}), it follows that for almost every $x\in X$ and all $a\in \Gamma$:
\begin{align*}
  \|\xi^{(n)}_{ax}-\xi^{(n)}_{x}\|_{2}^{2}&=\frac{1}{|F_{n}|}\sum_{c\in [x]_{\cR/\cS}}\left|\sqrt{|a^{-1}F_{n}\cap ((\Gamma/\Lambda)_{x,c})|}-\sqrt{|F_{n}\cap (\Gamma/\Lambda)_{x,c}|}\right|^{2}\\
  &\leq \frac{1}{|F_{n}|}\sum_{c\in [x]_{\cR/\cS}}\left||a^{-1}F_{n}\cap ((\Gamma/\Lambda)_{x,c})|-|F_{n}\cap (\Gamma/\Lambda)_{x,c}|\right|,  
\end{align*}
where in the last line we use the inequality $|\sqrt{t}-\sqrt{s}|^{2}\leq |t-s|$ valid for nonnegative real numbers $s,t$. Thus 
\begin{align*}
    \|\xi^{(n)}_{ax}-\xi^{(n)}_{x}\|_{2}^{2}&\leq \frac{1}{|F_{n}|}\sum_{c\in [x]_{\cR/\cS}}\left|\left|\sum_{p\in \Gamma/\Lambda}1_{a^{-1}F_{n}}(p)1_{(\Gamma/\Lambda)_{x,c}}(p)\right|-\left|\sum_{p\in \Gamma/\Lambda}1_{F_{n}}(p)1_{(\Gamma/\Lambda)_{x,c}}(p)\right|\right|\\
    &\leq \frac{1}{|F_{n}|}\sum_{c\in [x]_{\cR/\cS}}\sum_{p\in \Gamma/\Lambda}1_{(\Gamma/\Lambda)_{x,c}}(p)|1_{a^{-1}F_{n}}(p)-1_{F_{n}}(p)|\\
    &=\frac{1}{|F_{n}|}|a^{-1}F_{n}\Delta F_{n}|.
\end{align*} 
Thus
\[\|\xi^{(n)}_{ax}-\xi^{(n)}_{x}\|_{2}\to_{n\to\infty}0 \textnormal{ for all $a\in \Gamma.$}\]
Since $\|\xi^{(n)}_{x}\|_{2}=1$, Theorem \ref{thm: DCT argument} forces  $\|\lambda_{\cR/\cS}(\widetilde{\gamma})\xi^{(n)}-\xi^{(n)}\|_{2}\to_{n\to\infty}0$ for all $\widetilde{\gamma}\in [\cR]$.

By Theorem \ref{thm: everywhere rel man TFAE thm}, it suffices to show that if $E\subseteq X$ is a $\Lambda$-invariant measurable set with $\mu(E)>0$, $\lambda_{\cR|_{E}/\cS|_{E}}$ has almost invariant vectors.  For $x\in X$, let \[\iota_{x}\colon [x]_{\cR|_{E}}/\cS|_{E}\to [x]_{\cR}/\cS\]
be the injection defined at the beginning of Section \ref{sec: pointwise cospectral radius}. Let $\widetilde{\zeta}^{(n)}\in L^{2}(\cR|_{E}/\cS|_{E})$ be defined by $\widetilde{\zeta}^{(n)}(x,c)=\xi^{(n)}(x,\iota_{x}(c))$. For any $\widetilde{\gamma}\in [\cR|_{E}]$, we may find a $\gamma\in [\cR]$ with $\gamma|_{E}=\widetilde{\gamma}$. Thus
\[\|\lambda_{\cR|_{E}/\cS|_{E}}(\widetilde{\gamma})\widetilde{\zeta}^{(n)}-\widetilde{\zeta}^{(n)}\|_{2}\leq \|\lambda_{\cR/\cS}(\gamma)\xi^{(n)}-\xi^{(n)}\|_{2}\to_{n\to\infty}0.\]
By Theorem \ref{thm: big TFAE thm}, it thus suffices to show that 
\[\limsup_{n\to\infty}\|\widetilde{\zeta}^{(n)}\|_{2}>0.\]
Note that 
\begin{align*}
 \|\widetilde{\zeta}^{(n)}\|_{2}^{2}&=\frac{1}{|F_{n}|\mu(E)}\int_{E}\sum_{c\in [x]_{\cR|_{E}/\cS|_{E}}}|F_{n}\cap (\Gamma/\Lambda)_{x,\iota_{x}(c)}|\,d\mu(x)\\
 &=\frac{1}{|F_{n}|\mu(E)}\sum_{g\Lambda\in F_{n}}\int_{E}\sum_{c\in [x]_{\cR|_{E}/\cS|_{E}}}1_{(\Gamma/\Lambda)_{x,\iota_{x}(c)}}(g\Lambda)\,d\mu(x).   
\end{align*}
For almost every $x\in X$ and for every $g\Lambda\in \Gamma/\Lambda$ we have that $g\Lambda \in (\Gamma/\Lambda)_{x,\iota_{x}(c)}$ for some $c\in [x]_{\cR|_{E}/\cS|_{E}}$ if and only if $g^{-1}x$ is $\cS$-equivalent to some element of $E$. By $\cS$-invariance of $E$, it follows that $\sum_{c\in [x]_{\cR|_{E}/\cS|_{E}}}1_{(\Gamma/\Lambda)_{x,\iota_{x}(c)}}(g\Lambda)=1_{gE}(x)$ for almost every $x\in X$. Thus
\[ \|\widetilde{\zeta}^{(n)}\|_{2}^{2}=\frac{1}{|F_{n}|\mu(E)}\sum_{g\Lambda\in F_{n}}\mu(gE\cap E)=\frac{1}{\mu(E)}\ip{\frac{1}{|F_{n}|}\sum_{g\Lambda\in F_{n}}1_{gE},1_{E}}.\]
For $f\in L^{\infty}(X)$, let $\E_{\Fix_{\Gamma}(L^{\infty}(X))}(f)$ denote the conditional expectation of $f$ onto the sub-$\sigma$-algebra of $\Gamma$-invariant sets. Then, by Proposition \ref{prop: weak mean ergodic theorem} 
\begin{align*}
\lim_{n\to\infty}\|\widetilde{\zeta}^{(n)}\|_{2}^{2}=\frac{1}{\mu(E)}\ip{\E_{\Fix_{\Gamma}(L^{\infty}(X))}(1_{E}),1_{E}}&=\frac{1}{\mu(E)}\|\E_{\Fix_{\Gamma}(L^{\infty}(X))}(1_{E})\|_{2}^{2}\\
&\geq \frac{1}{\mu(E)}\|\E_{\Fix_{\Gamma}(L^{\infty}(X))}(1_{E})\|_{1}^{2}\\
&=\mu(E)\\
&>0.
\end{align*}
Setting $\zeta^{(n)}=\frac{\widetilde{\zeta}^{(n)}}{\|\widetilde{\zeta}^{(n)}\|_{2}}$, we thus have that $\|\zeta^{(n)}\|_{2}=1$ and $\|\lambda_{\cR|_{E}/\cS|_{E}}(\widetilde{\gamma})\zeta^{(n)}-\zeta^{(n)}\|_{2}\to_{n\to\infty}0$ for all $\widetilde{\gamma}\in [\cR|_{E}]$. So by Theorem \ref{thm: big TFAE thm}, we have that $\cS$ is everywhere coamenable in $\cR$.

(\ref{item: nowhere co-amenability from subgroups}):
Since $\Lambda$ is not coamenable in $\Gamma$, we have that $\Gamma\actson \ell^{2}(\Gamma/\Lambda)$ does not have almost invariant vectors. So we may find a symmetric, finite $F\subseteq \Gamma$ so that 
\[c:=\frac{1}{|F|}\left\|\sum_{g\in F}\lambda_{\Gamma/\Lambda}(g)\right\|_{B(\ell^{2}(\Gamma/\Lambda))}<1.\]
By, e.g. \cite[Lemma 10.1]{Woess},
\[c=\lim_{k\to\infty}\ip{A_{F}^{2k}\delta_{\Lambda},\delta_{\Lambda}}^{1/2k},\]
where $A_{F}=\frac{1}{|F|}\sum_{g\in F}\lambda_{\Gamma/\Lambda}(g).$
Let $u_{F}$ be the uniform measure on $F$. Since the $\Gamma$ actions is free, for almost every $x\in X$ we have $\rho^{\cR_{\Lambda,X}}_{u_{F}}=c<1$. Hence  Proposition \ref{prop: co spectral radius descent} with $\cS_{2}=\cR$ and $\cS_{1}=\cS$ implies that for every $\cS$-invariant set $E$ with $\mu(E)>0$ we have that $\cS|_{E}$ is not coamenable $\cR|_{E}$. The desired result now follows from Proposition \ref{prop: passing rel amen between subsets}.

\end{proof}

In the setup of the above Theorem, we have shown that $\cR_{\Lambda,X}$ is coamenable in $\cR_{\Gamma,X}$ if and only if the von Neumann algebra of $\cR_{\Lambda,X}$ is coamenable in $\cR_{\Gamma,X}$ (see Theorem \ref{thm: two different version of rel amen are the same}). Using this result one can apply \cite[Proposition 6]{MonodPopa} and \cite[Remark 1.5.6]{anantharaman-popa} to obtain a slightly less general version of Proposition \ref{prop: group case is an iff} for the case when the $\Gamma$ action is essentially free: that $\cR_{\Lambda,X}$ is coamenable in $\cR_{\Gamma,X}$ if and only if $\Lambda$ is coamenable in $\Gamma$.

We also show that coamenability can be deduced by checking coamenability of restricted equivalence relations among sets which partition the base space.

\begin{prop}
Let $\cS\leq \cR$ be discrete, probability measure-preserving equivalence relations over a standard probability space $(X,\mu)$. Suppose that $J$ is a countable index set, and that $(E_{j})_{j\in J}$ are positive measure sets with $\mu\left(X\setminus \bigcup_{j}E_{j}\right)=0$. 
\begin{enumerate}[(i)]
    \item If $\cS|_{E_{j}}$ is coamenable in $\cR|_{E_{j}}$ for all $j\in J$, then $\cS$ is coamenable in $\cR$. \label{item: piecing together relative amenability}
    \item If $\cS|_{E_{j}}$ is everywhere coamenable in $\cR|_{E_{j}}$, for all $j\in J$, then $\cS$ is everywhere coamenable in $\cR$.  \label{item: piecing together everywhere relative amenability}
\end{enumerate}

\end{prop}

\begin{proof}

(\ref{item: piecing together relative amenability}):
Let $\nu\in \Prob([\cR])$ be countably supported, symmetric and generate $\cR$. 
By Proposition \ref{prop: co spectral radius descent} with $\cS_{2}=\cR$ and $\cS_{1}=\cS$ and (\ref{eqn: co spectral radius is norm restricted to subset background}) we have that $\|\rho^{\cS}_{\nu}1_{E_{j}}\|_{\infty}=1$ for all $j\in J$. By countability of $J$, and the fact that $\mu\left(X\setminus \bigcup_{j}E_{j}\right)=0$ we deduce that
\[\|\rho^{\cS}_{\nu}\|_{\infty}=\sup_{j\in J}\|\rho^{\cS}_{\nu}1_{E_{j}}\|_{\infty}=1.\]
Thus Theorem \ref{thm: big TFAE thm}, it follows that $\cS$ is coamenable in $\cR$. 

(\ref{item: piecing together everywhere relative amenability}): By Proposition \ref{prop: passing rel amen between subsets}, we may assume that $E_{j}$ is $\cS$-invariant for all $j\in J$.
Let $\nu\in \Prob([\cR])$ be countably supported, symmetric and generate $\cR$. Fix $j\in J$, and a $\cS$-invariant measurable subset   $F$ of $E_{j}$ with positive measure. Apply Proposition \ref{prop: co spectral radius descent} with $\cS_{1}=\cS$ and $\cS_{2}=\cR$, and (\ref{eqn: co spectral radius is norm restricted to subset background}) to see that $\|\rho^{\cS}_{\nu}1_{F}\|_{\infty}=1$. Since $\rho^{\cS}_{\nu}$ is almost surely $\cS$-invariant and  $F$ was any positive measure, measurable, $\cS$-invariant set, it follows that $\rho^{\cS}_{\nu}|_{E_{j}}=1$ almost surely. By countability of $J$, and the fact that $\mu\left(X\setminus \bigcup_{j}E_{j}\right)=0$ we deduce that $\rho^{\cS}_{\nu}=1$ almost everywhere. Hence, by Theorem \ref{thm: everywhere rel man TFAE thm}, it follows that $\cS$ is everywhere coamenable in $\cR$. 

\end{proof}

\color{black}

Note that Example \ref{example: rel amen comes from small pieces} in Section \ref{sec: pointwise cospectral radius} is precisely the type  previously given in \cite[Example 6]{AFH}, establishing that the cospectral radius is not almost surely constant (even if $\cR$ is ergodic). Note that in \cite[Section 3.3]{AFH}, a general condition on $\cS$ was given under which the cospectral radius is almost surely constant if $\cR$ is ergodic. We recall this general condition here.

\begin{defn}
Let $\cS\leq \cR$ be discrete, probability measure-preserving equivalence relations on a standard probability space $(X,\mu)$. The \emph{one-sided partial normalizer} of $\cS\leq \cR$, denote $P^{(1)}N_{\cR}(\cS)$, is the set of all $\gamma\in [[\cR]]$ so that $\gamma([x]_{\cS}\cap \dom(\gamma))\subseteq [\gamma(x)]_{\cS}$ for all $x\in \dom(\gamma)$.
\end{defn}

In the setting of the above definition, if $E\subseteq X$ is measurable, and $\Phi\subseteq [[\cR]]$, the \emph{saturation of $E$ by $\Phi$} is any set $\widetilde{E}$ containing $E$ which is $\Phi$-invariant and has minimum measure among all $\Phi$-invariant measurable subsets which contain $E$. Note that $\widetilde{E}$ is unique modulo null sets.

\begin{prop}
Let $\cS\leq \cR$ be discrete, probability measure-preserving equivalence relations on a standard probability space $(X,\mu)$ and let $E\subseteq X$ be measurable with positive measure. Let $\widetilde{E}$ be the saturation of $E$ by $P^{(1)}N_{\cR}(\cS)$. Suppose that $F\subseteq X$ is measurable, and that $\mu(E\setminus F)=0=\mu(F\setminus\widetilde{E})=0$. 
\begin{enumerate}[(i)]
    \item If $\cS|_{E}$ is coamenable in $\cR|_{E}$, then $\cS|_{F}$ is coamenable in $\cR|_{F}$. \label{item: propogating relative amenability across partial normalizers}
    \item If $\cS|_{E}$ is everywhere coamenable in $\cR|_{E}$, then $\cS|_{\widetilde{E}}$ is everywhere coamenable in $\cR|_{\widetilde{E}}$. \label{item: propogating everywhere relative amenability across partial normalizers}
\end{enumerate}

\end{prop}

\begin{proof}
(\ref{item: propogating relative amenability across partial normalizers}):
Since $\mu(F\setminus\widetilde{E})=0$ we have that the saturation of $E$ by $P^{(1)}N_{\cR|_{F}}(\cS|_{F})$ is $F$, up to sets of measure zero. So we may, and will, assume that $F=X$. Let $\nu\in \Prob([\cR])$ be symmetric, countably supported, and generate $\cR$. By Proposition \ref{prop: co spectral radius descent} with $\cS_{2}=\cR$ and $\cS_{1}=\cS$, and (\ref{eqn: co spectral radius is norm restricted to subset background}) we have that $\|\rho^{\cS}_{\nu}1_{E}\|_{\infty}=1$. From \cite[Proposition 3.15]{AFH} we know that $\rho^{\cS}_{\nu}$ is invariant under $P^{(1)}N_{\cR}(\cS)$, and thus $\|\rho^{\cS}_{\nu}\|_{\infty}=1$ since the saturation of $E$ by $P^{(1)}N_{\cR}(\cS)$ is $X$. Theorem \ref{thm: big TFAE thm} thus implies that $\cS$ is coamenable in $\cR$. 

(\ref{item: propogating everywhere relative amenability across partial normalizers}):  We may, and will, assume that $\widetilde{E}=X$. Let $\nu\in \Prob([\cR])$ be symmetric, countably supported, and generate $\cR$. By Proposition \ref{prop: passing rel amen between subsets}, we have that $\cS|_{\cS E}$ is everywhere coamenable in $\cR|_{\cS E}$. 
By Corollary \ref{cor: equality of co-spectral radii} with $\cS_{2}=\cR$, $\cS_{1}=\cS$ we have that $\rho^{\cS}_{\nu}=1$ almost everywhere on $\cS E$. From \cite[Proposition 3.15]{AFH} we know that $\rho^{\cS}_{\nu}$ is invariant under $P^{(1)}N_{\cR}(\cS)$, and thus $\rho^{\cS}_{\nu}=1$ almost everywhere, since the  saturation of $E$ by $P^{(1)}N_{\cR}(\cS)$ is $X$. So the result follows Theorem \ref{thm: everywhere rel man TFAE thm}.

\end{proof}

\color{black}

\section{Ergodic decomposition reduction}\label{sec: ergodic decomp}

We close the paper by showing that the general situation of coamenability and everywhere coamenability can be reduced to the case where $\cR$ is ergodic. This can once again be deduced from Theorem \ref{thm: two different version of rel amen are the same} and general von Neumann algebra considerations via direct integral theory, but we will give a direct proof here from our previously established equivalences.

For this, we recall the ergodic decomposition theorem (see \cite[Theorem 4.2 and 4.4]{VErgDe} for a proof).

\begin{thm}[Ergodic decomposition theorem]
Let $\cR$ be a discrete, probability measure-preserving equivalence relation on a standard probability space $(X,\mu)$. Then there is a standard probability space $(Z,\zeta)$ a Borel, measure-preserving map $\pi\colon X\to Z$, and measures $\mu_{z}\in \Prob(\pi^{-1}(\{z\}))$ such that:
\begin{enumerate}[(i)]
    \item $\pi^{-1}(\{z\})$ is $\cR$-invariant for every $z\in Z$, \label{item: restrict the relation}
    \item $\cR_{z}=\cR\cap (\pi^{-1}(\{z\})\times \pi^{-1}(\{z\}))$ preserves the measure $\mu_{z}$, and $\cR_{z}$ is an ergodic equivalence relation over $(\pi^{-1}(\{z\}),\mu_{z})$, \label{item: ergodic fibers}
    \item for every $f\colon X\to \C$ bounded and Borel we have that $z\mapsto \int f\,d\mu_{z}$ is Borel, and
    \[\int_{X}f\,d\mu=\int_{Z}\left(\int f\,d\mu_{z}\right)\,d\zeta(z).\]
    \label{item: Fubinish}
\end{enumerate}
\end{thm}

For technical reasons, it will be helpful to define a measurable map on $X$ by defining it fiberwise. The following Proposition allows us to guarantee that the resulting map remains measurable.

\begin{prop}\label{prop: technical bs}
Let $\cR$ be a discrete, probability measure-preserving equivalence relation on a standard probability space $(X,\mu)$ and let $(Z,\zeta,(\mu_{z})_{z\in Z},\pi)$ be the ergodic decomposition as stated above. Suppose that $z\mapsto \xi^{z}\in L^{2}(\pi^{-1}(\{z\}),\mu_{z})$ is such that for every Borel $B\subseteq X$ we have that $z\mapsto \int_{B}\xi^{z}\,d\mu_{z}$ is Borel. Then:
\begin{enumerate}[(i)]
    \item the map $z\mapsto \|\xi^{z}\|_{L^{2}(\mu_{z})}$ is Borel, \label{item: norm measurability}
    \item there is a $\xi\in L^{2}(X,\mu)$ so that $\zeta$-almost every $z\in Z$ we have that $\xi|_{\pi^{-1}(\{z\})}=\xi^{z}$ almost everywhere with respect to $\mu_{z}$.  \label{item: collecting together to a single function}
\end{enumerate}
\end{prop}

\begin{proof}

(\ref{item: norm measurability}): Let $(B_{n})_{n=1}^{\infty}$ be a sequence of Borel subsets so that $\{B_{n}:n\in \N\}$ is an algebra of sets which generates the Borel $\sigma$-algebra. Then for every $\nu\in \Prob(X)$ we have that $\Span\{1_{B_{n}}:n\in \N\}$ is dense in $L^{2}(X,\mu)$. For $a\in c_{c}(\N,\Q[i])$ we set $f_{a}=\sum_{n}a(n)1_{B_{n}}$. Item (\ref{item: Fubinish}) in the ergodic decomposition implies that for every $a\in c_{c}(\N,\Q[i])$ we have that $z\mapsto \|f_{a}\|_{L^{2}(\mu_{z})}$ is Borel.
Then:
\[\|\xi^{z}\|=\sup_{a\in c_{c}(N,\Q[i])}1_{(0,\infty)}(\|f_{a}\|_{L^{2}(\mu_{z})})\frac{|\ip{\xi^{z},f_{a}}|}{\|f_{a}\|_{L^{2}(\mu_{z})}}.\]
By assumption, this is a supremum of countably many Borel functions, hence Borel.

(\ref{item: collecting together to a single function}): 
Approximating bounded, Borel functions by simple functions, our hypothesis implies that for every bounded, Borel function $f\colon X\to \C$ we have that $f\mapsto \ip{f,\xi^{z}}_{L^{2}(\mu_{z})}$ is Borel. Partitioning $Z$ into $\{z:n-1\leq \|\xi^{z}\|<n\}$ for $n\in \N$, we may assume that $z\mapsto \|\xi^{z}\|_{L^{2}(\mu_{z})}$ is $L^{2}$.  For $f$ bounded Borel define 
\[\phi(f)=\int_{Z}\ip{f,\xi^{z}}_{L^{2}(\mu_{z})}d\zeta(z).\]
Then
\[|\phi(f)|\leq \|f\|_{L^{2}(\mu)}\left(\int \|\xi^{z}\|_{L^{2}(\mu_{z})}^{2}\right)^{1/2}.\]
Since bounded Borel functions are dense in $L^{2}$, the above estimate implies that $\phi$ extends uniquely to a bounded linear functional $\phi\colon L^{2}(X,\mu)\to \C$. Thus there is a unique $\xi\in L^{2}(X,\mu)$ with $\phi(f)=\ip{f,\xi}_{L^{2}(\mu)}$. Letting $f_{a}$ be as in (\ref{item: norm measurability}), we see that for every $B\subseteq Z$ Borel we have 
\[\int_{B}\ip{f_{a},\xi}_{L^{2}(\mu_{z})}\,d\zeta(z)=\phi(1_{\pi^{-1}(B)}f_{a})=\int_{B}\ip{f_{a},\xi^{z}}_{L^{2}(\mu_{z})}\,d\zeta(z).\]
By countability of $c_{c}(\N,\Q[i])$ we thus see that for $\zeta$-almost every $z\in Z$ we have that $\ip{f_{a},\xi}_{L^{2}(\mu_{z})}=\ip{f_{a},\xi^{z}}_{L^{2}(\mu_{z})}$ for all $a\in c_{c}(\N,\Q[i])$. Since $\{f_{a}:a\in c_{c}(\N,\Q[i])\}$ is dense in $L^{2}(\mu_{z})$ for all $z\in Z$, we conclude that for $\zeta$-almost every $z\in Z$ we have that $\xi|_{\pi^{-1}(\{z\})}=\xi^{z}$ almost everywhere with respect to $\mu_{z}$.  
\end{proof}

As is typical in arguments involving the ergodic decomposition, we will ultimately rely on a measurable selection theorem. Recall that if $X$ is a standard Borel space, then $A\subseteq X$ is \emph{analytic} if it is the image of a standard Borel space under a Borel map. Such sets are measurable with respect to every Borel probability measure on $X$ (see \cite[Theorem 21.10]{KechrisClassic}).

\begin{thm}[Jankov-von Neumann Theorem, see e.g. Theorems 18.1, 21.10 in \cite{KechrisClassic}] \label{thm: measurable sections}
Let $X,Y$ be standard Borel spaces and $P\subseteq X\times Y$ be Borel. Let $\pi_{X}(P)=\{x\in X: \textnormal{ there exists $y\in Y$ with } (x,y)\in P\}$. Then there is a  $s\colon \pi_{X}(P)\to Y$ with $(x,s(x))\in P$ for every $x\in \pi_{X}(P)$ and such that $s$ is $\mu$-measurable for every Borel probability measure $\mu$ on $X$.  
\end{thm}

Note that $\pi_{X}(P)$ in the above is analytic, hence measurable with respect to every probability measure on $X$, thus it makes sense to ask that $s$ in the theorem statement is measurable.

For a standard Borel space $X$ and a $\nu\in \Prob(X)$, we let $\Omega(X,\nu)$ be all $\nu$-measurable sets modulo null sets with the distance $d(E,F)=\nu(E\Delta F)$. Being isometric with a closed subset of $L^{1}(X,\nu)$ we know that $\Omega(X,\nu)$ is complete.

\begin{thm}
Let $\cS\leq \cR$ be an inclusion of probability measure-preserving equivalence relations over a standard probability space $(X,\mu)$. Let $\pi\colon (X,\mu)\to (Z,\zeta)$ be the space of ergodic components, and $\mu_{z}\in \Prob(\pi^{-1}(\{z\}))$ be the ergodic decomposition of $\mu$. Then:
\begin{enumerate}[(i)]
\item $\cS$ is coamenable in $\cR$ if and only if for almost every $z\in Z$ we have that $\cS_{z}:=\cS\cap (\pi^{-1}(\{z\})\times \pi^{-1}(\{z\}))$ is coamenable in $\cR_{z}=\cR\cap (\pi^{-1}(\{z\})\times \pi^{-1}(\{z\}))$. \label{item: ergodic deocmposition rel amen}
\item $\cS$ is everywhere coamenable in $\cR$ if and only if for every almost every $z\in Z$ we have that $\cS_{z}$ is everywhere coamenable in $\cR_{z}$.  \label{item; ergodic decomposition everywhere rel amen}
\end{enumerate}
\end{thm}

\begin{proof}

(\ref{item: ergodic deocmposition rel amen}): For both directions, fix  countable $\Lambda,\Gamma\leq [\cR]$ which generate $\cS,\cR$ respectively. Then by uniqueness of the ergodic decomposition and countability of $\Gamma$, for almost every $z\in Z$ we have that 
\begin{itemize}
    \item the actions $\Lambda,\Gamma\actson X$ leave $\pi^{-1}(\{z\})$ invariant, 
    \item the actions $\Lambda,\Gamma\actson (\pi^{-1}(\{z\}),\mu_{z})$ are  measure-preserving,
    \item $\cR_{z}=\cR_{\Gamma,\pi^{-1}(\{z\})}$, $\cS_{z}=\cR_{\Lambda,\pi^{-1}(\{z\})}$. 
\end{itemize}
Fix a conull, Borel set $Z_{0}\subseteq Z$ such that all the above items hold. We also fix a conull, Borel, $\cR$-invariant set $X_{0}\subseteq X$ with $[x]_{\cR}=\Gamma x$, $[x]_{\cS}=\Lambda x$ for every $x\in X_{0}$.

First, suppose that $\cS$ is coamenable in $\cR$. Let $\xi^{(n)}\in L^{2}(\cR/\cS)$ be such that $\|\xi^{(n)}_{x}\|_{2}=1$ for almost every $x\in X$ and with $\lim_{n\to\infty}\|\xi^{(n)}_{x}-\xi^{(n)}_{y}\|_{2}=0$ for almost every $(x,y)\in \cR$. Note that for every $x\in X_{0}$, we have that $[x]_{\cR}=\Gamma x=[x]_{\cR_{\pi(x)}}$, and $[x]_{\cS}=\Lambda x=[x]_{\cS_{\pi(x)}}$. Thus $[x]_{\cR}/\cS=[x]_{\cR_{\pi(x)}}/\cS_{\pi(x)}$.  Thus, by Fubini, for almost every $z\in Z_{0}$ we have that $\|\xi^{(n)}_{x}\|_{\ell^{2}([x]_{\cR_{z}}/\cS_{z})}=1$ for almost every $x\in \pi^{-1}(\{z\})$. 
By Fubini, for almost every $z\in Z_{0}$ we have that 
\[\lim_{n\to\infty}\|\xi^{(n)}_{x}-\xi^{(n)}_{\gamma(x)}\|_{2}=0\]
for almost every $x\in \cR_{z}$ and every $\gamma\in \Gamma$. Since $\Gamma x=[x]_{\cR}$ for every $x\in X_{0}$, we see that for almost every $z\in Z$ we have that $\cS_{z}$ is coamenable in $\cR_{z}$.

Conversely, suppose that for almost every $z\in Z$ we have that $\cS_{z}$ is coamenable in $\cR_{z}$. Fix a countably supported $\nu\in \Prob([\cR])$ whose support generates $\cR$, and set $\Delta=\ip{\supp(\nu)}$. For each $g\in \Delta,$ choose a Borel $s(g)\in [\cR]$ with $s(g)=g$ almost everywhere, and set $\widetilde{\nu}=\sum_{g\in \Delta}\nu(\{g\})\delta_{s(g)}$. By countability of $\Delta$, for
almost every $z\in Z$ we have that $s(g)(\pi^{-1}(\{z\}))=\pi^{-1}(\{z\})$ for every $g\in \Delta$.
Hence, for almost every $z\in Z$, there is a well-defined map $\phi_{z}\colon \Delta\to [\cR_{z}]$ given by $\phi_{z}(g)=s(g)|_{\pi^{-1}(\{z\})}$, and so for almost every $z\in Z$ we can make sense of $\nu_{z}=(\phi_{z})_{*}\nu\in \Prob([\cR_{z}])$. 

Since $\Delta x=[x]_{\cR}$ for $\mu$-almost every $x\in X$, it follows that for almost every $z\in Z$ we have that $\Delta x=[x]_{\cR_{z}}$ for $\mu_{z}$-almost every $x\in \pi^{-1}(\{z\})$. Also, for almost every $x\in X$ we have $gx=s(g)x$ for all $g\in \Delta$. Hence it follows that for almost every $z\in Z$ we have $\nu_{z}$ is generating. Thus we can choose a conull Borel $Z_{1}\subseteq Z_{0}$ so that
for all $z\in  Z_{1}$:
\begin{itemize}
    \item we have that $\nu_{z}$ is generating,
    \item $\|\rho^{\cS_{z}}_{\nu_{z}}\|_{\infty}=1$,
    \item $\rho^{\cS_{z}}_{\nu_{z}}(x)=\rho^{\cS}_{\widetilde{\nu}}(x)$ for almost every $x\in \pi^{-1}(\{z\})$. 
\end{itemize}
Fix a Borel $F\subseteq X$ which is $\cR$-invariant and with $\mu(F)>0$. Set $F_{0}=\{z\in Z_{1}:\mu_{z}(F)=1\}$. Note that $F_{0}$ is a Borel set with $F=\pi^{-1}(F_{0})$ modulo null sets. 
Note that $ \rho^{\cS}_{\widetilde{\nu}}$ is Borel and almost everywhere equal to $\rho^{\cS}_{\nu}$. Hence, for any $\varepsilon>0$
\begin{align*}
    \mu(\{x\in F:\rho^{\cS}_{\nu}(x)>1-\varepsilon\})&=\mu(\{x\in F:\rho^{\cS}_{\widetilde{\nu}}(x)>1-\varepsilon\})\\
    &=\int_{F_{0}} \mu_{z}(\{x\in \pi^{-1}(\{z\}):\rho^{\cS_{z}}_{\nu_{z}}(x)>1-\varepsilon\})\,d\zeta(z).
\end{align*}
Since $\|\rho^{\cS_{z}}_{\nu_{z}}\|_{\infty}=1$ for all $z\in F_{0}$, the integrand is everywhere positive, and since $\zeta(F_{0})=\mu(F)>0$, we find that $\mu(\{x\in F:\rho^{\cS}_{\nu}(x)>1-\varepsilon\})>0$. Thus $\|\rho^{\cS}_{\nu}1_{F}\|_{\infty}=1$. Since $F,\nu$ were arbitrary, Theorem \ref{thm: big TFAE thm} implies that $\cS$ is coamenable in $\cR$.

(\ref{item; ergodic decomposition everywhere rel amen}): First suppose that for almost every $z\in Z$, we have that $\cS_{z}$ is everywhere coamenable in $\cR_{z}$. 
Let $E\subseteq X$ with $\mu(E)>0$, and choose a Borel $F\subseteq Z$ with $\cR E=\pi^{-1}(F)$ modulo null sets. Then $\mu_{z}(\pi^{-1}(\{z\})\cap E)>0$ for almost every $z\in F$. We may view $(\cR_{z}|_{\pi^{-1}(\{z\})\cap E})_{z\in F}$ and $(\cS_{\pi^{-1}(\{z\})\cap E})_{z\in F}$ as the ergodic decomposition of $\cR|_{E}$ and $\cS|_{E}$. Since $\cS_{z}$ is everywhere coamenable in $\cR_{z}$ for almost every $z\in Z$, we have that $(\cS_{z}|_{\pi^{-1}(\{z\})\cap E})_{z\in F}$ is coamenable in $(\cR_{\pi^{-1}(\{z\})\cap E})_{z\in F}$ for almost every $z\in F$. Thus by (\ref{item: ergodic deocmposition rel amen}), it follows that $\cS|_{E}$ is coamenable in $\cR|_{E}$.

Conversely suppose that  $F=\{z\in Z:S_{z} \textnormal{ is not everywhere coamenable in } \cR_{z}\}$ is not null. We ultimately will use measurable selection, and will proceed in several steps which will include establishing measurability of $F$. 

Let $(B_{n})_{n=1}^{\infty}$ be a sequence of $\cS$-invariant Borel subsets of $X$ so that $\{B_{n}:n\in \N\}$ is an algebra of sets which generates the $\sigma$-algebra of Borel $\cS$-invariant sets. Then for every $\nu\in \Prob(X)$ we have that $\{B_{n}:n\in \N\}$ is dense in $\Omega(X,\nu)$. 
Let $\cL=\{(\sigma,z):\sigma\in \N^{\N}, \lim_{n\to\infty}B_{\sigma(n)} \textnormal{ exists in } \Omega(X,\mu_{z})\}$. Note that $\cL$ is Borel, indeed by completeness of $\Omega(X,\mu_{z})$ 
\[\cL=\bigcap_{k=1}^{\infty}\bigcup_{N=1}^{\infty}\bigcap_{n,m\geq N}\{(\sigma,z):\mu_{z}(B_{\sigma(n)}\Delta B_{\sigma(m)})<1/k\}.\]
Thus (\ref{item: Fubinish}) of the ergodic decomposition implies that $\cL$ is Borel. For $(\sigma,z)\in \cL$, we let $B_{\sigma,z}=\lim_{n\to\infty}B_{\sigma(n)}$ with the limit taken in $\Omega(X,\mu_{z})$. Note that $B_{\sigma,z}$ is only well-defined modulo $\mu_{z}$-null sets. We make the following 

\emph{Claim:}
\[\cN=\{(\sigma,z):\cS_{z}|_{B_{\sigma,z}} \textnormal{ is coamenable in } \cR_{z}|_{B_{\sigma,z}}\}\]
\emph{is Borel.} 

Once we show this claim, we have that $F$ is the projection of $\cN^{c}$ onto the second factor. Thus $F$ is analytic and hence $\zeta$-measurable by \cite[Theorem 21.10]{KechrisClassic}.

To prove the claim, let $\Gamma\leq [\cR]$ with $\cR_{\Gamma,X}=\cR$, and with all elements of $\Gamma$ Borel. 
For $n\in\N$,$g\in \Gamma$, let $f_{n,g}\colon \cR/\cS\to \C$ be given by $f_{n,g}(x,c)=1_{B_{n}}(x)\delta_{[gx]_{\cS}=c}$. For $a\in c_{c}(\N\times \Gamma,\Q[i])$, let $f_{a}=\sum_{n,g}a(n,g)f_{n,g}$. By arguments similar to \cite[Lemma 3.5]{AFH}, we have $\{f_{a}:a\in c_{c}(\N\times \Gamma,\Q[i])\}$ is dense in $L^{2}(\cR_{z}/\cS_{z},\mu_{R_{z}/\cS_{z}})$ for every $z\in Z$. For $B\subseteq X$ which is $\cS$-invariant and measurable, we let $f_{a}^{B}=R(1_{B})f_{a}$. Since the map $L^{2}(\cR_{z}/\cS_{z})\to R(1_{B\cap \pi^{-1}(\{z\})})L^{2}(\cR_{z}/\cS_{z})$ given by $k\mapsto R(1_{B\cap \pi^{-1}(\{z\})})k$ is continuous and surjective, we conclude that $\{f_{a}^{B}:a\in c_{c}(\N,\Q[i])\}$ is dense in $R(1_{B\cap \pi^{-1}(\{z\})})L^{2}(\cR_{z}/\cS_{z})$. 
Set
\[\widetilde{\cN}=\bigcap_{\substack{k\in \N,\\ F\subseteq \Gamma \textnormal{ finite}}}\bigcup_{a\in c_{c}(\N\times \Gamma,\Q[i])}\bigcap_{g\in F}\Upsilon_{a,g,k},\]
where
\[\Upsilon_{a,g,k}=\{(\sigma,z):\|\lambda_{\cR_{z}/\cS_{z}}(g)f_{a}^{B_{\sigma,z}}-f_{a}^{B_{\sigma,z}}\|_{L^{2}(\cR_{z}/\cS_{z})}<1/k,\|f_{a}^{B_{\sigma,z}}\|_{L^{2}(\cR_{z}/\cS_{z})}<1+1/k\}\cap\] \[\left\{(\sigma,z): \int \|(f^{B_{\sigma,z}}_{a})_{x}\|_{2}\,d\mu_{z}(x)\geq 1-1/k 
\right\}.\]
Note that $(\sigma,z)\in \Upsilon_{a,g,k}$ implies that 
\[\int \left|\|(f^{B_{\sigma,z}}_{a})_{x}\|_{2}-1\right|^{2}\,d\mu_{z}(x)\leq \frac{5}{k}. \]
Because $\{f_{a}^{B}:a\in c_{c}(\N,\Q[i])\}$ is dense in $R(1_{B\cap \pi^{-1}(\{z\})})L^{2}(\cR_{z}/\cS_{z})$, and since for fixed $z$ the maps $\sigma\mapsto B_{\sigma,z}$ defined for $\sigma$ with $(\sigma,z)\in \cL$ is surjective, Theorem \ref{thm: big TFAE thm} and a perturbation argument implies that $\widetilde{\cN}=\cN$. Thus to show that $\cN$ is Borel, it suffices to show that 
\begin{itemize}
    \item $(\sigma,z)\mapsto \|\lambda_{\cR_{z}/\cS_{z}}(g)f_{a}^{B_{\sigma,z}}-f_{a}^{B_{\sigma,z}}\|_{L^{2}(\cR_{z}/\cS_{z})}$,
    \item $(\sigma,z)\mapsto \|f_{a}^{B_{\sigma,z}}\|_{L^{2}(\cR_{z}/\cS_{z})}$,
    \item $(\sigma,z)\mapsto  \int \|(f_{a}^{B_{\sigma,z}})_{x}\|_{2}\,d\mu_{z}(x)$,
\end{itemize}
are all Borel functions of $(\sigma,z)\in \cL$. 

To do this, fix an enumeration $(g_{m})_{m=1}^{\infty}$ of $\Gamma$ with $g_{1}=e$, and let $E_{m}$ be as in the proof of Proposition \ref{prop: standard}. Then:
\[ \|\lambda_{\cR_{z}/\cS_{z}}(g)f_{a}^{B_{\sigma,z}}-f_{a}^{B_{\sigma,z}}\|_{L^{2}(\cR_{z}/\cS_{z})}^{2}=\sum_{m}\int_{E_{m}\cap g_{m}^{-1}(B_{\sigma,z})}|f_{a}(g^{-1}x,[g_{m}x]_{\cS})-f_{a}(x,[g_{m}x])|^{2}\,d\mu_{z}(x).\]
Since $|f_{a}(g^{-1}x,[g_{m}x]_{\cS})-f_{a}(x,[g_{m}x])|^{2}$ is a bounded, Borel function of $x$ we have that 
\[\int_{E_{m}\cap g_{m}^{-1}(B_{\sigma,z})}|f_{a}(g^{-1}x,[g_{m}x]_{\cS})-f_{a}(x,[g_{m}x])|^{2}\,d\mu_{z}(x)=\]\[\lim_{n\to\infty}\int_{E_{m}\cap g_{m}^{-1}(B_{\sigma(n)})}|f_{a}(g^{-1}x,[g_{m}x]_{\cS})-f_{a}(x,[g_{m}x])|^{2}\,d\mu_{z},\]
and property (\ref{item: Fubinish}) of the ergodic decomposition implies that this is a Borel function of $(\sigma,z)\in \cL$. The proof that  $(\sigma,z)\in \cL\mapsto \|f_{a}^{B_{\sigma,z}}\|_{L^{2}(\cR_{z}/\cS_{z})}$ is Borel is the same.

The map $(\sigma,z)\in \cL\mapsto  \int \|(f_{a}^{B_{\sigma,z}})_{x}\|_{2}\,d\mu_{z}(x)$ is given by
\[\lim_{M\to\infty}\int \left(\sum_{m=1}^{M}1_{E_{m}}(x)1_{B_{\sigma,z}}(g_{m}x)|f_{a}(x,[g_{m}x]_{\cS})|^{2}\right)^{1/2}\,d\mu_{z}(x).\]
Since $f_{a}$ is a bounded Borel function, we see that 
\[\lim_{M\to\infty}\int \left(\sum_{m=1}^{M}1_{E_{m}}(x)1_{B_{\sigma,z}}(g_{m}x)|f_{a}(x,[g_{m}x])|^{2}\right)^{1/2}\,d\mu_{z}=\]\[\lim_{M\to\infty}\lim_{n\to\infty}\int \left(\sum_{m=1}^{M}1_{E_{m}}(x)1_{B_{\sigma(n)}}(g_{m}x)|f_{a}(x,[g_{m}x]_{\cS})|^{2}\right)^{1/2}\,d\mu_{z},\]
and property (\ref{item: Fubinish}) of the ergodic decomposition implies that this is a Borel function of $(\sigma,z)\in \cL$. 



Having shown the claim, our hypothesis is then that $\zeta(F)>0$. Choose a Borel, conull $F_{0}'\subseteq F$, and let $\cT_{0}=\{(\sigma,z)\in \cL\setminus \cN:z\in F_{0}'\}$. Then the projection map $\cT_{0}\to F_{0}'$ is Borel and surjective, so Theorem \ref{thm: measurable sections} implies that we can find a $\zeta$-measurable map $s\colon F_{0}'\to \N^{\N}$ so that for all $z\in F_{0}'$ we have that $\cS|_{B_{s(z),z}}$ is not coamenable in $\cR|_{B_{s(z),z}}$. Let $s'\colon F_{0}'\to \N^{\N}$ be Borel with $s'=s$ almost everywhere, and let $F_{0}\subseteq F_{0}'$ be Borel and conull with $s|_{F_{0}}=s'|_{F_{0}}$. 
Note that for every Borel set $B\subseteq F_{0}$ we have that 
\[\mu_{z}(B_{s(z),z}\cap \pi^{-1}(B))=\lim_{n\to\infty}\mu_{z}(B_{s(z)(n)}\cap \pi^{-1}(B)).\]
Since $s|_{F_{0}}=s'|_{F_{0}}$ is Borel,
part (\ref{item: Fubinish})  of the ergodic decomposition thus implies that $z\mapsto \mu_{z}(B_{s(z),z}\cap \pi^{-1}(B))$ is Borel. Hence Proposition \ref{prop: technical bs} implies that there is a measurable $C\subseteq X$ so that $\zeta$-almost every $z$ we have that $\mu_{z}((C\cap \pi^{-1}(\{z\}))\Delta B_{s(z),z})=0$. By design  for almost every $z\in F$, we have that $\cS_{z}|_{C\cap \pi^{-1}(\{z\})}$ is not coamenable in  $\cR_{z}|_{C\cap \pi^{-1}(\{z\})}$ hence part (\ref{item: ergodic deocmposition rel amen}) implies that $\cS|_{C}$ is not coamenable in $\cR|_{C}$.

\end{proof}



%

\end{document}